\documentclass[12pt,reqno]{amsart}
\usepackage{hyperref}
\usepackage{amscd,amssymb,amsfonts,amsbsy,amsmath,verbatim,color, mathrsfs}
\usepackage{pst-node}
\usepackage{tikz-cd}
\usepackage{graphicx,cite}
\usepackage{subfigure}
\usepackage{float}
\usepackage{tikz,ifthen,bm}
\usepackage{verbatim,mathdots,caption,epstopdf}
\usetikzlibrary{intersections}
\usetikzlibrary{calc}
\usepackage{amsmath}

\newtheorem{thm}{Theorem}[section]

\newtheorem{lem}[thm]{Lemma}

\newtheorem{prop}[thm]{Proposition}
\newtheorem{exam}[thm]{Example}

\theoremstyle{definition}
\newtheorem{defi}{Definition}[section]
\newtheorem{rema}[thm]{Remark}
\numberwithin{equation}{section}
\newcommand{\Real}{\mathbb{R}}

\newcommand{\R}{\Real}
\newcommand{\Z}{{\mathbb Z}}

\newcommand{\N}{{\mathbb N}}

\makeatletter

\newcommand{\Rmnum}[1]{\expandafter\@slowromancap\romannumeral #1@}
\makeatother
\begin{document}
	\baselineskip=17pt
	\setcounter{figure}{0}
	
	\title[Iterated relation systems]
	{Iterated relation systems on Riemannian manifolds}
	\date{\today}
	\author[J. Liu]{Jie Liu}
	\address{Key Laboratory of High Performance Computing and Stochastic Information
		Processing (HPCSIP) (Ministry of Education of China, College of Mathematics and Statistics, Hunan Normal University, Changsha, Hunan 410081, China.}
	\email{1019198461@qq.com}

	\author[S.-M. Ngai]{Sze-Man Ngai}
	\address{Beijing Institute of Mathematical Sciences and Applications, Huairou District, 101400, Beijing,  China,  and Key Laboratory of High Performance Computing and Stochastic
		Information Processing (HPCSIP) (Ministry of Education of China), College of
		Mathematics and Statistics, Hunan Normal University, Changsha, Hunan
		410081,  China.}
	\email{ngai@bimsa.cn}
	
	\author[L. Ouyang]{Lei Ouyang}
	\address{Key Laboratory of High Performance Computing and Stochastic Information
		Processing (HPCSIP) (Ministry of Education of China, College of Mathematics and Statistics, Hunan Normal University, Changsha, Hunan 410081, China.}
	\email{ouyanglei@hunnu.edu.cn}

	\subjclass[2010]{Primary: 28A80; Secondary: 28A78.}
	\keywords{Riemannian manifold; iterated relation system; graph iterated function system; graph open set condition; graph finite type condition; Hausdorff dimension.}

	\thanks{The authors are supported in part by the National Natural Science Foundation of China, grant 12271156, and Construct Program of the Key Discipline in Hunan Province.}

	\begin{abstract}
		For fractals on Riemannian manifolds, the theory of iterated function systems often does not apply well directly, as fractal sets are often defined by relations that are multivalued or non-contractive. To overcome this difficulty,  we introduce the notion of iterated relation systems. We study the attractor of an iterated relation system and formulate a condition under which such an attractor can be identified with that of an associated graph-directed iterated function system. Using this method, we obtain dimension formulas for the attractor of an iterated relation system under the graph open set condition or the graph finite type condition. This method improves the one in [Ngai-Xu, J. Geom. Anal. {\bf 33} (2023), 262], which relies on knowing the specific structure of the attractor.
		\tableofcontents
	\end{abstract}
	
	\maketitle
	\section{Introduction}\label{S:IN}
	\setcounter{equation}{0}
	Fractals in Riemannian manifolds, especially in Lie groups, Heisenberg groups, and projective spaces, have been studied extensively by many authors, (see  Strichartz \cite{Strichartz_1992}, Balogh and Tyson \cite{Balogh-Tyson_2005}, Barnsley and Vince \cite{Barnsley-Vince_2012}, Hossain {\em et al.} \cite{Hossain-Akhtar-Navascues_2024}, etc.). A basic theory of iterated function systems (IFS) on Riemannian manifolds, including  $L^q$-spectrum, Hausdorff dimension, the weak separation condition (WSC), and the finite type condition (FTC), etc., has been established by  the second author and Xu (\cite{Ngai-Xu_2021,Ngai-Xu_2023}).
	In \cite{Ngai-Xu_2023}, a method for computing the Hausdorff dimension of a certain fractal set  in a flat torus is obtained; the fractal set can be identified with the attractor of some IFS on $\mathbb{R}^2/\mathbb{Z}^2$.
	However,  this method is too cumbersome, and is not easy to generalize. The purpose of this paper is to establish a new method that can be used to calculate the Hausdorff dimensions of general fractal sets in Riemannian manifolds. Unlike fractal sets in $\R^n$, many interesting fractals in Riemannian manifolds cannot be described by an IFS. For example,  the relations generating a Sierpinski gasket on the cylindrical surface are multivalued (see Example \ref{exam:osctri}). To deal with this problem, we introduce the notion of iterated relation systems (IRS) (see Definition \ref{defi:2.1}). We will show that many interesting fractals can be described naturally by IRSs.

	\begin{comment}
	{\color{magenta}In this paper, we first study the properties of the attractor of an IRS. Let $x\in \bigcap_{n=0}^\infty\bigcup_{t=1}^NR_t(E_n)$ such that for any $n\geq 0$, there exists $l_n\in \Sigma$ satisfying $x\in R_{l_n}(E_n)$. If for any $k\geq0$, there exists a sequence $\{y_{n_k}\}\in E_{n_k}\subseteq E_{n}$ such that $x=R_{l_n}(y_{n_k})$, then there exists a subsequence $\{y_{n_{k_j}}\}$ of $\{y_{n_k}\}$ converging to some $y\in\bigcap_{n=0}^\infty E_n$  satisfying $x=R_{l_n}(y)$.}
	\end{comment}
	
	Let $\{R_t\}_{t=1}^N$ be an IRS on a nonempty compact subset of a topological space, and let $K$ be the attractor of $\{R_t\}_{t=1}^N$ (see Definition \ref{defi:2.2}). Then
		\begin{align}\label{eq:K=RK}
			K=\bigcup_{t=1}^NR_t(K)
			\end{align}
		(see Proposition \ref{prop:2.2}).
		 However, in general, a nonempty compact set $K$ satisfying
	\eqref{eq:K=RK} need not be unique, and therefore need not be the attractor of $\{R_t\}_{t=1}^N$;  Example \ref{exa:notunique} illustrates this. To study the uniqueness of $K$, we introduce the notion of a graph-directed iterated function system (GIFS) associated to an IRS (see Definition \ref{defi:3.1}), and prove the uniqueness of $K$ under the assumption that such a GIFS exists.
%i.e., $K$ in  \eqref{eq:K=RK} is not a sufficient condition for the attractor of $\{R_t\}_{t=1}^N$.

	Since relations need not be single-valued or contractive, it might not be possible to construct a GIFS associated to an IRS (see Examples \ref{thm(a)}--\ref{notcon}). 	Our first goal is to study, under suitable conditions, the existence of a GIFS associated to an IRS (see Theorem \ref{thm:exi}).
	Let $\# F$ be the cardinality of a set $F$. The definition of a finite family of contractions decomposed from $\{R_t\}_{t=1}^N$ and the definition of a good partition that appear in our first main theorem below are given in Definition \ref{defi:good}.
	\begin{thm}\label{thm:exi}
		Let $X$ be a complete metric space and let $E_0\subseteq X $ be a nonempty compact set. Let $\{R_t\}_{t=1}^N$ ($N\geq 1$) be an IRS on $E_0$. For any $k\geq0$, let $E_{k+1}:=\bigcup_{t=1}^NR_t(E_k)$. Assume that $\{R_t\}_{t=1}^N$ satisfies the following conditions.
		\begin{enumerate}
			\item[(a)]There exists an integer $k_0\geq 1$ such that for any $t\in\{1,\ldots,N\}$ and any $x\in E_{k_0-1}$,
			$$\#\{R_t(x)\}<\infty.$$
			\item[(b)] For any $t\in\{1,\ldots,N\}$, if $H_t:=\{x\in E_{k_0-1}|2\leq\#\{R_t(x)\}<\infty\}\neq \emptyset$,  we require that the following conditions are satisfied.
			\begin{enumerate}
				\item[(i)] $r_t:=R_t|_{H_t}$  can be decomposed as a finite family of contractions  $\{h _t^{l,i}:H_t^i\to h _t^{l,i}(H_t^i)\}_{l=1,i=1}^{n_t,m_t}$,  where $\bigcup_{i=1}^{m_t}H_t^i=H_t$, and $\tilde{r}_t:=R_t|_{E_{k_0-1}\backslash H_t}$  can be decomposed as a finite family of contractions  $\{h _t^{0,i}:J_t^i\to h _t^{0,i}(J_t^i)\}_{i=1}^{s_t}$,  where $\bigcup_{i=1}^{s_t}J_t^i=E_{k_0-1}\backslash H_t$.
				\item[(ii)]  There exists a good partition on $E_{k_0}$ with respect to some finite family of contractions decomposed from $\{R_t\}_{t=1}^N$ (not necessarily the one in $(i)$).
			\end{enumerate}
		\end{enumerate}
		Then there exists a GIFS associated  to $\{R_t\}_{t=1}^N$ on $E_{k_0}$.
	\end{thm}
	We give examples (see Examples \ref{thm(a)}--\ref{notgood}) to investigate the conditions in Theorem \ref{thm:exi}. Under the assumption that $\{R_t\}_{t=1}^N$ satisfies conditions (a) and (b)(i) in Theorem \ref{thm:exi},  we give a sufficient condition for the hypothesis (b)(ii) to be satisfied (see Theorem \ref{thm:exi2}).

	In Section \ref{S:GOSC}, we study IRS attractors under the assumption that an associated GIFS satisfies (GOSC).   Our second goal is to describe a method for computing the Hausdorff dimension of an IRS attractor in a Riemannian manifold under the assumption that an associated GIFS $G$ exists and  $G$ satisfies (GOSC).
The construction of the incidence matrices  that appear in Theorems \ref{thm:main1}--\ref{thm:main2} will be given in Section \ref{S:GOSC} (see \eqref{eq:matrix}).
	\begin{thm}\label{thm:main1}
		Let $M$ be a complete n-dimensional smooth orientable Riemannian manifold that is locally Euclidean. Let $\{R_t\}_{t=1}^N$ be an IRS on a nonempty compact subset of $M$, and $K$ be the associated attractor. Assume that there exists a GIFS $G=(V,E)$ associated to $\{R_t\}_{t=1}^N$ and assume that $G$ consists of contractive similitudes. Suppose that $G$ satisfies (GOSC). Let $\lambda_\alpha$ be the spectral radius of the incidence matrix $A_\alpha$ associated to $G$.  Then
		$$\dim_H(K)=\alpha,$$
		where $\alpha$ is the unique number such that $\lambda_\alpha=1$.
	\end{thm}

	The definitions of a simplified GIFS and a minimal simplified GIFS are given in Definitions \ref{simplified graph} and \ref{defi:min}, respectively.
	
	\begin{thm}\label{thm:main2}
		Let $M$, $\{R_t\}_{t=1}^N$, $K$, and $G=(V,E)$ be as in Theorem \ref{thm:main1}. Let $\widehat{G}=(\widehat{V},\widehat{E})$ be a  minimal simplified GIFS associated to $G$. Assume that $\widehat{G}$ satisfies (GOSC).  Let $\widehat{\lambda}_\alpha$ be the spectral radius of the incidence matrix $\widehat{A}_\alpha$ associated to $\widehat{G}$.  
	\begin{enumerate}
			\item [(a)] $\dim_H(K)=\alpha$, 	where $\alpha$ is the unique real number such that $\widehat{\lambda}_\alpha=1$.
			\item [(b)] In particular, let $\widehat{G}=(\widehat{V},\widehat{E})$ be a minimal simplified GIFS associated to $G$ and assume that $\widehat{G}$ consists of contractive similitudes $\{f_t\}_{t=1}^m$. Suppose that $\#\widehat{V}=1$. If $\{f_t\}_{t=1}^m$ satisfies (OSC), then $\dim_H(K)$ is the unique number $\alpha$ satisfying
			$$\sum_{t=1}^m\rho_t^\alpha=1,$$
			where $\rho_t$ is the contraction ratio of $f_t$.	
		\end{enumerate}
	\end{thm}
	
	It is worth pointing out that Balogh and Rohner \cite{Balogh-Rohner_2007}
	extended the Moran-Hutchinson theorem (see \cite{Hutchinson_1981}) to the more general setting of doubling metric spaces. Wu and Yamaguchi \cite{Wu-Yamaguchi_2017} generalized the results of Balogh and Rohner.

For IFSs on $\mathbb{R}^n$, the finite type was first introduced by the second author and Wang \cite{Ngai-Wang_2001}, and was extended to the general finite type condition independently by Jin and Yau \cite{Jin-Yau_2005}, Lau and  the second author \cite{Lau-Ngai_2007}. The graph finite type condition (GFTC) was extended to Riemannian manifolds by  the second author and Xu \cite{Ngai-Xu_2023}. In Section \ref{S:GFTC}, we study IRS attractors under the assumption that a GIFS associated to an IRS satisfies (GFTC). Our third goal is to obtain a method for computing the Hausdorff dimension of the attractor that improves the one in \cite{Ngai-Xu_2023}.
	
	The construction of the weighted incidence matrices that appear in the following theorems will be given in Section \ref{S:GFTC}.
	\begin{thm}\label{thm:main4}
		Let $M$, $\{R_t\}_{t=1}^N$, $K$, and $G=(V,E)$ be as in Theorem \ref{thm:main1}.  Assume that $G$ satisfies (GFTC).  Let $\lambda_\alpha$ be the spectral radius of the weighted incidence matrix $A_\alpha$ associated to $G$. Then
		$$\dim_H(K)=\alpha,$$
		where $\alpha$ is the unique real number such that $\lambda_\alpha=1$.
	\end{thm}
	
	\begin{thm}\label{thm:main5}
		Let $M$, $\{R_t\}_{t=1}^N$, $K$, and $G=(V,E)$ be as in Theorem \ref{thm:main1}.  Let $\widehat{G}=(\widehat{V},\widehat{E})$ be a  minimal simplified GIFS associated to $G$. Assume that $\widehat{G}$ satisfies (GFTC). Let $\widehat{\lambda}_\alpha$ be the spectral radius of the  weighted incidence matrix $\widehat{A}_\alpha$ associated to $\widehat{G}$.  Then
		$$\dim_H(K)=\alpha,$$
		where $\alpha$ is the unique real number such that $\widehat{\lambda}_\alpha=1$.
	\end{thm}
	
	This paper is organized as follows. In Section \ref{S:Pre}, we give the definition of an IRS, and study the properties of its attractor. Section \ref{S:IFS} is devoted to the proof of Theorem \ref{thm:exi}. In Section \ref{S:GOSC}, we prove Theorems \ref{thm:main1}--\ref{thm:main2}. We also give  examples of IRSs satisfying the conditions of Theorem \ref{thm:main2}, and compute the Hausdorff dimension of the associated attractors.  In Section \ref{S:GFTC}, we prove Theorems \ref{thm:main4}--\ref{thm:main5} and provide examples to illustrate Theorem \ref{thm:main5}.

	\section{Iterated relation systems}\label{S:Pre}
	\setcounter{equation}{0}
	Throughout this paper, we let $\Sigma:=\{1,\ldots, N\}$, $\Pi_t:=\{1,\ldots,n_t\}$,  $\Delta_t:=\{1,\ldots,s_t\}$, $\Lambda_t:=\{1,\ldots,m_t\}$, and  $\Psi_t^i:=\{1,\ldots,p_t^i\}$, where $N\in \N$.

	We first introduce the definition of an IRS.
	\begin{defi}\label{defi:2.1}
		Let $ T $ be a topological space and let $E_0\subseteq T $ be a nonempty compact set.  Let $\{R_t\}_{t=1}^N$ be a family of relations defined on $E_0$.  For any $n\geq0$, let
		\begin{align}\label{E_n}
			E_{n+1}:=\bigcup_{t=1}^NR_t(E_n).
		\end{align}
		We call $\{R_t\}_{t=1}^N$  an \textit{iterated relation system  (IRS)} if it satisfies the following conditions:
		\begin{enumerate}
			\item[(a)] for any nonempty compact set $F\subseteq E_0$ and $t\in \Sigma$,  $R_t(F)$ is a nonempty compact set;
			\item[(b)] for any $t\in \Sigma$, $R_t(E_0)\subseteq E_0$.
		\end{enumerate}
	\end{defi}
	
	Examples of IRSs given in Examples \ref{exam:osc1}--\ref{exam:osctri} and \ref{exam:gftc1}--\ref{exam:gftctri}.

	\begin{prop}\label{prop:2.1}
		Let $ T $ be a topological space, and let $E_0\subseteq T $ be a nonempty compact set. Let $\{R_t\}_{t=1}^N$ be an IRS on $E_0$. For any $n\geq 0$, let $E_n$ be defined as in \eqref{E_n}.
		 Then $\bigcap_{n=0}^\infty E_n$ is a nonempty compact set.
	\end{prop}
	
	\begin{proof}
		By Definition \ref{defi:2.1}(b), we have $E_1\subseteq E_0$.
		Assume that $E_{n+1}\subseteq E_n$ for some $n\geq0$. Then
		\begin{align*}
			E_{n+2}&=\bigcup_{t=1}^NR_t(E_{n+1})\subseteq\bigcup_{t=1}^NR_t(E_n)=E_{n+1}.
		\end{align*}
		Hence for any $n\geq0$, $E_{n+1}\subseteq E_n$.
		We know  that $E_0$ is a nonempty compact set. Assume that $E_n$ is a nonempty compact set for some $n\geq1$. By Definition \ref{defi:2.1}(a), we have $R_t(E_n)$ is a nonempty compact set for any $t\in\Sigma$. Hence $E_{n+1}$ is a nonempty compact set. Therefore, $\{E_n\}_{n=0}^\infty$ is a decreasing sequence of nonempty compact subsets of $ T $. This proves the proposition.	
	\end{proof}

	\begin{defi}\label{defi:2.2}
	Let $ T $, $E_0$, and $\{R_t\}_{t=1}^N$ be as in Proposition \ref{prop:2.1}. For any $n\geq0$, let $E_n$ be defined as in \eqref{E_n}. We call $K:=\bigcap_{n=0}^\infty E_n$ the {\em invariant set} or \textit{attractor} of $\{R_t\}_{t=1}^N$.
\end{defi}	

	 Next, we introduce the following condition, which we will call Condition (C).
	 
	\begin{defi}\label{condi}
	Let $ T $, $E_0$, and $\{R_t\}_{t=1}^N$ be as in Proposition \ref{prop:2.1}. For any $n\geq 0$, let $E_n$ be defined as in \eqref{E_n}. We say that  $\{R_t\}_{t=1}^N$ satisfies {\em Condition (C)} if it has the following property.
			\begin{enumerate}
			\item[(C)] For any $x\in E_0$, if there exist some $l\in \Sigma$, a subsequence $\{n_k\}$ of $\{n\}$, and $y_{n_k}\in E_{n_k}$ satisfying $x\in R_{l}(y_{n_k})\subseteq R_{l}(E_{n_k})$ for all $k\geq 0$, then there exists a subsequence $\{y_{n_{k_j}}\}$ of $\{y_{n_k}\}$ converging to some $y\in\bigcap_{n=0}^\infty E_n$  satisfying 	$x\in R_{l}(y)$.
		\end{enumerate}
	\end{defi}

	Condition (C) ensures that the inclusion $K\subseteq \bigcup_{t=1}^NR_t$ holds.	

	\begin{prop}\label{prop:2.2}
	Let $ T $, $E_0$, and $\{R_t\}_{t=1}^N$ be defined as in Proposition \ref{prop:2.1}. Let $K$ be the attractor of $\{R_t\}_{t=1}^N$.
\begin{enumerate}
	\item[(a)] $\bigcup_{t=1}^NR_t(K)\subseteq K$.
	\item[(b)] 	If Condition (C) holds, then $K\subseteq \bigcup_{t=1}^NR_t(K)$, and thus \eqref{eq:K=RK} holds.
	\end{enumerate}
	\end{prop}
	\begin{proof}
	(a) By the definition of $K$, for any $n\geq0$, we have $K\subseteq E_n$. Hence, for any $t\in\Sigma$,
		$$R_t(K)\subseteq\bigcup_{t=1}^NR_t(E_n)\subseteq E_n.$$ 
		Thus,
	 	$\bigcup_{t=1}^NR_t(K)\subseteq\bigcap_{n=0}^\infty E_n=K.$
		
	(b)  We first prove the following claims.
		
		\noindent{\em Claim 1.  $\bigcap_{n=0}^\infty\bigcup_{t=1}^NR_t(E_n)\subseteq\bigcup_{t=1}^N\bigcap_{n=0}^\infty R_t(E_n)$.}  Let $x\in\bigcap_{n=0}^\infty\bigcup_{t=1}^NR_t(E_n)$. Then for any $n\geq 0$, there exists $l_n\in \Sigma$ such that $x\in R_{l_n}(E_{n})$. Hence there exist some $l\in \Sigma$ and a subsequence $\{n_k\}$ of $\{n\}$ such that $x\in R_{l}(E_{n_k})$ for all $k\geq 0$. Let $y_{n_k}\in E_{n_k}$ such that $x\in R_l(y_{n_k})$. By Condition (C), there exists a subsequence $\{y_{n_{k_j}}\}$ converging to some $$y\in\bigcap_{n=0}^\infty E_n\qquad \text{satisfying}\qquad x\in R_l(y)\subseteq \bigcap_{n=0}^\infty R_l(E_n).$$ Therefore, we have $x\in\bigcup_{t=1}^N\bigcap_{n=0}^\infty R_t(E_n)$.
		
		\noindent{\em Claim 2.  $\bigcup_{t=1}^N\bigcap_{n=0}^\infty R_t(E_n)\subseteq\bigcup_{t=1}^NR_t\big(\bigcap_{n=0}^\infty E_n\big)$.}  Let $x\in\bigcup_{t=1}^N\bigcap_{n=0}^\infty R_t(E_n)$. Then there exists $l\in\Sigma$ such that for any $n\geq0$, $x\in R_l(E_n)$. Let $y_n\in E_n$ such that $x\in R_l(y_n)$.
	By Condition (C), there exists a subsequence $\{y_{n_k}\}$ converging to some $$y\in\bigcap_{n=0}^\infty E_n \qquad \text{satisfying}\qquad x\in R_l(y)\subseteq R_l\Bigg(\bigcap_{n=1}^\infty E_n\Bigg).$$ Therefore, we have $x\in \bigcup_{t=1}^NR_t(\bigcap_{n=0}^\infty E_n)$.	
		By  Claims 1 and 2, we have
		$$K=\bigcap_{n=0}^\infty\bigcup_{t=1}^N R_t(E_n)\subseteq\bigcup_{t=1}^N\bigcap_{n=0}^\infty R_t(E_n)\subseteq\bigcup_{t=1}^NR_t\Bigg(\bigcap_{n=0}^\infty E_n\Bigg)=\bigcup_{t=1}^NR_t(K).$$
		Using the result of (a), we have $K= \bigcup_{t=1}^NR_t(K)$.
		This proves (b).
	\end{proof}
	
The following counterexample shows that Proposition \ref{prop:2.2}(b) fails if Condition (C) is not assumed.	
	
 \begin{exam}\label{defi(c)}
		Let $E_0:=\{0,1\}\bigcup\{1/2^n:n\in\mathbb{N}_+\}$ and let $R_1$ be an IRS on $E_0$  defined as
		\begin{align*}
			R_1(\boldsymbol{x}) &:=
			\begin{cases}
				1,  &   \,\, \boldsymbol{x}=0, \\
				\{0,1/2^{n+1}\}, &   \,\,\boldsymbol{x}=1/2^n,\text{ where }n\in\mathbb{N}_+,\\
				1, &   \,\,\boldsymbol{x}=1.\\
			\end{cases}
		\end{align*}
	Then $E_n=\{0,1\}\bigcup\{1/2^k:k\geq n+1\}$, where $n\in\mathbb{N}_+$.
		 Let $x=0\in E_0$,  $y_{n_k}=1/2^k\in E_{n_k}$, and $0\in R_1(y_{n_k})$, where $k\geq n+1$ and $n\in\mathbb{N}_+$. Then
	 $\{y_{n_k}\}$  converges to $y:=0$, but  $x\notin R_1(y)=\{1\}$. Hence $R_1$ does not satisfy Condition (C). Moreover, we have
			 $K:=\bigcap_{n=0}^\infty E_n=\{0,1\}$ and $R_1(K)=\{1\}$. Hence $K\not\subseteq R_1(K)$.

\end{exam}
	
	For a general IRS, there could be more than one nonempty compact set $K$ satisfying  \eqref{eq:K=RK}. The following example illustrates this.

	\begin{exam}\label{exa:notunique}
		Let $E_0:=[0,1]\subseteq \R$, and let $\{R_t\}_{t=1}^2$ be an IRS on $E_0$ defined as
		\begin{align*}
			R_1(\boldsymbol{x}) &:= (3/4)\boldsymbol{x};\\
			R_2(\boldsymbol{x}) &:=
			\begin{cases}
				0,  &    \boldsymbol{x}=0, \\
				(1/4)\boldsymbol{x}+3/4, &    0<\boldsymbol{x}<1, \\
				\{3/4,1\}, &    \boldsymbol{x}=1.
			\end{cases}
		\end{align*}
		Then there exist at least two nonempty compact sets $K$ satisfying equation \eqref{eq:K=RK}.
	\end{exam}
	\begin{proof}
		We know that $E_0=[0,1]$, \ldots, $E_n=[0,1]$. Hence $K=[0,1]$. By Proposition \ref{prop:2.2}, we have $K=\bigcup_{t=1}^2R_t(K)$. Now let $\widetilde{K}=\{0\}$. Note that $$\{0\}=R_1\big(\{0\}\big)\bigcup R_2\big(\{0\}\big).$$ Then $\widetilde{K}=\bigcup_{t=1}^2R_t(\widetilde{K})$, which  completes the proof.
	\end{proof}
	
	In the next section, we will study the uniqueness of a nonempty compact set $K$ satisfying the equality \eqref{eq:K=RK}.

	\section{Associated graph-directed iterated function systems}\label{S:IFS}
	\setcounter{equation}{0}
	In this section, we let $ X $ be a complete metric space, $E_0\subseteq X $ be a nonempty compact set, and $\{R_t\}_{t=1}^N$ be an IRS on $E_0$. We introduce the notion of a graph-directed iterated function system (GIFS) associated to $\{R_t\}_{t=1}^N$, and prove Theorem \ref{thm:exi}.
	
	{\em A graph-directed iterated function system} (GIFS) of contractions $\{f_e\}_{e\in E}$ on $X$ is an orderd pair $G = (V, E)$, where $V = \{1,\ldots, m\}$ is the set of vertices and $E$ is
	the set of all directed edges.  We allow more than one edge between two vertices. A {\em directed path} in $G$ is a finite string $\mathbf{e} =(e_1,\ldots, e_q)$ of edges in $E$ such that the terminal vertex of each $e_i$ is the initial vertex of the edge $e_{i+1}$. For such a path, denote the length of  $\mathbf{e}$ by
	$|\mathbf{e}|=q$. For any two vertices $i,j\in V$ and any positive integer $q$, let $E^{i,j}$ be the set
	of all directed edges from $i$ to $j$, $E^{i,j}_q$ be the set of all directed paths of length $q$ from
	$i$ to $j$, $E_q$ be the set of all directed paths of length $q$, and $E^{*}$ be the set of all directed
	paths, i.e.,
	\begin{align*}
		E_q:=\bigcup_{i,j=1}^mE^{i,j}_q\qquad \text{and}\qquad E^{*}:=\bigcup_{q=1}^\infty E_q.
	\end{align*}
Then there exists a unique collection of nonempty compact sets  $\{\widetilde{K}_i\}_{i\in V}$ satisfying
\begin{align*}
	\widetilde{K}_i=\bigcup_{j=1}^m\bigcup_{e\in E^{i,j}} f_e(\widetilde{K}_j), \qquad i\in V
\end{align*}
(see, e.g., \cite{Mauldin-Williams_1988,Edgar_1990,Olsen_1994}).
Let	$\widetilde{K}:=\bigcup_{i=1}^m \widetilde{K}_i$
be the {\em invariant set} or {\em attractor} of the GIFS.
Recall that a GIFS $G = (V, E)$ is said to be {\em strongly connected} provided that whenever $i,j$, there exists a directed path from $i$ to $j$.
	
	\begin{defi}\label{defi:3.1}
		Let $X$ be a complete metric space and let $E_0\subseteq X $ be a nonempty compact set. Let $\{R_t\}_{t=1}^N$ be an IRS on $E_0$. For any $k\geq0$, let $E_{k+1}$  be defined as in \eqref{E_n}. We assume that there exists a finite integer $k_0\geq 0$ such that $E_{k_0}=\bigcup_{j=1}^mW_j$, where each $W_j\subseteq E_{k_0}$ is compact. Suppose that there exists a GIFS $G=(V,E)$ of contractions $\{f_e\}_{e\in E}$, where $V=\{1,\ldots,m\}$,  such that for any $q\geq 1$,
			\begin{align}\label{eq:gifs}
				\bigcup_{t\in\Sigma^q}R_t(E_{k_0})=\bigcup_{i,j=1}^m\bigcup_{\mathbf{e}\in E_{q}^{i,j}}f_\mathbf{e}(W_j).
			\end{align}
		Then $G$ is called a \textit{graph-directed iterated function system (GIFS) associated to} $\{R_t\}_{t=1}^N$ on $E_{k_0}$. We call
	\begin{align*}
	\widetilde{K}:=\bigcap_{q=1}^\infty\bigcup_{i,j=1}^m\bigcup_{\mathbf{e}\in E_{q}^{i,j}}f_\mathbf{e}(W_j)
	\end{align*}
 the {\em attractor} generated by the GIFS $G$ associated to $\{R_t\}_{t=1}^N$. 
	\end{defi}

	\begin{comment}
	\begin{exam}
	Use an example to show that no matter how $E_0$ is divided, $\{R_t\}_{t=1}^N$ is an IRS. Let $E_0:=[0,1]$, let IRS $\{R_t\}_{t=2}$ on $E_0$ satisfy:
	\begin{align*}
	&R_1(\boldsymbol{x})=
	\begin{cases}
	\{0,1/2\},  &    \boldsymbol{x}=0, \\
	(1/3)\boldsymbol{x}, &    0<\boldsymbol{x}\leq1;
	\end{cases}\\
	&R_2(\boldsymbol{x})=(1/3)\boldsymbol{x}+2/3.
	\end{align*}
	\end{exam}}
	\end{comment}
	We have the following proposition.
	\begin{prop}\label{prop:3.1}
		Let $X$, $E_0$, and $\{R_t\}_{t=1}^N$ be as in Definition \ref{defi:3.1}.  Let $K$ be the attractor of $\{R_t\}_{t=1}^N$. Assume that there exists a GIFS $G=(V,E)$ associated to $\{R_t\}_{t=1}^N$ and assume that $G$ consists of contractions $\{f_e\}_{e\in E}$.
	 Let $\widetilde{K}$ be the attractor of G. Then $K=\widetilde{K}$.
	\end{prop}
	\begin{proof}
		By Definition \ref{defi:3.1}, for any $q\geq1$, we have
		\begin{align}\label{eq:q}
			\bigcap_{q=1}^\infty\bigcup_{t\in\Sigma^q}R_t(E_{k_0})=\bigcap_{q=1}^\infty\bigcup_{i,j=1}^m\bigcup_{\mathbf{e}\in E_{q}^{i,j}}f_\mathbf{e}(W_j).
		\end{align}
		By Definition \ref{defi:2.1} and \eqref{eq:q}, we have
		\begin{align*}
			K=\bigcap_{n=0}^\infty E_{n+1}=\bigcap_{q=1}^\infty E_{k_0+q}=\bigcap_{q=1}^\infty\bigcup_{t\in\Sigma^q}R_t(E_{k_0})
			=\bigcap_{q=1}^\infty\bigcup_{i,j=1}^m\bigcup_{\mathbf{e}\in E_{q}^{i,j}}f_\mathbf{e}(W_j)= \widetilde{K}.
		\end{align*}
	\end{proof}

	\begin{prop}\label{prop:3.2}
		Let $\{R_t\}_{t=1}^N$ and  $K$ be defined as in Definition \ref{defi:3.1}. Assume that there exists a GIFS $G=(V,E)$ associated to $\{R_t\}_{t=1}^N$ and assume that $G$ consists of contractions $\{f_e\}_{e\in E}$.  Then there exists a unique nonempty compact set $K$ satisfying \eqref{eq:K=RK}.
	\end{prop}

	\begin{proof}
		Assume that there exists another nonempty compact set $K_1$ satisfying \eqref{eq:K=RK}. Note that 
		\begin{align*}
			K_1=\bigcup_{t_1=1}^N\cdots\bigcup_{t_n=1}^NR_{t_1}\cdots R_{t_n}(K_1)
			\subseteq\bigcup_{t_1=1}^N\cdots\bigcup_{t_n=1}^NR_{t_1}\cdots R_{t_n}(E_0). 
		\end{align*}
		Hence
		\begin{align*}
			&	K_1\subseteq\bigcap_{n=1}^\infty\bigcup_{t_1=1}^N\cdots\bigcup_{t_n=1}^NR_{t_1}\cdots R_{t_n}(E_0)
			=\bigcap_{n=1}^\infty E_{n}
			=K.
		\end{align*}
		Next  we show that $K\subseteq	K_1$. 
	We know that
		$$	K_1=\bigcap_{n=1}^\infty\bigcup_{t_1=1}^N\ldots\bigcup_{t_n=1}^NR_{t_1}\ldots R_{t_n}(K_1).$$
		For any $e\in E$, let $\rho_e$ be a contraction ratio of $f_e$.
		Then for any $x\in\widetilde{K}$ and for any integer $m$, $n$, where $m>n>0$, we have
		\begin{align*}
			&\big|R_{t_1}\cdots R_{t_{m-k_0}}(R_{t_{m-k_0+1}}\cdots R_{t_{m}}x)-R_{t_1}\cdots R_{t_{n-k_0}}(R_{t_{n-k_0+1}}\cdots R_{t_{n}}x)\big|\\
			=&\big|f_{e_{j_1}}\cdots f_{e_{j_{m-k_0}}}(R_{t_{m-k_0+1}}\cdots R_{t_{m}}x)-f_{e_{j_1}}\cdots f_{e_{j_{n-k_0}}}(R_{t_{n-k_0+1}}\cdots R_{t_{n}}x)\big|\\
			\leq&\,\rho_{e_{j_1}}\cdots \rho_{e_{j_{n}}}|E_{k_0}|.
		\end{align*}
		Hence $\lim_{n\rightarrow\infty}R_{t_1}\cdots R_{t_{n-k_0}}(R_{t_{n-k_0+1}}\cdots R_{t_n}x)$ exists. Thus, for any $z\in K$, there exists $x^*\in K_1$ satisfying
		$$z=\lim_{n\rightarrow\infty}f_{e_{j_1}}\cdots f_{e_{j_{n-k_0}}}(x^*)\in K_1.$$
		Therefore $K\subseteq K_1$. This proves the proposition. 
	\end{proof}
	
Let $G=(V,E)$ be a GIFS of contractions $\{f_e\}_{e\in E}$.  We say that $\{U_i\}_{i=1}^m$ is an {\em invariant family}  under  $G$ if
		\begin{align*}
			\bigcup_{e\in E^{i,j}}f_e(U_j)\subseteq U_i\qquad \text{for all}\,\,i,j\in \{1,\ldots,m\}.
	\end{align*}
	
	\begin{defi}\label{defi:good}
		Let $ X $, $E_0$, $\{R_t\}_{t=1}^N$, and $k_0$ be defined as in Definition \ref{defi:3.1}.
		\begin{enumerate}
			\item[(a)]	 Assume that conditions (i) and (ii) below hold.
			\begin{enumerate}
				\item[(i)]	For any $t\in\Sigma$ and $x\in E_{k_0-1}$,
				$\#\{R_t(x)\}<\infty.$
				\item[(ii)] For any $t\in\Sigma$, there exists $F_t^i\in{\rm dom}(R_t)\subseteq E_{k_0-1}$ such that for each $l\in \Pi_t$ and $i\in\{1,\ldots,i_t\}$, there exists a contraction $f_t^{l,i}: F_t^i\to f_t^{l,i}(F_t^i)$ and the following holds
				$$\bigcup_{t=1}^N\bigcup_{i=1}^{i_t}R_t(F_t^i)=\bigcup_{t=1}^N\bigcup_{l=1}^{n_t}\bigcup_{i=1}^{i_t}f_t^{l,i}(F_t^i).$$
			\end{enumerate}
			Then we say that $\{f_t^{l,i}\}_{t=1,l=1,i=1}^{N,n_t,i_t}$ is {\em a finite family of contractions decomposed from $\{R_t\}_{t=1}^N$}, and that each $f_t^{l,i}$ is a {\em branch} of $R_t|_{F_t^i}$.
			\item[(b)] Let $\{W_t^{\alpha,\beta}\}_{t=1,\alpha=1,\beta=1}^{N,m_t,p_t}$ be a partition of $E_{k_0}$ satisfying
			$$\bigcup_{t=1}^N\bigcup_{\alpha=1}^{m_t}\bigcup_{\beta=1}^{p_t}W_t^{\alpha,\beta}=E_{k_0}.$$
			Let $\mathcal{W}_{k_0}:=\{W_t^{\alpha,\beta}:t\in \Sigma, \alpha\in \Lambda_t,\beta\in \{1,\ldots,p_t\}\}$ be a collection of compact subsets of $E_{k_0}$. Let $\{g_t^{l,\alpha,\beta}:W_t^{\alpha,\beta}\to g_t^{l,\alpha,\beta}(W_t^{\alpha,\beta})\}_{t=1,l=1,\alpha=1,\beta=1}^{N,n_t,m_t,p_t}$ be a finite family of contractions decomposed from $\{R_t\}_{t=1}^N$. Fix $t\in \Sigma$. Assume that for any $\alpha\in \Lambda_t$ and $\beta\in \Psi_t$,  $W_t^{\alpha,\beta}$ is invariant under $g_t^{l,\alpha,\beta}$ for all $l\in \Pi_t$.
			Then  we call $\mathcal{W}_{k_0}$ {\em a good partition of $E_{k_0}$ with respect to $\{g_t^{l,\alpha,\beta}\}_{t=1,l=1,\alpha=1,\beta=1}^{N,n_t,m_t,p_t}$}.
		\end{enumerate}
		
	\end{defi}
	
	\begin{proof}[Proof of Theorem \ref{thm:exi}] To prove this theorem, we consider the following two cases.
		
		\noindent{\em Case 1.  For any $t\in\Sigma$, $H_t=\emptyset$.} In this case, $\{R_t\}_{t=1}^N$ is an IFS. Hence there exists a GIFS associated to $\{R_t\}_{t=1}^N$ on $E_{k_0}$, where $k_0\geq0$.

		\noindent{\em Case 2.  For some $t\in\Sigma$, $H_t\neq\emptyset$.}
		In this case, in order to prove the existence of the GIFS associated to $\{R_t\}_{t=1}^N$, we first prove the following five claims.
		
		\noindent{\em Claim 1.  For any $t\in\Sigma$, we have $R_t(E_{k_0-1})=\Big(\bigcup_{l=1}^{n_t}\bigcup_{i=1}^{m_t}\overline{h_t^{l,i}(H_t^i)}\Big)\bigcup\Big(\bigcup_{i=1}^{s_t}\overline{h_t^{0,i}(J_t^i)}\Big)$.}
		By (b)(i), we have
		$$R_t(E_{k_0-1})=\bigcup_{l=1}^{n_t}\bigcup_{i=1}^{m_t}h_t^{l,i}(H_t^i)\bigcup \Big(\bigcup_{i=1}^{s_t} h_t^{0,i}(J_t^i)\Big)\subseteq\bigcup_{l=1}^{n_t}\bigcup_{i=1}^{m_t}\overline{h_t^{l,i}(H_t^i)}\bigcup\Big(\bigcup_{i=1}^{s_t}\overline{h_t^{0,i}(J_t^i)}\Big).$$
		Since $R_t(E_{k_0-1})$ is compact, we have
		$\bigcup_{l=1}^{n_t}\bigcup_{i=1}^{m_t}\overline{h_t^{l,i}(H_t^i)}\bigcup\Big(\bigcup_{i=1}^{s_t}\overline{h_t^{0,i}(J_t^i)}\Big)\subseteq R_t(E_{k_0-1}).$
		This proves Claim 1.
		
		\noindent{\em Claim 2.  For any $t\in\Sigma$, let
			\begin{align*}
				& \underline{W}_t^{0,i}:=\overline{h_t^{0,i}(J_t^i)},\quad\text{where}\,\,i\in\Delta_t,\,\,\text{and}\\
				&\underline{W}_t^{l,i}:=\overline{h_t^{l,i}(H_t^i)},\quad \underline{W}_t^{n_t+1,i}:=\overline{E_{k_0}\bigcap H_t^i},\quad\text{where}\,\,l\in\Pi_t\,\,\text{and}\,\, i\in\Lambda_t.
			\end{align*}
			Then
			\begin{align}\label{eq:E_k0}
				E_{k_0}=\bigcup_{t=1}^N\Bigg(\bigcup_{i=1}^{s_t}\underline{W}_t^{0,i}\bigcup\Big( \bigcup_{l=1}^{n_t+1}\bigcup_{i=1}^{m_t}\underline{W}_t^{l,i}\Big)\Bigg).
			\end{align}
		} By Claim 1, we have
		\begin{align*}
			E_{k_0}=&\bigcup_{t=1}^NR_t(E_{k_0-1})=\bigcup_{t=1}^N\Bigg(\Big(\bigcup_{l=1}^{n_t}\bigcup_{i=1}^{m_t}\overline{h_t^{l,i}(H_t^i)}\Big)\bigcup\Big(\bigcup_{i=1}^{s_t}\overline{h_t^{0,i}(J_t^i)}\Big)\Bigg)\\
			=&\bigcup_{t=1}^N\Bigg(\bigcup_{i=1}^{s_t}\underline{W}_t^{0,i}\bigcup \Big( \bigcup_{l=1}^{n_t}\bigcup_{i=1}^{m_t}\underline{W}_t^{l,i}\Big)\Bigg).
		\end{align*}
		Note that for any $t\in\Sigma$ and $i\in\Lambda_t$, we have
		$\underline{W}_t^{n_t+1,i}=\overline{E_{k_0}\bigcap H_t^i}\subseteq E_{k_0}.$
		Hence \eqref{eq:E_k0} holds.
		
		\noindent{\em Claim 3.  For any $t\in\Sigma$ and $i\in\Lambda_t$, let $\{x_n\}$ be a convergent sequence in $H_t^i$. Then for any $l\in\Pi_t$, the sequence $\{h_t^{l,i}(x_n)\}$ converges. If $H_t^i$ is an open set, then we can extend $h_t^{l,i}$ from $H_t^i$ to $\overline{H_t^i}$ and let $\tilde{h}_t^{l,i}:\overline{H_t^i}\rightarrow \underline{W}_t^{l,i}$ be defined as
			\begin{align}\label{eq:hl}
				\tilde{h}_t^{l,i}(x): =
				\begin{cases}
					h_t^{l,i}(x),  &   x\in H_t^i, \\
					\lim_{n\rightarrow\infty}h_t^{l,i}(x_n), &   x\in\overline{H_t^i}\backslash H_t^i,
				\end{cases}
			\end{align}
			where for any $x\in  \overline{H_t^i}\backslash H_t^i$, $x_n\rightarrow x$ ($n\rightarrow\infty$). Moreover, $\tilde{h}_t^{l,i}$ is a surjection.}
		In fact, by (b)(i),  we have $h_t^{l,i}$ is uniform continuous in $H_t^i$. Hence the sequence $\{h_t^{l,i}(x_n)\}$ converges, and thus we can define $h_t^{l,i}(x):=\lim_{n\rightarrow\infty}h_t^{l,i}(x_n)$ if $x\in \overline{H_t^i}\backslash H_t^i$. Note that for any $x\in\overline{H_t^i}\backslash H_t^i$, $h_t^{l,i}(x)$ is independent of the choice of the sequence $\{x_n\}$.
		Therefore we can extend $h_t^{l,i}$ from $H_t^i$ to $\overline{H_t^i}$ and let $\tilde{h}_t^{l,i}$ be defined as in \eqref{eq:hl}. To show that $\tilde{h}_t^{l,i}$ is a surjection, we let $y\in \underline{W}_t^{l,i}=\overline{h_t^{l,i}(H_t^i)}$. Then there exists a sequence $\{h_t^{l,i}(x_n)\}\subseteq h_t^{l,i}(H_t^i)$ such that
		\begin{align}\label{eq:6a}
			\lim_{n\rightarrow\infty} h_t^{l,i}(x_n)=y.
		\end{align}
		We know that $\{x_n\}$ is bounded, and hence there exists a convergent subsequence $\{x_{n_k}\}$ such that
		\begin{align}\label{eq:6b}
			\lim_{n\rightarrow\infty} x_{n_k}=x\in\overline{H_t^i}.
		\end{align}
		Note that
		\begin{align}\label{eq:6c}
			|\tilde{h}_t^{l,i}(x)-y|\leq|\tilde{h}_t^{l,i}(x)-h_t^{l,i}(x_{n_k})|+|h_t^{l,i}(x_{n_k})-y|.
		\end{align}
		Combining (b)(i) and the definition of $\tilde{h}_t^{l,i}$, we have for all $\delta>0$, there exists $ N_1\in\mathbb{N}$ such that for all $k>N$,
		\begin{align}\label{eq:6d}
			|h_t^{l,i}(x_{n_k})-\tilde{h}_t^{l,i}(x)|<\frac{\varepsilon}{2}.
		\end{align}
		By \eqref{eq:6a} and \eqref{eq:6b}, we have $\lim_{k\rightarrow\infty}h_t^{l,i}(x_{n_k})=y$, and thus for any $\varepsilon>0$, there exists $N_2\in\mathbb{N}$ such that for all $k>N_2$,
		\begin{align}\label{eq:6e}
			|h_t^{l,i}(x_{n_k})-y|<\frac{\varepsilon}{2}.
		\end{align}
		Let $N:=\max\{N_1,N_2\}$. Then for any $k>N$, \eqref{eq:6d} and \eqref{eq:6e} hold. By \eqref{eq:6c}, we have
		$$|\tilde{h}_t^{l,i}(x)-y|<\frac{\varepsilon}{2}+\frac{\varepsilon}{2}=\varepsilon.$$
		It follows that $\tilde{h}_t^{l,i}(x)=y$. This proves the Claim 3.

		\noindent{\em Claim 4.  For any $t\in\Sigma$ and $i\in\Delta_t$, let $\{x_n'\}$ be a convergent sequence in $J_t^i$. Then the sequence $\{h_t^{0,i}(x_n')\}$ converges. If $J_t^i$ is an open set, then  we can extend  $h_t^{0,i}$ from $J_t^i$ to $\overline{J_t^i}$, and let $\tilde{h}_t^{0,i}:\overline{J_t^i}\rightarrow \underline{W}_t^{0,i}$ be defined as
			\begin{align}\label{eq:h0}
				\tilde{h}_t^{0,i}(x): =
				\begin{cases}
					h_t^{0,i}(x),  & x\in J_t^i,\\
					\lim_{n\rightarrow\infty}h_t^{l,i}(x'_n), & x\in \overline{J_t^i}\backslash J_t^i,
				\end{cases}
			\end{align}
			where for any $x\in \overline{J_t^i}\backslash J_t^i$, $x'_n\rightarrow x$ when $n\rightarrow\infty$. Moreover, $\tilde{h}_t^{0,i}$ is a surjection.}
		The proof of Claim 4 is similar to that of Claim 3; we omit the details.

		\noindent{\em Claim 5.  For any $t\in\Sigma$,
			\begin{align}\label{eq:9f}
				R_t(E_{k_0})=\bigcup_{l=1}^{n_t}\bigcup_{i=1}^{m_t}\tilde{h}_t^{l,i}\big(\overline{E_{k_0}\bigcap H_t^i}\big)\bigcup \Bigg(\bigcup_{i=1}^{s_t}\tilde{h}_t^{0,i}\big(\overline{J_t^i}\big)\Bigg).
		\end{align}}
		In fact, by using a method similar to that for Claim 1, for any $t\in\Sigma$, we have
		\begin{align}\label{eq:9a}
			R_t(E_{k_0})=\bigcup_{l=1}^{n_t}\bigcup_{i=1}^{m_t}\overline{h_t^{l,i}(E_{k_0}\bigcap H_t^i)}\bigcup\Big(\bigcup_{i=1}^{s_t}\overline{h_t^{0,i}(J_t^i)}\Big).
		\end{align}
		 As in the proof for Claim 3, for any $t\in\Sigma,i\in\Lambda_t$, and $l\in\Pi_t$, we have
		\begin{align}\label{eq:9b}
			\overline{h_t^{l,i}(E_{k_0}\bigcap H_t^i)}\subseteq\tilde{h}_t^{l,i}(\overline{E_{k_0}\bigcap H_t^i}).
		\end{align}
		By Claim 3, for any $x\in(E_{k_0}\bigcap H_t^i)\subseteq H_t^i$, we have
		$$\tilde{h}_t^{l,i}(x)=h_t^{l,i}(x)\subseteq h_t^{l,i}(E_{k_0}\bigcap H_t^i)\subseteq\overline{h_t^{l,i}(E_{k_0}\bigcap H_t^i)}.$$
		If $x\in\overline{E_{k_0}\bigcap H_t^i}\backslash(E_{k_0}\bigcap H_t^i)$, then $x\in\overline{H_t^i}\backslash H_t^i$. We know that there exists a sequence $\{x_n\}\subseteq(E_{k_0}\bigcap H_t^i)$, converging to $x$, such that
		\begin{align*}
			\tilde{h}_t^{l,i}(x)=\lim_{n\rightarrow\infty}h_t^{l,i}(x_n)\subseteq h_t^{l,i}(E_{k_0}\bigcap H_t^i).
		\end{align*}
		Hence for any $t\in\Sigma,i\in\Lambda_t$, and $l\in\Pi_t$,
		\begin{align}\label{eq:9c}
			\tilde{h}_t^{l,i}(\overline{E_{k_0}\bigcap H_t^i})\subseteq\overline{h_t^{l,i}(E_{k_0}\bigcap H_t^i)}.
		\end{align}
		Combining \eqref{eq:9b} and \eqref{eq:9c}, for any  $t\in\Sigma,i\in\Lambda_t$, and $l\in\Pi_t$, we have
		\begin{align}\label{eq:9d}
			\tilde{h}_t^{l,i}(\overline{E_{k_0}\bigcap H_t^i})=\overline{h_t^{l,i}(E_{k_0}\bigcap H_t^i)}.
		\end{align}
		By using a similar argument,  for any $t\in \Sigma$ and $i\in\Delta_t$, we have
		\begin{align}\label{eq:9e}
			\tilde{h}_t^{0,i}(\overline{J_t^i})=\overline{h_t^{0,i}(J_t^i)}.
		\end{align}
		Combining \eqref{eq:9a}, \eqref{eq:9d}, and \eqref{eq:9e} proves \eqref{eq:9f}.

		 Next, fix $t\in\Sigma$. By (b)(ii) and Claim 2, we can rename the nonempty elements in  $\{\underline{W}_t^{l,i}\}_{t=1,s=1,i=1}^{N,n_t+1,m_t}$ and $\{W_t^{0,i}\}_{t=1,i=1}^{N,s_t}$ as $W_t^{s,i}$,  where $s$ and $i$ satisfy the following conditions.
		\begin{enumerate}
			\item[(1)] For any $s\in\Psi_t^i$ and $i\in\Lambda_t$, $W_t^{s,i}\subseteq\overline{E_{k_0}\bigcap H_t^i}$.
			\item[(2)] For any  $s\in\{p_t^i+1,\ldots,p_t^i+h_t^i\}$ and $i\in\Delta_t$,  $W_t^{s,i}\subseteq\overline{J_t^i}$.
			\item[(3)] $\{W_t^{s,i}\}_{t=1,s=p_t^i+1,i=1}^{N,p_t^i+h_t^i,s_t}$ is invariant under $g_t^{0,s,i}:=\tilde{h}_t^{0,i}|_{W_t^{s,i}}$, and $\{W_t^{s,i}\}_{t=1,s=1,i=1}^{N,p_t^i,m_t}$ is invariant under $g_t^{l,s,i}:=\tilde{h}_t^{l,i}|_{W_t^{s,i}}$, where $l\in \Pi_t$.
		\end{enumerate}
		Note that
		$$\overline{E_{k_0}\bigcap H_t^i}=\bigcup_{s=1}^{p_t^i}W_t^{s,i} \qquad\text{and}\qquad \overline{J_t^i}=\bigcup_{s=p_t^i+1}^{p_t^i+h_t^i}\tilde{h}_t^{0,i}(W_t^{s,i}).$$
		 By \eqref{eq:9f}, for any $t\in\Sigma$, we have
		\begin{align*}
			R_t(E_{k_0})&=\bigcup_{l=1}^{n_t}\bigcup_{i=1}^{m_t}\tilde{h}_t^{l,i}\big(\overline{E_{k_0}\bigcap H_t^i}\big)\bigcup \Big(\bigcup_{i=1}^{s_t}\tilde{h}_t^{0,i}\big(\overline{J_t^i}\big)\Big)\\
			&=\Bigg(\bigcup_{l=1}^{n_t}\bigcup_{i=1}^{m_t}\bigcup_{s=1}^{p_t^i}\tilde{h}_t^{l,i}\Big(W_t^{s,i}\Big)\Bigg)\bigcup\Bigg(\bigcup_{i=1}^{s_t}\bigcup_{s=p_t^i+1}^{p_t^i+h_t^i}\tilde{h}_t^{0,i}\Big(W_t^{s,i}\Big)\Bigg)\\
			&=\Bigg(\bigcup_{l=1}^{n_t}\bigcup_{i=1}^{m_t}\bigcup_{s=1}^{p_t^i}g_t^{l,s,i}(W_t^{s,i})\Bigg)\bigcup\Bigg(\bigcup_{i=1}^{s_t}\bigcup_{s=p_t^i+1}^{p_t^i+h_t^i}g_t^{0,s,i}(W_t^{s,i})\Bigg).
		\end{align*}
		By the definitions of $g_t^{l,s,i}$ and $W_t^{s,i}$, for fixed $t$, $l$, $s$, $i$, and $g_u:=g_t^{l,s,i}$, $W_j:=W_t^{s,i}$, we have
		$$g_u:W_j\rightarrow W_k, \quad\text{for}\,\, j,k\in\{1,\ldots,p\},$$
		where $W_k:=g_t^{l,s,i}(W_t^{s,i})$ and $u\in\{1,\ldots,L\}$.  Let $e$ be the edge corresponding to $g_u$, and let $E^{k,j}$ be the set of all edges from $k$ to $j$.  Let
		$$f_e:=g_u,\quad e\in E^{k,j}.$$
	    Then for any $e\in E$, $f_e$ is contractive.
		Hence $G:=(V,E)$ is a GIFS of contractions $\{f_e\}_{e\in E}$, where $V:=\{1,\ldots,p\}$ and $E=\bigcup_{k,j=1}^pE^{k,j}$ is the set of all edges.
		
		Finally, we show that  $G$ is a GIFS associated to $\{R_t\}_{t=1}^N$ on $E_{k_0}$.
		By \eqref{eq:9f} and the definition of $g_u$, we have
		\begin{align}\label{eq:10b}
			\bigcup_{t=1}^NR_t(E_{k_0})
			=\Bigg(\bigcup_{u=1}^{L'}\bigcup_{j=1}^{p'}g_u(W_j)\Bigg)\bigcup\Bigg(\bigcup_{u=L'+1}^{L}\bigcup_{j=p'+1}^pg_u(W_j)\Bigg)=\bigcup_{k,j=1}^p\bigcup_{e\in E^{k,j}}f_e(W_j).
		\end{align}
		Now we show by induction that for $q\geq1$,
		\begin{align}\label{eq:10c}
			\bigcup_{t\in\Sigma^q}R_t(E_{k_0})=\bigcup_{k,j=1}^p\bigcup_{\mathbf{e}\in E_q^{k,j}}f_\mathbf{e}(W_j).
		\end{align}
		For $q=1$,  \eqref{eq:10b} implies that \eqref{eq:10c} holds. Assume that \eqref{eq:10c} holds when $q=q_0$. For $q=q_0+1$,
		\begin{align*}
			&\bigcup_{t\in\Sigma^{q_0+1}}R_t(E_{k_0})=\bigcup_{t_1=1}^NR_{t_1}\Bigg(\bigcup_{k,j=1}^p\bigcup_{e\in E_{q_0}^{k,j}}f_e(W_j)\Bigg)\\
			=&\Bigg(\bigcup_{u=1}^{L'}\bigcup_{i=1}^{m'}g_u\Bigg(\bigcup_{j=1}^p\bigcup_{\mathbf{e}\in E_{q_0}^{k,j}}f_\mathbf{e}(W_j)\Bigg)\Bigg)\bigcup\Bigg(\bigcup_{u=L'+1}^{L}\bigcup_{i=m'+1}^pg_u\Bigg(\bigcup_{j=1}^p\bigcup_{\mathbf{e}\in E_{q_0}^{k,j}}f_\mathbf{e}(W_j)\Bigg)\Bigg)\\
			=&\bigcup_{k,s=1}^p\bigcup_{\mathbf{e}\in E^{s,k}}f_\mathbf{e}\Bigg(\bigcup_{j=1}^p\bigcup_{\mathbf{e}\in E_{q_0}^{k,j}}f_\mathbf{e}(W_j)\Bigg)\qquad (\text{by \eqref{eq:10b}})\nonumber\\
			=&\bigcup_{k,j=1}^p\bigcup_{\mathbf{e}\in E_{q_0+1}^{k,j}}f_\mathbf{e}(W_j).
		\end{align*}
		Hence \eqref{eq:10c} holds.
		By definition, $G$ is a GIFS associated to $\{R_t\}_{t=1}^N$ on $E_{k_0}$, where $k_0\geq1$. This proves Case 2. Combining Cases 1 and 2 completes the proof.
	\end{proof}
We give a counterexample to show that Theorem \ref{thm:exi} fails if condition (a) is not satisfied. 
	\begin{exam}\label{thm(a)}
		Let $E_0:=[0,1]$, and let  $\{R_t\}_{t=1}^2$  be an IRS on $E_0$ defined as
		\begin{align*}
			R_1(\boldsymbol{x}) &:= (1/3)\,\boldsymbol{x}+2/3;\\
			R_2(\boldsymbol{x}) &:=
			\begin{cases}
				(1/3)\boldsymbol{x},  &   0\leq\boldsymbol{x}\leq1, \\
				\{1/2^p:p\in\mathbb{N}\}, &   \boldsymbol{x}=0.
			\end{cases}
		\end{align*}
		Then $\#\{R_2(0)\}=\infty.$ Thus, there does not exist any (finite) GIFS associated to $\{R_t\}_{t=1}^2$.
	\end{exam}
	
	Examples \ref{notcon} and \ref{decomposed} show that the conditions (b)(i) in Theorem \ref{thm:exi} need not be satisfied.
	\begin{exam}\label{notcon}
		Let $E_0:=[0,1]$. Let $\{R_t\}_{t=1}^2$ be an IRS on $E_0$ defined as
		\begin{align*}
			&R_1(\boldsymbol{x}): =
			\begin{cases}
				\{\boldsymbol{x},\boldsymbol{x}+1/4\},  & \boldsymbol{x}\in [0,1/2], \\
				(1/2)\boldsymbol{x}+1/4, & \boldsymbol{x}\in (1/2];\\
			\end{cases}\\
			&R_2(\boldsymbol{x}): =\boldsymbol{x}/2.
		\end{align*}
		Then $H_1=[0,1/2]$, and thus for any $\boldsymbol{x}\in H_1$,
		$$h_1^{1,1}(\boldsymbol{x})=\boldsymbol{x}\qquad\text{and}\qquad h_1^{2,1}(\boldsymbol{x})=\boldsymbol{x}+1/4.$$
	Therefore $h_1^{1,1}$ and $h_1^{2,1}$ are not contractions.
	\end{exam}

	\begin{exam}\label{decomposed}
		Let $E_0:=[0,1]$. For any $i\in\mathbb{N}_+$, let $F^i:=[1-(1/2^{2i-2}),1-(1/2^{2i-1})]$, and define
		$$f^{1,i}(\boldsymbol{x}):=\frac{1}{2}\boldsymbol{x}+\frac{2^{2i-1}-2}{2^{2i}}\qquad\text{and}\qquad f^{2,i}(\boldsymbol{x}):=\frac{1}{3}\boldsymbol{x}+\frac{2^{2i+1}-8}{3\cdot2^{2i}}\qquad\text{for}\quad\boldsymbol{x}\in F^i.$$
		Note that
		\begin{align*}
			&f^{1,i}(F^i)=\Big[\frac{1}{2}+\frac{2^{2i-1}-4}{2^{2i}},\,\frac{1}{2}+\frac{2^{2i-1}-3}{2^{2i}}\Big]\subseteq F^i\quad\text{and}\\ &f^{2,i}(F^i)=\Big[\frac{1}{3}+\frac{2^{2i+1}-12}{3\cdot2^{2i}},\,\frac{1}{3}+\frac{2^{2i+1}-10}{3\cdot2^{2i}}\Big]\subseteq F^i,
		\end{align*}
		where $i\in \mathbb{N}_+$ (see Figure \ref{fig.b(i)}). Let $R_1$ be an IRS on $E_0$ defined as
		\begin{align*}
			R_1(\boldsymbol{x}): =
			\begin{cases}
				\{f^{1,i}(\boldsymbol{x}),f^{2,i}(\boldsymbol{x})\},  & \boldsymbol{x}\in F^i, \,\,i\in\mathbb{N}_+, \\
				\boldsymbol{x}/4, & \boldsymbol{x}\in E_0\backslash (\bigcup_iF^i\bigcup\{1\}),\\
				\{0,1\}, & \boldsymbol{x}=1.
			\end{cases}
		\end{align*}
		Then $H_1=(\bigcup_{i=1}F^i)\bigcup\{1\}$, and $r_1:=R_1|_{H_1}$ cannot be decomposed as a finite family of contractions.
	\end{exam}
	\begin{figure}[htbp]
		\centering
		\begin{tikzpicture}[scale=0.9]
			\draw[red,ultra thick](0,0)--(5,0);
			\draw[red,ultra thick](7.5,0)--(8.75,0);
			\draw[red,ultra thick](9.375,0)--(9.6875,0);
			\draw[red,ultra thick](9.84375,0)--(9.9921875,0);
			\draw[red,ultra thick](9.9609375,0)--(9.98046875,0);
			\draw[red,ultra thick](9.990234375,0)--(9.995117188,0);
			\draw[red,ultra thick](9.997558594,0)--(9.998779297,0);
			
			\draw[blue,ultra thick](5,0)--(7.5,0);
			\draw[blue,ultra thick](8.75,0)--(9.375,0);
			\draw[blue,ultra thick](9.6875,0)--(9.84375,0);
			\draw[blue,ultra thick](9.9921875,0)--(9.9609375,0);
			\draw[blue,ultra thick](9.98046875,0)--(9.990234375,0);
			\draw[blue,ultra thick](9.995117188,0)--(9.997558594,0);
			
			\draw[red,ultra thick](0,-1.5)--(2.5,-1.5);
			\draw[red,ultra thick](0,-2.1)--(1.67,-2.1);
			\draw[red,ultra thick](7.5,-1.5)--(8.125,-1.5);
			\draw[red,ultra thick](7.5,-2.1)--(7.9166666667,-2.1);
			\draw[red,ultra thick](9.375,-1.5)--(9.53125,-1.5);
			\draw[red,ultra thick](9.375,-2.1)--(9.47916666667,-2.1);
			\draw[red,ultra thick](10,-1.5)--(10.05,-1.5);
			
			\draw[blue,ultra thick](1.25,-2.75)--(1.875,-2.75);
			\draw[blue,ultra thick](2.1875,-2.75)--(2.34375,-2.75);
			\draw[blue,ultra thick](2.421875,-2.75)--(2.4609375,-2.75);
			\draw[blue,ultra thick](2.498046875,-2.75)--(2.490234375,-2.75);
			\draw[blue,ultra thick](2.49117188,-2.75)--(2.497558594,-2.75);
			\draw[blue,ultra thick](2.498779297,-2.75)--(2.499389648,-2.75);

			\draw[blue](5,0)circle(.07);
			\draw[blue](7.5,0)circle(.07);
			\draw[blue](8.75,0)circle(.07);
			\draw[blue](9.375,0)circle(.07);
			\draw[blue](9.6875,0)circle(.07);
			\draw[blue](9.84375,0)circle(.07);
			
			\draw[blue](1.25,-2.75)circle(.07);
			\draw[blue](1.875,-2.75)circle(.07);
			\draw[blue](2.1875,-2.75)circle(.07);
			\draw[blue](2.34375,-2.75)circle(.07);
			\draw[fill=black](10,0.)circle(.05);
			\draw[fill=black](10,-1.5)circle(.03);
			%	\draw[fill=black](10,-2.1)circle(.05);
			\draw[fill=black](0.0,-1.5)circle(.03);
			%	\draw[fill=black](0.0,-2.1)circle(.05);

			\draw[black] (-1,0) node[right]{$E_0$};
			\draw[black] (-1,-1.8) node[right]{$E_1$};
			\draw[black] (2.5,0.3) node[right]{$F^1$};
			\draw[black] (7.9,0.3) node[right]{$F^2$};
			\draw[black] (9.3,0.3) node[right]{$F^3$};
			\draw[black] (0.8,-1.25) node[right][scale=0.8]{$f^{1,1}(F^1)$};
			\draw[black] (0.3,-1.85) node[right][scale=0.8]{$f^{2,1}(F^1)$};
			\draw[black] (1.1,-2.5) node[right][scale=0.75]{$R_1(E_0\backslash H_1)$};
			\draw[black] (-0.2,-0.25) node[right]{$0$};
			\draw[black] (9.8,-0.25) node[right]{$1$};	
		\end{tikzpicture}
		\caption{The sets $E_1=R_1(E_0)$ and $H_1=(\bigcup_{i=1}F^i)\bigcup\{1\}$ in Example \ref{decomposed}.}\label{fig.b(i)}
	\end{figure}
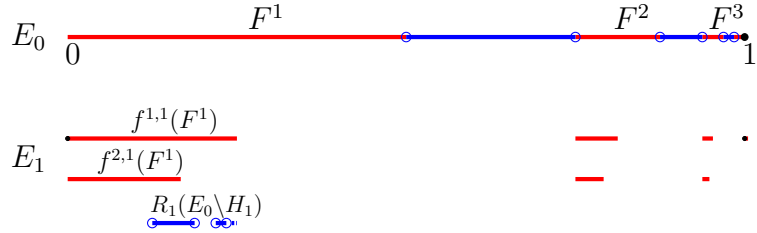
	\begin{proof} By the definition of $R_1$, we have $H_1=(\bigcup_{i=1}F^i)\bigcup\{1\}$. We prove the following claim.
		
		\noindent{\em Claim 1. For any $i,$ $j\in\mathbb{N}_+$, let $f:F^i\bigcup F^j\rightarrow f(F^i\bigcup F^j)$ be a function decomposed from $R_1$. Then $f$ is not contractive.}
		To prove this claim, we let $x_1:=1-(1/2^{2i-2})\in F^i$ and $x_2:=1-(1/2^{2j-2})\in F^j$. Then
		$$f(x_1)=1-\frac{1}{2^{2i-2}}\qquad\text{and}\qquad f(x_2)=1-\frac{1}{2^{2j-2}}.$$
		Note that
		$$d(f(x_1),f(x_2))=\Big|\frac{1}{2^{2i-2}}-\frac{1}{2^{2j-2}}\Big|=d(x_1,x_2).$$
		Thus $f$ is not a contraction. 
		Next, suppose that $r_1=R_1|_{H_1}$ can be decomposed as a finite family of contractions $\{h_1^{l,s}:H_1^s\to h_1^{l,s}(H_1^s)\}_{l=1,s=1}^{n,m}$. Then there would exist at least one $H_1^s$ such that
		\begin{align*}
			\bigcup_{i=q}^\infty F^i\subseteq H_{1}^s\,\,\text{ for some }\,\,\,q\geq 1,\qquad\text{and}\qquad h_1^{l,s}\Big|_{\big(\bigcup_{i=q}^\infty F^i\big)}\,\,\text{is contractive.}
		\end{align*}This contradicts Claim 1 and proves that $r_1$ cannot be decomposed as a finite family of contractions.
	\end{proof}
	
	The following example shows that condition (b)(ii) in Theorem \ref{thm:exi} need not be satisfied.
	\begin{exam}\label{notgood}
		Let $E_0:=[0,1]$. For any $i\in \mathbb{N}_+$, let
		\begin{align*}
			F^i:=\Big[1-\frac{1}{2^{i-1}},1-\frac{3}{5\cdot2^{i-1}}\Big],\quad
			I^i:=\Big(1-\frac{3}{5\cdot2^{i-1}},1-\frac{1}{2^{i}}\Big),
		\end{align*}
		and let
		$$f(\boldsymbol{x}):=\frac{5}{8}\boldsymbol{x}+\frac{3}{8}+\frac{3}{5\cdot2^{i+2}},\qquad\boldsymbol{x}\in F^i.$$
		Note that
		$$f(F^i)=\Big[1-\frac{11}{5\cdot2^{i+1}},1-\frac{3}{5\cdot2^{i}}\Big]\subseteq I^i\bigcup F^{i+1}$$
		(see Figure \ref{fig.b(ii)}). Let $R_1$ be an IRS on $E_0$ defined as
		\begin{align*}
			R_1(\boldsymbol{x}) &:=
			\begin{cases}
				\big\{f(\boldsymbol{x}), \boldsymbol{x}/20\big\},  &\boldsymbol{x}\in F^i,\,\,i\in \mathbb{N}_+, \\
				\boldsymbol{x}/20,& \boldsymbol{x}\in \bigcup_{i=1}^\infty I^i,\\
				\{1/20,1\}, & \boldsymbol{x}=1.
			\end{cases}
		\end{align*}
		Let $E_1:=R_1(E_0)$.
		Then there does not exist a good partition of $E_1$.
	\end{exam}
  	\begin{figure}[htbp]
	\centering
	\begin{tikzpicture}[scale=1]
		\draw[red,ultra thick](0,0.2)--(4,0.2);
		\draw[red,ultra thick](5,0.2)--(7,0.2);
		\draw[red,ultra thick](7.5,0.2)--(8.5,0.2);
		\draw[red,ultra thick](8.75,0.2)--(9.25,0.2);
		\draw[red,ultra thick](9.375,0.2)--(9.625,0.2);
		\draw[red,ultra thick](9.6875,0.2)--(9.8125,0.2);
		\draw[red,ultra thick](9.84375,0.2)--(9.90625,0.2);
		\draw[red,ultra thick](9.921875,0.2)--(9.953125,0.2);
		\draw[red,ultra thick](9.9609375,0.2)--(9.976525,0.2);
		\draw[red,ultra thick](9.9609375,0.2)--(9.976525,0.2);
		\draw[red,ultra thick](9.98046875,0.2)--(9.98828125,0.2);
		\draw[blue,ultra thick](4,0.2)--(5,0.2);
		\draw[blue,ultra thick](7,0.2)--(7.5,0.2);
		\draw[blue,ultra thick](8.5,0.2)--(8.75,0.2);
		\draw[blue,ultra thick](9.25,0.2)--(9.375,0.2);
		\draw[blue,ultra thick](9.625,0.2)--(9.6875,0.2);
		\draw[blue,ultra thick](9.8125,0.2)--(9.84375,0.2);
		\draw[blue,ultra thick](9.90625,0.2)--(9.921875,0.2);
		\draw[blue,ultra thick](9.976525,0.2)--(9.9609375,0.2);
		\draw[blue,ultra thick](9.976525,0.2)--(9.98046875,0.2);
		\draw[red,ultra thick](4.5,-1.5)--(7,-1.5);
		\draw[red,ultra thick](7.25,-1.5)--(8.5,-1.5);
		\draw[red,ultra thick](8.625,-1.5)--(9.25,-1.5);
		\draw[red,ultra thick](9.3125,-1.5)--(9.625,-1.5);
		\draw[red,ultra thick](9.65625,-1.5)--(9.8125,-1.5);
		\draw[red,ultra thick](9.828125,-1.5)--(9.90625,-1.5);
		\draw[red,ultra thick](10,-1.5)--(10.05,-1.5);
		
		\draw[magenta,ultra thick](0,-1.5)--(0.2,-1.5);
		\draw[magenta,ultra thick](0.25,-1.5)--(0.35,-1.5);
		\draw[magenta,ultra thick](0.375,-1.5)--(0.425,-1.5);
		\draw[magenta,ultra thick](0.4375,-1.5)--(0.4625,-1.5);
		\draw[magenta,ultra thick](0.46875,-1.5)--(0.48125,-1.5);
		\draw[magenta,ultra thick](0.484375,-1.5)--(0.490625,-1.5);
		\draw[magenta,ultra thick](0.4921875,-1.5)--(0.4953125,-1.5);
		\draw[blue,ultra thick](0.2,-1.5)--(0.25,-1.5);
		\draw[blue,ultra thick](0.35,-1.5)--(0.375,-1.5);
		\draw[blue,ultra thick](0.425,-1.5)--(0.4375,-1.5);
		\draw[blue,ultra thick](0.4625,-1.5)--(0.46875,-1.5);
		\draw[blue,ultra thick](0.46875,-1.5)--(0.484375,-1.5);
		\draw[blue,ultra thick](0.490625,-1.5)--(0.4921875,-1.5);
		
		\draw[blue](4,0.2)circle(.07);
		\draw[blue](5,0.2)circle(.07);
		\draw[blue](7,0.2)circle(.07);
		\draw[blue](7.5,0.2)circle(.07);
		\draw[blue](8.5,0.2)circle(.07);
		\draw[blue](8.75,0.2)circle(.07);
		\draw[blue](9.25,0.2)circle(.07);
		\draw[blue](9.375,0.2)circle(.07);
		\draw[fill=black](10,0.2)circle(.05);
		\draw[fill=black](10,-1.5)circle(.05);
		\draw[fill=black](0.5,-1.5)circle(.05);
		
		\draw[black] (-1,0.1) node[right]{$E_0$};
		\draw[black] (-1,-1.5) node[right]{$E_1$};
		\draw[black] (2,0.45) node[right]{$F^1$};
		\draw[black] (5.8,0.45) node[right]{$F^2$};
		\draw[black] (7.6,0.45) node[right]{$F^3$};
		\draw[black] (4.2,0.45) node[right]{$I^1$};
		\draw[black] (7,0.45) node[right]{$I^2$};
		\draw[black] (8.35,0.45) node[right]{$I^3$};
		\draw[black] (-0.2,-0.1) node[right]{$0$};
		\draw[black] (-0.2,-1.78) node[right][scale=0.85]{$0$};
		\draw[black] (9.8,-0.1) node[right]{$1$};
		\draw[black] (9.8,-1.8) node[right]{$1$};
		\draw[black] (0.2,-1.8) node[right][scale=0.8]{$1/20$};

		\draw(5,-1.35)--(5,-1.65);
		\draw(5.4,-1.35)--(5.4,-1.65);
		\draw(5.72,-1.35)--(5.72,-1.65);
		\draw(5.976,-1.35)--(5.976,-1.65);
		\draw(6.1808,-1.35)--(6.1808,-1.65);
		\draw(6.34464,-1.35)--(6.34464,-1.65);
		\draw(6.475712,-1.35)--(6.475712,-1.65);
		\draw(6.5805696,-1.35)--(6.5805696,-1.65);
		\draw(6.66445568,-1.35)--(6.66445568,-1.65);
		\draw(6.731564544,-1.35)--(6.731564544,-1.65);
		\draw(6.798673408,-1.35)--(6.798673408,-1.65);
		\draw(6.84162308096,-1.35)--(6.84162308096,-1.65);
		\draw(6.87598281933,-1.35)--(6.87598281933,-1.65);
		\draw(6.90347061002,-1.35)--(6.90347061002,-1.65);
		\draw(6.92546084248,-1.35)--(6.92546084248,-1.65);
		\draw(6.95,-1.35)--(6.95,-1.65);
		\draw(6.97,-1.35)--(6.97,-1.65);
		\draw(6.98,-1.35)--(6.98,-1.65);
		\draw(6.99,-1.35)--(6.99,-1.65);
		
		\draw(8.75,-1.35)--(8.75,-1.65);
		\draw(9.375,-1.35)--(9.375,-1.65);
		
		\draw[black] (4.8,-1.2) node[right][scale=0.7]{$d_1^1$};
		\draw[black] (5.2,-1.2) node[right][scale=0.6]{$d_2^1$};
		\draw[black] (5.5,-1.2) node[right][scale=0.6]{$d_3^1$};
		\draw[black] (7.25,-1.2) node[right][scale=0.6]{$d_1^2$};
		\draw[black] (7.45,-1.2) node[right][scale=0.6]{$d_2^2$};
		\draw[black] (8.55,-1.2) node[right][scale=0.6]{$d_1^3$};
		\draw[black] (9.15,-1.2) node[right][scale=0.6]{$d_1^4$};
		\draw[black] (5,-1.9) node[right][scale=0.8]{$h_1^{1,1}(F^1)$};
		\draw[black] (7.3,-1.9) node[right][scale=0.8]{$h_1^{1,1}(F^2)$};
		%\draw[black] (-0.2,-1.35) node[right][scale=0.6]{$h_1^{2,1}(H_1^1)$};
		%	\draw[black] (0.5,-1.65) node[right][scale=0.6]{$h_1^{0,1}(J_1^1)$};
		
		\draw(7.5,-1.35)--(7.5,-1.65);
		\draw(7.66,-1.35)--(7.66,-1.65);
		\draw(7.788,-1.35)--(7.788,-1.65);
		\draw(7.89,-1.35)--(7.89,-1.65);
		\draw(7.97,-1.35)--(7.97,-1.65);
		\draw(8.03,-1.35)--(8.03,-1.65);
		\draw(8.1,-1.35)--(8.1,-1.65);
		\draw(8.15,-1.35)--(8.15,-1.65);
		\draw(8.19,-1.35)--(8.19,-1.65);
		\draw(8.23,-1.35)--(8.23,-1.65);
		\draw(8.26,-1.35)--(8.26,-1.65);
		\draw(8.28,-1.35)--(8.28,-1.65);
		\draw(8.3,-1.35)--(8.3,-1.65);
		\draw(8.32,-1.35)--(8.32,-1.65);
		\draw(8.34,-1.35)--(8.34,-1.65);
		\draw(8.36,-1.35)--(8.36,-1.65);
		\draw(8.38,-1.35)--(8.38,-1.65);
		\draw(8.4,-1.35)--(8.4,-1.65);
		\draw(8.41,-1.35)--(8.41,-1.65);
		\draw(8.43,-1.35)--(8.43,-1.65);
		\draw(8.45,-1.35)--(8.45,-1.65);
		\draw(8.47,-1.35)--(8.47,-1.65);
		\draw(8.48,-1.35)--(8.48,-1.65);
		\draw(8.49,-1.35)--(8.49,-1.65);
		
	\end{tikzpicture}
	\caption{The sets $H_1=(\bigcup_{i=1}F^i)\bigcup\{1\}$ and  $E_0\backslash H_1=\bigcup_{i=1}I^i$ in Example \ref{notgood}.  $\{d_j^i\}_{i,j\in \mathbb{N}_+}$ represents the set of division points in $E_1$. Note that for all $i\in\mathbb{N}_+$, $\#\{d_j^i:j\in \mathbb{N}_+\}=\infty$.} \label{fig.b(ii)}
\end{figure}
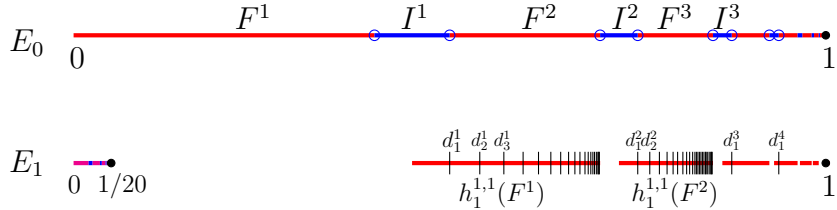
	\begin{proof}
		By the definition of $R_1$, we have
		$H_1^1=\bigcup_{i=1}^\infty F^i$, $H_1^2=\{1\},$ and $J_1^1=\bigcup_{i=1}^\infty I^i.$
		Thus, for any $i\in \mathbb{N}_+$,
		\begin{align*}
			&h_1^{1,1}(\boldsymbol{x})=f(\boldsymbol{x}), \,\,\boldsymbol{x}\in F^i;\qquad h_1^{2,1}(\boldsymbol{x})=\boldsymbol{x}/20,\,\, \boldsymbol{x}\in H_1^1;\\
			&h_1^{1,2}(\boldsymbol{x})=1/20,\,\,	\boldsymbol{x}=1;\qquad\,\,\,
			h_1^{2,2}(\boldsymbol{x})=1,\,\,	\boldsymbol{x}=1;\qquad\,\,\,\,
			h_1^{0,1}(\boldsymbol{x})=\boldsymbol{x}/20,\,\, \boldsymbol{x}\in J_1^1.
		\end{align*}
		Hence for $l$, $k\in\{1,2\}$, $h_1^{0,1}$ and $h_1^{l,k}$ are contractions. For any $i\in \mathbb{N}_+$, let
		\begin{align*}
			d_1^i:=1-\frac{1}{2^i}\quad\text{and}\quad \{d_j^i\}:=\big\{\big(h_1^{1,1}\big)^{-1}(d_{j-1}^{i+1})\big\}\bigcap h_1^{1,1}(F^i),\,j=2,3,\ldots
		\end{align*}
	be a family of division points in $h_1^{1,1}(F^i)$.
		Then for any $j\in\{2,3,\ldots\}$,
		\begin{align*}
			d_{j-1}^{i+1} \in h_1^{1,1}(F^{i+1})\qquad\text{and}\qquad F^{i+1}\subseteq h_1^{1,1}(F^i).
		\end{align*}
		Hence for any $i,j\in\mathbb{N}_+$,
		$\{d_j^i\}=\big\{\big(h_1^{1,1}\big)^{-1}(d_{j-1}^{i+1})\big\}\bigcap h_1^{1,1}(F^i)\neq\emptyset.$
		Next we show that for any $i\in \mathbb{N}_+$, $d_j^i$ is an increasing function of $j$. By the definition of $h_1^{1,1}$, we have $d_1^i<d_2^i$, where $i\in \mathbb{N}_+$, and $\big(h_1^{1,1}\big)^{-1}$ is an increasing function on $h_1^{1,1}(H_1^1)$. Assume that for any $i\in \mathbb{N}_+$ and $j=n$, we have $d_n^i<d_{n+1}^i$. Then
		\begin{align}\label{eq:increa}
			\big(h_1^{1,1}\big)^{-1}(d_n^i)<\big(h_1^{1,1}\big)^{-1}(d_{n+1}^i).
		\end{align}
		We let $j=n+1$ and let
		\begin{align*}
			\{d_{n+1}^i\}=\big\{\big(h_1^{1,1}\big)^{-1}(d_{n}^{i+1})\big\}\bigcap h_1^{1,1}(F^i)\quad\text{and}\quad \{d_{n+2}^i\} :=\big\{\big(h_1^{1,1}\big)^{-1}\big(d_{n+1}^{i+1})\big\}\bigcap h_1^{1,1}(F^i).
		\end{align*}
		Then by \eqref{eq:increa}, $d_{n+1}^i<d_{n+2}^i$.
		By induction, for any $i\in \mathbb{N}_+$, $d_j^i$ is an increasing function of $j$.
		Hence for any $i\in \mathbb{N}_+$, $\#\{d_j^i:j\in \mathbb{N}_+\}=\infty$. It follows that for any $i\in \mathbb{N}_+$,
		\begin{align*}
			h_1^{1,1}(F^i)=&\Big[1-\frac{11}{5\cdot2^{i+1}},d_1^i\Big]\bigcup[d_1^i,d_2^i]\bigcup[d_2^i,d_3^i]\bigcup\cdots\\
			=:&W_1^{1,i,1}\bigcup W_1^{1,i,2}\bigcup W_1^{1,i,3}\bigcup\cdots=\bigcup_{j=1}^\infty W_1^{1,i,j}.
		\end{align*}
		Note that for any $i\in \mathbb{N}_+$ and $j=2,3,\ldots$, we have
		$W_1^{1,i,1}\subseteq I^i$ and $W_1^{1,i,j}\subseteq F^{i+1}.$
		
		\noindent{\em Claim 1. For any $i\in \mathbb{N}_+$ and any $j\in \{2,3,\ldots\}$, let
			$$W_1^{i,t}:=W_1^{1,i,1}\bigcup W_1^{1,i,j} \qquad \text{and}\qquad g_1^{1,1,i,t}:=h_1^{1,1}|_{W_1^{i,t}}.$$
			Then $g_1^{1,1,i,t}$ is not contractive.} In fact, for any $x\in W_1^{1,i,1}$ and any $y\in W_1^{1,i,j}$, we have $d(x,y)\leq 1/4$, while $d(g_1^{1,1,i,t}(x),g_1^{1,1,i,t}(y))\geq 13/20$. This proves Claim 1.
		
		\noindent{\em Claim 2. Fix $i\in \mathbb{N}_+$. For any $p,q\in\mathbb{N}_+$ with $p<q$, let
			$$ W_1^{i,s}:=W_1^{1,i,p}\bigcup W_1^{1,i,q}\qquad\text{and}\qquad g_1^{1,1,i,s}:=h_1^{1,1}|_{W_1^{i,s}}.$$
			Then it is not possible for $g_1^{1,1,i,s}$ to be contractive and $W_1^{i,s}$ to be invariant under $g_1^{1,1,i,s}$.} In fact, suppose that $g_1^{1,1,i,s}$ is contractive and $W_1^{i,s}$ is invariant under $g_1^{1,1,i,s}$. By the definitions of $W_1^{1,i,p}$ and $W_1^{1,i,q}$, we have
		$$W_1^{1,i,p}\subseteq W_1^{1,i+1,p-1}\qquad\text{and}\qquad W_1^{1,i,q}\subseteq W_1^{1,i+1,q-1}.$$
		Hence $g_1^{1,1,i,s-1}$ is contractive and $W_1^{i+1,s-1}:=W_1^{1,i+1,p-1}\bigcup  W_1^{1,i+1,q-1}$ to be invariant under $g_1^{1,1,i,s-1}$. Since
		$$W_1^{1,i+1,p-1}\subseteq W_1^{1,i+2,p-2}\qquad\text{and}\qquad W_1^{1,i+1,q-1}\subseteq W_1^{1,i+2,q-2}.$$
		Hence $g_1^{1,1,i,s-2}$ is contractive and $W_1^{i+2,s-2}:=W_1^{1,i+2,p-2}\bigcup  W_1^{1,i+2,q-2}$ is invariant under $g_1^{1,1,i,s-2}$. Continue this process. We see that  $g_1^{1,1,i,s-p+1}$ is contractive and  $W_1^{i+p-1,s-p+1}:=W_1^{1,i+p-1,1}\bigcup  W_1^{1,i+p-1,q-p+1}$ is invariant under $g_1^{1,1,i,s-p+1}$. This contradicts Claim 1. This proves Claim 2. Let
		$$W_1^{2,1,1}:=h_1^{2,1}(H_1^1), \,\,W_1^{1,2,2}:=h_1^{1,2}(H_1^2),\,\, W_1^{2,2,2}:=h_1^{2,2}(H_1^2),\,\,\text{and}\,\,W_1^{0,1,1}:=\overline{h_1^{0,1}(J_1^1)}.$$
		Therefore we can extend $h_1^{0,1}$ from $J_1^1$ to $\overline{J_1^1}$, and let $\tilde{h}_1^{0,1}:\overline{J_1^1}\to W_1^{0,1,1}$ be defined as $\tilde{h}_1^{0,1}(\boldsymbol{x}):=\boldsymbol{x}/20$. Hence
		\begin{align*}
			E_1=&\tilde{h}_1^{0,1}(\overline{J_1^1})\bigcup\Big(\bigcup_{l=1}^2\bigcup_{k=1}^2 h_1^{l,k}(H_1^k)\Big)
			=W_1^{0,1,1}\bigcup\Big(\bigcup_{i,j\in \mathbb{N}_+} W_1^{1,i,j}\Big)\bigcup W_1^{2,1,1}\bigcup W_1^{1,2,2}\bigcup W_1^{2,2,2}.
		\end{align*}
		We assume that $\{W_1^{s,t}\}_{s=1,t=1}^{p,n}$ is a good partition of $E_1$ with respect to $$\{g_1^{l,s,t,k}\}_{l=0,k=1,s=1,t=1}^{2,2,p,n}, \,\,\text{where}\,\, g_1^{l,k,s,t}:=h_1^{l,k}|_{W_1^{s,t}}.$$ Then there exists at least one $W_1^{\alpha,\beta}$, where $\alpha\in\{1,\ldots,p\}$ and $\beta\in\{1,\ldots,n\}$, such that for any $i\in \mathbb{N}_+$,
		$$\bigcup_{j=q_1}^\infty W_1^{1,i,j}\subseteq W_1^{\alpha,\beta}\,\,\text{for some} \,\,\,q_1\geq 1.$$
		This contradicts Claim 2 and proves there does not exist a good partition of $E_1$.
	\end{proof}

	\begin{thm}\label{thm:exi2}
		Let $ X $ be a complete metric space and let $E_0\subseteq X $ be a nonempty compact set. Let $\{R_t\}_{t=1}^N$ be an IRS on $E_0$. For any $k\geq0$, let $E_{k+1}$  be defined as in \eqref{E_n}.  Assume $\{R_t\}_{t=1}^N$ satisfies the following conditions.
		\begin{enumerate}
			\item[(a)]There exists an integer $k_0\geq 1$ such that for any $t\in\Sigma$ and any $x\in E_{k_0-1}$,
			$$\#\{R_t(x)\}<\infty.$$
			\item[(b)] For any $t\in\Sigma$, if $H_t:=\{x\in E_{k_0-1}|2\leq\#\{R_t(x)\}<\infty\}\neq \emptyset$,  we require that the following conditions are satisfied.
			\begin{enumerate}
				\item[(i)] $r_t:=R_t|_{H_t}$  can be decomposed as a finite family of contractions  $\{h _t^{l,i}:H_t^i\to h _t^{l,i}(H_t^i)\}_{l=1,i=1}^{n_t,m_t}$,  where $\bigcup_{i=1}^{m_t}H_t^i=H_t$, and $\tilde{r}_t:=R_t|_{E_{k_0-1}\backslash H_t}$  can be decomposed as a finite family of contractions  $\{h _t^{0,i}:J_t^i\to h _t^{0,i}(J_t^i)\}_{i=1}^{s_t}$, where $\bigcup_{i=1}^{s_t}J_t^i=E_{k_0-1}\backslash H_t$.
				\item[(ii)]   There exist  $\alpha,\beta\in\Lambda_t$ or $\sigma,\tau\in\Delta_t$ such that for any $l\in \Pi_t$ and any $i\in\Lambda_t$,
				$$\overline{h _t^{l,i}(H_t^i)}\subseteq \overline{H_t^\alpha}\qquad \text{or}\qquad \overline{h _t^{l,i}(H_t^i)}\subseteq \overline{J_t^\sigma},$$
				and for any $i\in\Delta_t$,
				$$\overline{h _t^{0,i}(J_t^i)}\subseteq \overline{H_t^\beta}\qquad \text{or}\qquad \overline{h _t^{0,i}(J_t^i)}\subseteq \overline{J_t^\tau}.$$
			\end{enumerate}
		\end{enumerate}
		Then there exists a GIFS associated  to $\{R_t\}_{t=1}^N$ on $E_{k_0}$.
	\end{thm}
	\begin{proof}
		For any $t\in\Sigma$, let
		\begin{align*}
			& \underline{W}_t^{0,i}:=\overline{h_t^{0,i}(J_t^i)},\quad\text{where}\,\,i\in\Delta_t,\quad\text{and}\\
			&\underline{W}_t^{l,i}:=\overline{h_t^{l,i}(H_t^i)},\quad \underline{W}_t^{n_t+1,i}:=\overline{E_{k_0}\bigcap H_t^i},\quad\text{where}\,\,l\in\Pi_t \,\,\text{and}\,\, i\in\Lambda_t.
		\end{align*}
Fix $t\in\Sigma$. By (b)(ii),  we can rename the nonempty elements in  $\{\underline{W}_t^{l,i}\}_{t=1,s=1,i=1}^{N,n_t+1,m_t}$ and $\{W_t^{0,i}\}_{t=1,i=1}^{N,s_t}$ as $W_t^{s,i}$, where $s$ and $i$ satisfy the following conditions:
	\begin{enumerate}
		\item[(a)] for any $s\in\Psi_t^i$ and $i\in\Lambda_t$, $W_t^{s,i}\subseteq\overline{E_{k_0}\bigcap H_t^i}$;
		\item[(b)] for any  $s\in\{p_t^i+1,\ldots,p_t^i+h_t^i\}$ and $i\in\Delta_t$,  $W_t^{s,i}\subseteq\overline{J_t^i}$.
		\end{enumerate}
Note that
		$$\overline{E_{k_0}\bigcap H_t^i}=\bigcup_{s=1}^{p_t^i}W_t^{s,i} \qquad\text{and}\qquad \overline{J_t^i}=\bigcup_{s=p_t^i+1}^{p_t^i+h_t^i}W_t^{s,i}.$$
		For $t\in \Sigma$, $l\in\Pi_t$, $i\in\Lambda_t$, and $s\in \Psi_t^i$, let
		$g_t^{l,s,i}:=\tilde{h}_t^{l,i}|_{W_t^{s,i}}$. For  $i\in\Delta_t$ and $s\in\{p_t^i+1,\ldots, p_t^i+h_t^i\}$, let $g_t^{0,s,i}:=\tilde{h}_t^{0,i}|_{W_t^{s,i}}$. Here $\tilde{h}_t^{l,i}$ and $\tilde{h}_t^{0,i}$ are defined as in \eqref{eq:hl} and \eqref{eq:h0}, respectively. Then for any $W_t^{s,i}$, where $t\in\Sigma$, $i\in\Lambda_t$, and $s\in \Psi_t^i$, there exists some $W_t^{s_0,j}$,  where $s_0\in \Psi_t^j$ and $j\in\Lambda_t$, such that
		$$g_t^{l,s,i}(W_t^{s,i})\subseteq W_t^{s_0,j}.$$
		Similarly, for any $W_t^{s,i}$, where $t\in\Sigma$, $i\in\Delta_t$, and $s\in \{p_t^i+1,\ldots,p_t^i+h_t^i\}$, there exists some $W_t^{s_0,j}$, where $j\in\Delta_t$ and $s_0\in \Psi_t^j$, such that
		$$g_t^{0,s,i}(W_t^{s,i})\subseteq W_t^{s_0,j}.$$
		Hence $\{\{W_t^{s,i}\}_{t=1,s=1,i=1}^{N,p_t^i,m_t}, \{W_t^{s,i}\}_{t=1,s=p_t^i+1,i=1}^{N,p_t^i+h_t^i,s_t}\}$ is a good partition of $E_{k_0}$ with respect to   $\{\{g_t^{l,s,i}\}_{t=1,l=1,s=1,i=1}^{N,n_t,p_t^i,m_t},\{g_t^{0,s,i}\}_{t=1,s=p_t^i+1,i=1}^{N,p_t^i+h_t^i,s_t}\}$. By Theorem \ref{thm:exi},  there exists a GIFS associated  to $\{R_t\}_{t=1}^N$ on $E_{k_0}$.	
	\end{proof}

	\begin{defi}\label{simplified graph}
		Let $ X $ be a complete metric space, and let $G=(V,E)$ be a GIFS of contractions $\{f_e\}_{e\in E}$ on $ X $, where $V:=\{1,\ldots,p\}$ and $E$ is the set of all directed edges. Let $\{W_j\}_{j=1}^p$ be an invariant family  under $G$. We call $\widetilde{G}=(\widetilde{V},\widetilde{E})$ a \textit{simplified graph-directed iterated function system} associated to $G$, if $\widetilde{G}$ satisfies the following conditions.
		\begin{enumerate}
			\item[(a)] $\widetilde{E}\subseteq E$ and $\{\widetilde{W}_j\}_{j=1}^{\widetilde{p}} \subseteq \{W_j\}_{j=1}^{p}$, where $p\geq\widetilde{p}$.
			\item[(b)] Let $\{f_e\}_{e\in\widetilde{E}}$ be contractions associated to $\widetilde{G}$, and let $\{\widetilde{W}_j\}_{j=1}^{\widetilde{p}}$ be an invariant family under $\widetilde{G}$. Then for any $q\geqslant1$,
			\begin{align}\label{eq:min}
				\bigcup_{i,j=1}^p\bigcup_{\mathbf{e}\in E_q^{i,j}}f_\mathbf{e}(W_j)=\bigcup_{i,j=1}^{\widetilde{p}}\bigcup_{\mathbf{e}\in\widetilde{E}_q^{i,j}}f_\mathbf{e}(\widetilde{W}_j).
			\end{align}
		\end{enumerate}
	\end{defi}
	
	By Definition \ref{simplified graph}, we know that the attractor of $G$ is equal to the attractor of $\widetilde{G}$.  Note that the simplified GIFS is not unique.
	
	\begin{defi}\label{defi:min}
		We say that a simplified GIFS $\widehat{G}$ composed of $\big(\{\widehat{W}_j\}_{j=1}^{\widehat{p}},\{f_e\}_{e\in \widehat{E}}\big)$ is a \textit{minimal simplified graph-directed iterated function system} if among all simplified GIFSs 	$\widetilde{G}=(\widetilde{V},\widetilde{E})$ composed of $\big(\{\widetilde{W}_j\}_{j=1}^{\widetilde{p}},\{f_e\}_{e\in \widetilde{E}}\big)$, we have $\widehat{p}\leq \widetilde{p}$, and among all those simplified GIFSs with $ \widetilde{p}=\widehat{p}$, we have 	$\#\{f_e\}_{e\in \widehat{E}}\leq \#\{f_e\}_{e\in \widetilde{E}}.$ 
			\end{defi}
			
			\begin{comment}
			$\widetilde{G}=(\widetilde{V},\widetilde{E})$ of contractions $\{f_e\}_{e\in \widetilde{E}}$, where $\widetilde{V}=\{1,\ldots,\widetilde{p}\}$,    by first requiring that
	$\#\{\widehat{W}_j\}_{j=1}^{\widehat{p}}\leq \#\{\widetilde{W}_j\}_{j=1}^{\widetilde{p}},$
	and then  requiring} that
	$\#\{f_e\}_{e\in \widehat{E}}\leq \#\{f_e\}_{e\in \widetilde{E}}.$
	\end{comment}

	%We call the graph $\widehat{G}$ generated by the minimum number of $\{\widehat{W}_j\}_{j=1}^{\tilde{m}}$ satisfying the above conditions \textit{a minimal simplified graph}.
	
	\begin{prop}\label{prop:3.3}
		Assume that $\{R_t\}_{t=1}^N$ satisfies the conditions of Theorem \ref{thm:exi} and $G=(V,E)$ is a GIFS associated to $\{R_t\}_{t=1}^N$ guaranteed by Theorem \ref{thm:exi},  where $V=\{1,\ldots,p\}$ and  $G$ consists of contractions $\{f_e\}_{e\in E}$. Then there exists a minimal simplified GIFS $\widehat{G}=(\widehat{V},\widehat{E})$ associated to $G$.
	\end{prop}
	\begin{proof}
	Let $\{W_t^{s,i}\}_{t=1,s=1,i=1}^{N,p_t^i,m_t}$ and $\{W_t^{s,i}\}_{t=1,s=p_t^i+1,i=1}^{N,p_t^i+h_t^i,s_t}$ be defined as in the proof of Theorem \ref{thm:exi}. For fixed $t$, $s$, $i$, we write $W_j:=W_t^{s,i}$. Then $$\bigcup_{t=1}^N\Bigg(\bigcup_{i=1}^{m_t}\bigcup_{s=1}^{p_t^i}W_t^{s,i}\bigcup\Big(\bigcup_{i=1}^{s_t}\bigcup_{s=p_t^i+1}^{p_t^i+h_t^i}W_t^{s,i}\Big)\Bigg)=\bigcup_{j=1}^pW_j.$$
	 Fix $t\in\Sigma$. For any $W_t^{s,i}$, where $s\in \Psi_t^i$ and $i\in\Lambda_t$, if there exists $s_0\in\Psi_t^i$ with $s_0\neq s$ such that $W_t^{s,i} \subseteq W_t^{s_0,i}$, then we remove $W_t^{s,i}$. In particular, if $W_t^{s,i}=W_t^{s_0,i}$, then we remove one of them. If there are multiple elements in $\{W_t^{s,i}\}_{t=1,s=1,i=1}^{N,p_t^i,m_t}$ that are equal, then we keep one of them and remove the others.  We rename the remaining $\{W_t^{s,i}\}_{t=1,s=1,i=1}^{N,p_t^i,m_t}$ as $\widetilde{W}_t^{s,i}$, where $s\in \{1,\ldots,\widetilde{p}_t^i\}$ and $i\in\{1,\ldots,\widetilde{m}_t\}$. We use a similar method to keep the elements in the set $\{W_t^{s,i}\}_{t=1,s=p_t^i+1,i=1}^{N,p_t^i+h_t^i,s_t}$, and thus we rename the remaining $\{W_t^{s,i}\}_{t=1,s=p_t^i+1,i=1}^{N,p_t^i+h_t^i,s_t}$ as $\widetilde{W}_t^{s,i}$, where $s\in \{\widetilde{p}_t^i+1,\ldots,\widetilde{p}_t^i+\widetilde{h}_t^i\}$ and $i\in\{1,\ldots,\widetilde{s}_t\}$.
		Note that $$\bigcup_{i=1}^{m_t}\bigcup_{s=1}^{p_t^i}W_t^{s,i}=\bigcup_{i=1}^{\widetilde{m}_t}\bigcup_{s=1}^{\widetilde{p}_t^i}\widetilde{W}_t^{s,i}\qquad\text{and}\qquad \bigcup_{i=1}^{s_t}\bigcup_{s=p_t^i+1}^{p_t^i+h_t^i}W_t^{s,i}=\bigcup_{i=1}^{\widetilde{s}_t}\bigcup_{s=\widetilde{p}_t+1}^{\widetilde{p}_t^i+\widetilde{h}_t^i}\widetilde{W}_t^{s,i}.$$
		For any $t\in\Sigma$, we note that the number of elements removed from  $\{W_t^{s,i}\}_{t=1,s=1,i=1}^{N,p_t^i,m_t}$ and $\{W_t^{s,i}\}_{t=1,s=p_t^i+1,i=1}^{N,p_t^i+h_t^i,s_t}$ is equal to the number of elements removed from $\{W_j\}_{j=1}^p$.
		 We rename the remaining $\{W_j\}_{j=1}^p$ as $\widetilde{W}_j$, where $j\in\{1,\ldots,\widetilde{p}\}$. Note that $$E_{k_0}=\bigcup_{j=1}^pW_j=\bigcup_{j=1}^{\widetilde{p}}\widetilde{W}_j=\bigcup_{t=1}^N\Bigg(\bigcup_{i=1}^{\widetilde{m}_t}\bigcup_{s=1}^{\widetilde{p}_t^i}\widetilde{W}_t^{s,i}\bigcup\Big(\bigcup_{i=1}^{\widetilde{s}_t}\bigcup_{s=\widetilde{p}_t^i+1}^{\widetilde{p}_t^i+\widetilde{h}_t^i}\widetilde{W}_t^{s,i}\Big)\Bigg).$$
		Let $\widetilde{G}=(\widetilde{V},\widetilde{E})$ be a  GIFS of contractions $\{f_e\}_{e\in\widetilde{E}}$, where $\widetilde{V}=\{1,\ldots,\widetilde{p}\}$, $\widetilde{E}\subseteq E$, and $\{\widetilde{W}_j\}_{j=1}^{\widetilde{p}}$ is an invariant family under $\widetilde{G}$. For any $e\in E\backslash\widetilde{E}$, we have
		$f_e(W_t^{s,i})\subseteq W_t^{s',i}.$
		Note that there exist $\widetilde{W}_t^{s,i}$ and $\widetilde{W}_t^{s',i}$ such that $W_t^{s,i}\subseteq  \widetilde{W}_t^{s,i}$ and $W_t^{s',i}\subseteq\widetilde{W}_t^{s',i}.$ Hence $e\in\widetilde{E}$. Therefore, we have $E\backslash\widetilde{E}\subseteq\widetilde{E}$ and
		\begin{align*}
			\bigcup_{i,j=1}^p\bigcup_{e\in E^{i,j}}f_e(W_j)=\bigcup_{i,j=1}^{\widetilde{p}}\bigcup_{e\in\widetilde{E}^{i,j}}f_e(\widetilde{W}_j).
		\end{align*}
		By induction, for all $q\geq1$, we have
		\begin{align*}
			\bigcup_{i,j=1}^p\bigcup_{\mathbf{e}\in E_q^{i,j}}f_\mathbf{e}(W_j)=\bigcup_{i,j=1}^{\widetilde{p}}\bigcup_{\mathbf{e}\in \widetilde{E}_q^{i,j}}f_\mathbf{e}(\widetilde{W}_j).
		\end{align*}
	Therefore $\widetilde{G}=(\widetilde{V},\widetilde{E})$ is a simplified GIFS associated to $G$.
		
Among all simplified GIFSs that have been constructed by the above process, we first select the subcollection with the smallest number of vertices. Then among members of this subcollection, we further select the subfamily with the smallest number of contractions. Members of this subfamily are minimal simplified GIFSs associated to $G$, denoted $\widehat{G}=(\widehat{V},\widehat{E})$. 
	\end{proof}

	\section{Hausdorff dimension of graph self-similar sets without overlaps}\label{S:GOSC}
	\setcounter{equation}{0}
	In this section, we give the definition of the graph open set condition (GOSC) and prove Theorems \ref{thm:main1}--\ref{thm:main2}. Moreover, we give some examples of IRSs that satisfy the conditions of Theorem \ref{thm:main2}, and compute the Hausdorff dimension of the corresponding attractors.
	\subsection{Graph open set condition}
	\begin{defi}\label{defi:4.2}
		Let $ X $ be a complete metric space. Let $\{f_t\}_{t=1}^m$ be an IFS of contractions on $ X $. We say that $\{f_t\}_{t=1}^m$ satisfies the \textit{open set condition} (OSC) if there exists a nonempty bounded open set $U$ on $ X $ such that 
		$$\bigcup_{t=1}^mf_t(U)\subseteq U\qquad \text{and}\qquad f_{t_1}(U)\bigcap f_{t_2}(U)=\emptyset\quad\text{for}\,\, t_1\neq t_2.$$
	\end{defi}
	\begin{defi}\label{defi:4.1}
		Let $ X $ be a complete metric space.  Let $G=(V,E)$ be a GIFS of contractions $\{f_e\}_{e\in E}$ on $ X $. We say that $G$ satisfies the \textit{graph open set condition} (GOSC) if there exists a family $\{U_i\}_{i=1}^m$ of nonempty bounded open sets on $ X $ such that for all $ i\in\{1,\ldots,m\}$,
		\begin{enumerate}
			\item[(a)] $\bigcup_{e\in E^{i,j}}f_e(U_j)\subseteq U_i$;
			\item[(b)] $f_e(U_{j_1})\bigcap f_{e'}(U_{j_2})=\emptyset$, for all distinct $e\in E^{i,j_1}$ and $e'\in E^{i,j_2}$.
		\end{enumerate}
	\end{defi}
	
\begin{defi}\label{defi:rtgosc}
	Let $ X $ be a complete metric space and let $\{R_t\}_{t=1}^N$ be an IRS on a nonempty compact subset of $ X $. Assume that there exists a GIFS $G$ associated to $\{R_t\}_{t=1}^N$ and assume that  $G$ consists of contractions. If $G$ satisfies (GOSC), then we say that {\em $\{R_t\}_{t=1}^N$ satisfies (GOSC) with respect to $G$}. If $G$ does not satisfies (GOSC), then we say that  $\{R_t\}_{t=1}^N$ as {\em overlaps with respect to $G$.}
	\end{defi}
	Let $G=(V,E)$ be a GIFS of contractions $\{f_e\}_{e\in E}$ on $ X $.  For any $e\in E^{i,j}$, let $\rho_e$ be the contraction ratio of $f_e$.  Recall that an {\em incidence matrix} $A_\alpha$ associated with $G$  is an $m\times m$ matrix defined by
	\begin{align}\label{eq:matrix}
		A_\alpha= [\rho_e^\alpha]_{m\times m},
	\end{align}
	where for $i,j\in V$ and $e\notin E^{i,j}$, $\rho_e=0$.

	\begin{proof}[Proof of Theorem \ref{thm:main1}]
		By Proposition \ref{prop:3.1}, we know that $G$ and $\{R_t\}_{t=1}^N$ have the same attractor. If $G$ satisfies (GOSC), then $G$ satisfies (GFTC). The proof follows by using the results of \cite[Theorem 1.6]{Ngai-Xu_2023}; we omit the details.
	\end{proof}
	
	\begin{lem}\label{lem:4.1}
		Let $M$ be a complete $n$-dimensional smooth orientable Riemannian manifold with non-negative Ricci curvature. Let $K_M$ be the sectional curvature of $M$ and $K_M\leq b^2$. Let $\{V_i\}$ be a collection of disjoint open subsets of $M$ such that each $V_i$ contains a ball $B^M(p_1,a_1r)$ and is contained in a ball of radius $B^M(p_2,a_2r)$. Then any ball $B^M(p,r)$ intersects at most $C(n)(r+2a_2r)^n/(C(n,b,a_1r))$ of the $\overline{V_i}$, where $C(n)=\pi^{n/2}/\Gamma(1+n/2)$ is the volume of the unit ball in $\R^n$.
	\end{lem}
	\begin{proof}
	Let $B^M(p,r)\bigcap\overline{V}_i\neq\emptyset$. Then $V_i$ is contained in a ball concentric with $B^M(p,r+2a_2r)$. Let $q$ be the number of $V_i$ such that $B^M(p,r)\bigcap\overline{V}_i\neq\emptyset$. Then summing volumes, we have
		$$q { \rm Vol}_M(B^M(p_1,a_1r))\leq\sum_{B^M(p,r)\bigcap\overline{V}_i\neq\emptyset}{\rm Vol}_M(\overline{V}_i)\leq{\rm Vol}_M(B^M(p,r+2a_2r)).$$
		By the Bishop-Gromov inequality (see, e.g., \cite{Bishop-Crittenden_1964}), we have
		$${\rm Vol}_M(B^M(p,r+2a_2r))\leq C(n)(r+2a_2r)^n.$$
		Since $K_M\leq b^2$, we have ${ \rm Vol}_M(B^M(p_1,a_1r))\geq C(n,b,a_1r).$
		This completes the proof.
	\end{proof}

	\begin{proof}[Proof of Theorem \ref{thm:main2}]
		Combining Proposition \ref{prop:3.1} and Definition \ref{simplified graph}, we know that $G$ and $\widehat{G}$ have the same attractor. If $\widehat{G}$ satisfies (GOSC), then $\widehat{G}$ satisfies (GFTC). By using the results of \cite[Theorem 1.6]{Ngai-Xu_2023}, we can prove (a). As $M$ is locally Euclidean, the sectional curvatures and the Ricci curvatures  of $M$ are everywhere zero.
		By using the results of Lemma \ref{lem:4.1}, and a similar method as in \cite{Falconer_2003},  we can prove (b); we omit the details.
	\end{proof}
	{\begin{comment}
		\begin{rema}\label{thm:main3}
		Let $M$, $\{R_t\}_{t=1}^N$, $K$, and $G=(V,E)$ be defined as in Theorem \ref{thm:main1}. Let $\widehat{G}=(\widehat{V},\widehat{E})$ be a minimal simplified graph of contractive similitudes $\{f_t\}_{t=1}^m$ associated to $G$. Assume that $\#\widehat{V}=1$. If $\{f_t\}_{t=1}^m$ satisfies (OSC), then $\dim_H(K)$ is the unique number $\alpha$ satisfying
		$$\sum_{t=1}^m\rho_t^\alpha=1,$$
		where $\rho_t$ is the contraction ratio of $f_t$.
		\end{rema}
		\begin{proof}
		The proof follows by using the results of Lemma \ref{lem:4.1}, and a similar method as \cite{Falconer_2003}; we omit the details.
		\end{proof}
		\end{comment}
		\subsection{Examples}
		In this subsection, we provide three examples of IRSs satisfying Theorem \ref{thm:main2}, and compute the Hausdorff dimension of the associated attractors.
		
		\begin{exam}\label{exam:osc1}
			Let $\mathcal{C}^2:=\mathbb{S}^1\times\R^1=\big\{(\cos\theta,\sin\theta,z):\theta\in[-\pi,\pi],z\in[0,2\pi]\big\}$ be a cylindrical surface.
			Let $E_0:=\mathcal{C}^2$. For $r\in [0,\pi/2)$, let $\boldsymbol{x}:=(\cos\theta,\sin\theta,z)\in E_0$, $H:=\{(-1,0,z):z\in[0,2\pi]\}$, and $\{R_t\}_{t=1}^3$ be an IRS on $E_0$ defined as
			\begin{align*}
				&R_1(\boldsymbol{x}):=\left\{
				\begin{aligned}
					&(-\cos (\theta/2),-\sin (\theta/2),z/2+\pi/2),\quad \qquad \,\,&&\boldsymbol{x}\in E_0\backslash H,\\
					&\{(0,-1,z/2+\pi/2), (0,1,z/2+\pi/2)\}\quad &&\boldsymbol{x}\in H;\\
				\end{aligned}	
				\right.\\
				&R_2(\boldsymbol{x}):=\left\{
				\begin{aligned}
					&(\cos (\theta/2),\sin (\theta/2),z/2+\pi-r),\quad  &&\boldsymbol{x}\in E_0\backslash H,\\
					&\{(0,1,z/2+\pi-r),(0,-1,z/2+\pi-r) \}\quad &&\boldsymbol{x}\in H;\\
				\end{aligned}	
				\right.\\
				&R_3(\boldsymbol{x}):=\left\{
				\begin{aligned}
					&(\cos (\theta/2),\sin (\theta/2),z/2+r),\qquad\qquad \qquad\,\, &&\boldsymbol{x}\in E_0\backslash H,\\
					&\{(0,1,z/2+r),(0,-1,z/2+r)\} \qquad\qquad &&\boldsymbol{x}\in H.
				\end{aligned}	
				\right.
			\end{align*}
			Let $K$ be the associated attractor (see Figure \ref{fig:osc1}).  Then
			$$\dim_H(K)=\frac{\log3}{\log2}= 1.58496\ldots.$$
		\end{exam}
		
		\begin{figure}[H]
			\centering
			\mbox{\subfigure[]
			{	\includegraphics[scale=0.34]{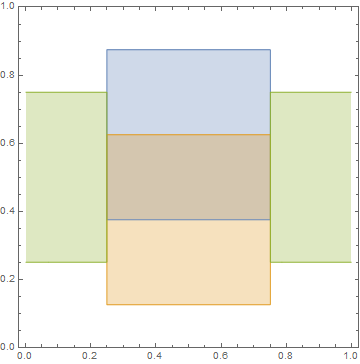}}
		}\quad
		\mbox{\subfigure[]
			{	\includegraphics[scale=0.34]{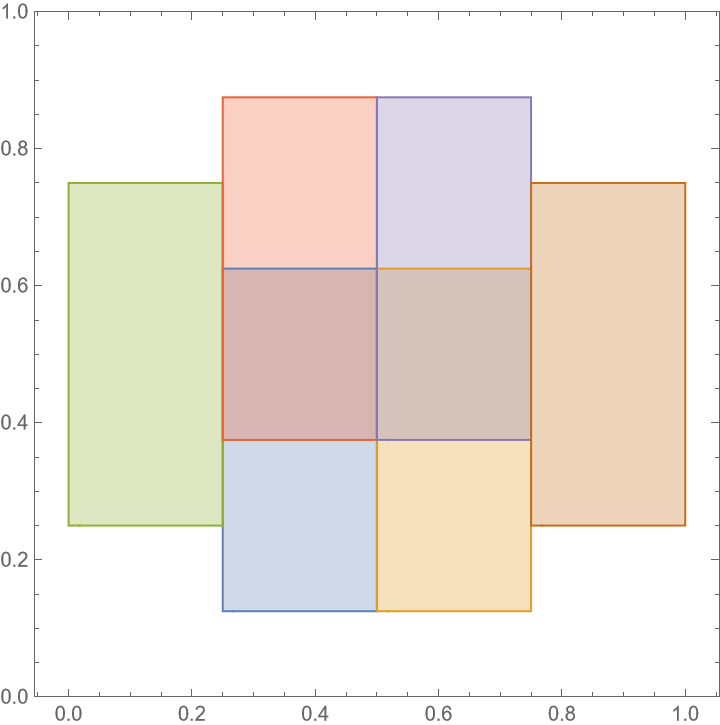}}
		}	\quad
	\mbox{\subfigure[]
	{	\includegraphics[scale=0.34]{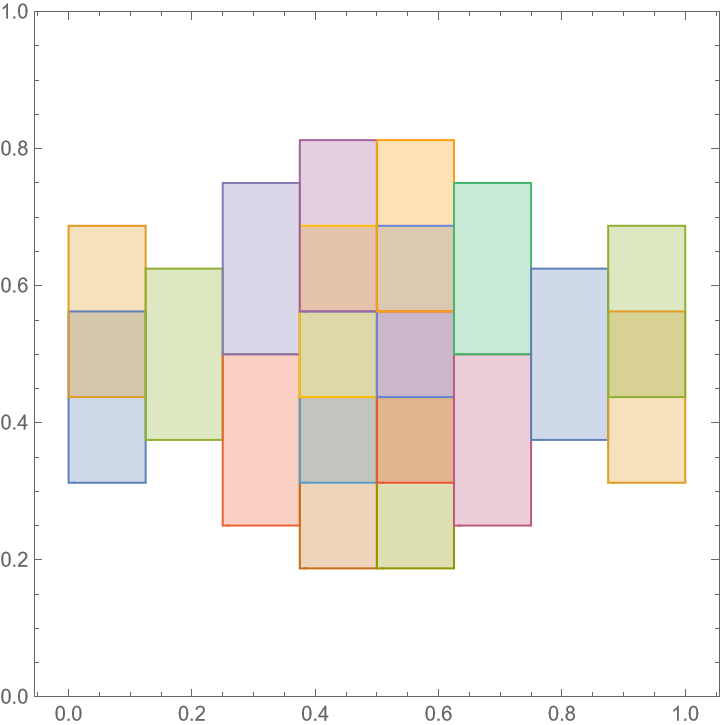}}
}	\\	
			\mbox{\subfigure[Front]
				{	\includegraphics[scale=0.24]{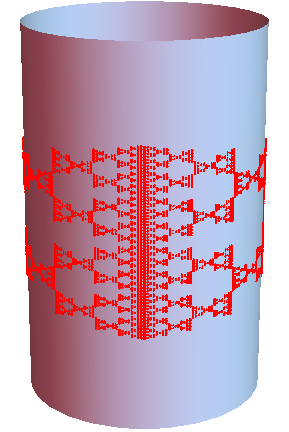}}
			}\qquad\qquad
			\mbox{\subfigure[Back]
				{	\includegraphics[scale=0.24]{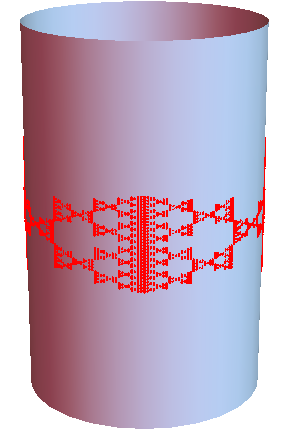}}
			}
			\caption{Figures for Example \ref{exam:osc1} with $r=\pi/4$. (a)--(c) are drawn on $\mathbb R^2$ and shrunk by $2\pi$. (a) The first iteration of  $E_0$ under $\{R_t\}_{t=1}^3$, where $R_1(E_0)$ consists of the left and right rectangles, $R_2(E_0)$ is the top square, and $R_3(E_0)$ is the bottom square. (b) Vertices of the GIFS associated to  $\{R_t\}_{t=1}^3$. (c) The first iteration of the vertices under the GIFS.  (d)--(e) The attractor of $\{R_t\}_{t=1}^3$.}
			\label{fig:osc1}
		\end{figure}

		\begin{proof}
			For any $t\in\{1,2,3\}$, by the definition of $R_t$, we have $H_t^1=H_t:=\{(-1,0,z):z\in[0,2\pi]\}$. Let $r_t:=R_t|_{H_t}$ and $r_t(H_t)=\bigcup_{l=1}^2h_t^{l,1}(H_t^1)$. Then for any $\boldsymbol{x}\in H_t^1$,
			\begin{align*}
				&h_1^{1,1}(\boldsymbol{x})=(0,1,z/2+\pi/2),\quad	&&h_1^{2,1}(\boldsymbol{x})=(0,-1,z/2+\pi/2),\\
				&h_2^{1,1}(\boldsymbol{x})=(0,-1,z/2+\pi-r),\quad	&&h_2^{2,1}(\boldsymbol{x})=(0,1,z/2+\pi-r),\\
				&h_3^{1,1}(\boldsymbol{x})=(0,-1,z/2+r),\quad&&	h_3^{2,1}(\boldsymbol{x})=(0,1,z/2+r).
			\end{align*}
			For any $t\in\{1,2,3\}$, let
			\begin{align*}
				&J_t^1:=\big\{(\cos\theta,\sin\theta,z):\theta\in(-\pi,0],z\in[0,2\pi]\big\}\quad \text{and}\\
				&J_t^2:=\big\{(\cos\theta,\sin\theta,z):\theta\in(0,\pi),z\in[0,2\pi]\big\}.
			\end{align*}
			Then $E_0\backslash H_t=\bigcup_{i=1}^2J_t^i$. Let $\tilde{r}_t:=R_t|_{E_0\backslash H_t}$ and $\tilde{r}_t(E_0\backslash H_t)=\bigcup_{i=1}^2h_t^{0,i}(J_t^i)$. Then for any $\boldsymbol{x} \in J_t^i$, where $i\in\{1,2\}$,
			\begin{align*}
				&h_1^{0,i}(\boldsymbol{x})=(-\cos (\theta/2),-\sin (\theta/2),z/2+\pi/2),\\
				&h_2^{0,i}(\boldsymbol{x})=(\cos (\theta/2),\sin (\theta/2),z/2+\pi-r),\\
				&h_3^{0,i}(\boldsymbol{x})=(\cos (\theta/2),\sin (\theta/2),z/2+r).
			\end{align*}
			Hence for any $t\in \{1,2,3\}$, $h_t^{l,1}$ are contractions, where $l\in\{1,2\}$, and $h_t^{0,i}$ are  contractions, where $i\in\{1,2\}$. Moreover, for $t=1$, we have 
			\begin{align*}
			\overline{h_t^{1,1}(H_t)}\subseteq \overline{J_t^2},\qquad \overline{h_t^{2,1}(H_t)}\subseteq \overline{J_t^1},\qquad \overline{h_t^{0,1}(J_t^1)}\subseteq \overline{J_t^2},\qquad \overline{h_t^{0,2}(J_t^2)}\subseteq \overline{J_t^1};
			\end{align*}
			for any $t\in\{2,3\}$, we have
				\begin{align*}
			 \overline{h_t^{1,1}(H_t)}\subseteq \overline{J_t^1},\qquad \overline{h_t^{2,1}(H_t)}\subseteq \overline{J_t^2}, \qquad  \overline{h_t^{0,i}(J_t^i)}\subseteq \overline{J_t^i},\,\,\text{where} \,\,i\in \{1,2\}.
			\end{align*} 
			Hence  $\{R_t\}_{t=1}^3$ satisfies the conditions of Theorem \ref{thm:exi2}.
			By Theorem \ref{thm:exi2} and Proposition \ref{prop:3.3}, we can find a minimal simplified GIFS $\widehat{G}=(\widehat{V},\widehat{E})$ with $\widehat{V}=\{1,\ldots,6\}$ and $\widehat{E}=\{e_1,\ldots,e_{18}\}$. The invariant family $\{\widehat{W}_i\}_{i=1}^6$ and the associated similitudes $\{f_e\}_{e\in\widehat{E}}$ are defined as
			\begin{align*}
				&\widehat{W}_1:=\{(\cos\theta,\sin\theta,z):\theta\in [-\pi,-\pi/2],\,\, z\in[\pi/2,3\pi/2]\},\\	
				&\widehat{W}_2:=\{(\cos\theta,\sin\theta,z):\theta\in [-\pi/2,0],\,\, z\in[\pi-r,2\pi-r]\},\\
				&\widehat{W}_3:=\{(\cos\theta,\sin\theta,z):\theta\in [-\pi/2,0],\,\, z\in[r,\pi+r]\},\\
				&\widehat{W}_4:=\{(\cos\theta,\sin\theta,z):\theta\in [0,\pi/2],\,\, z\in[\pi-r,2\pi-r]\},\\
				&\widehat{W}_5:=\{(\cos\theta,\sin\theta,z):\theta\in [0,\pi/2],\,\, z\in[r,\pi+r]\},\\
				&\widehat{W}_6:=\{(\cos\theta,\sin\theta,z):\theta\in [\pi/2,\pi],\,\, z\in[\pi/2,3\pi/2]\},
			\end{align*}
			while $\widehat{E}^{i,j}$, $i,j\in\{1,\ldots,6\}$,  and the associated similitudes $\{f_e\}_{e\in\widehat{E}}$ are defined as
			\begin{alignat*}{6}
&e_1\in \widehat{E}^{1,4},\qquad &e_2\in \widehat{E}^{1,5},\qquad &e_3\in \widehat{E}^{1,6},\qquad &e_4\in \widehat{E}^{2,1},\qquad &e_5\in \widehat{E}^{2,2},\qquad&e_6\in \widehat{E}^{2,3},\\
				&e_7\in\widehat{E}^{3,1},\qquad &e_8\in \widehat{E}^{3,2}, \qquad &e_9\in \widehat{E}^{3,3}, \qquad&e_{10}\in \widehat{E}^{4,4},\qquad	&e_{11}\in \widehat{E}^{4,5}, \qquad&e_{12}\in \widehat{E}^{4,6},\\
				&e_{13}\in \widehat{E}^{5,4}, \qquad&e_{14}\in \widehat{E}^{5,5},\qquad &e_{15}\in \widehat{E}^{5,6},\qquad
				&e_{16}\in \widehat{E}^{6,1},\qquad &e_{17}\in \widehat{E}^{6,2},\qquad &e_{18}\in \widehat{E}^{6,3},
			\end{alignat*}
			and
			\begin{align*}
				&f_{e_{1}}:=\widetilde{h}_1^{0,2}|_{\widehat{W}_4}, \quad\,\,	f_{e_{2}}:=\widetilde{h}_1^{0,2}|_{\widehat{W}_5},  \,\,\quad\,\,\,	f_{e_{3}}:=\widetilde{h}_1^{0,2}|_{\widehat{W}_6},\quad\,\,\,\,\, f_{e_{4}}:=\widetilde{h}_2^{0,1}|_{\widehat{W}_1}, \quad  \,\,\, f_{e_{5}}:=\widetilde{h}_2^{0,1}|_{\widehat{W}_2},\\  		&f_{e_{6}}:=\widetilde{h}_2^{0,1}|_{\widehat{W}_3}, \quad\,\,	f_{e_{7}}:=\widetilde{h}_3^{0,1}|_{\widehat{W}_1},  \quad\,\,\,\,\,	f_{e_{8}}:=\widetilde{h}_3^{0,1}|_{\widehat{W}_2},\quad\,\,\,\,\, f_{e_{9}}:=\widetilde{h}_3^{0,1}|_{\widehat{W}_3}, \quad  \,\,\, f_{e_{10}}:=\widetilde{h}_2^{0,2}|_{\widehat{W}_4},\\
				&f_{e_{11}}:=\widetilde{h}_2^{0,2}|_{\widehat{W}_5}, \quad\,\,	f_{e_{12}}:=\widetilde{h}_2^{0,2}|_{\widehat{W}_6},  \quad\,\,	f_{e_{13}}:=\widetilde{h}_3^{0,2}|_{\widehat{W}_4},\quad\,\, f_{e_{14}}:=\widetilde{h}_3^{0,2}|_{\widehat{W}_5}, \quad  \,\, f_{e_{15}}:=\widetilde{h}_3^{0,2}|_{\widehat{W}_6},\\
				&f_{e_{16}}:=\widetilde{h}_1^{0,1}|_{\widehat{W}_1}, \quad\,\,	f_{e_{17}}:=\widetilde{h}_1^{0,1}|_{\widehat{W}_2},  \quad\,\,	f_{e_{18}}:=\widetilde{h}_1^{0,1}|_{\widehat{W}_3},
			\end{align*}
			where $\tilde{h}_t^{l,i}$ and $\tilde{h}_t^{0,i}$ are defined as in \eqref{eq:hl} and \eqref{eq:h0}, respectively.
			Note that $\widehat{G}$ is strong connected.
			Let
			\begin{align*}
				\underline{W}_1&:=\{(\cos\theta,\sin\theta,z):\theta\in (-\pi,-\pi/2),\,\, z\in(\pi/2+r,3\pi/2-r)\},\\		
				\underline{W}_2&:=\{(\cos\theta,\sin\theta,z):\theta\in (-\pi/2,0),\,\, z\in(3\pi/2-2r,2\pi-2r)\},\\
				\underline{W}_3&:=\{(\cos\theta,\sin\theta,z):\theta\in (-\pi/2,0),\,\, z\in(2r,2r+\pi/2)\},\\
				\underline{W}_4&:=\{(\cos\theta,\sin\theta,z):\theta\in (0,\pi/2),\,\, z\in(3\pi/2-2r,2\pi-2r)\},\\
				\underline{W}_5&:=\{(\cos\theta,\sin\theta,z):\theta\in (0,\pi/2),\,\, z\in(2r,2r+\pi/2)\},\\
				\underline{W}_6&:=\{(\cos\theta,\sin\theta,z):\theta\in (\pi/2,\pi),\,\, z\in(\pi/2+r,3\pi/2-r)\}.
			\end{align*}
			Let
			$$\widetilde{K}_i=\bigcup_{j=1}^6\bigcup_{e\in \widehat{E}^{i,j}}f_e(\widetilde{K}_j).$$
			Then for $i\in\{1,\ldots,6\}$,  $U_i:=\underline{W}_i\backslash \widetilde{K}_i$ is an open set. For all $i\in\{1,\ldots,6\}$, $\{f_e\}_{e\in \widehat{E}}$ satisfies
			\begin{align*}
				\bigcup_{e\in \widehat{E}^{i,j}}f_e(U_j)\subseteq U_i\quad\text{and}\quad f_{e_1}(U_{j_1})\bigcap f_{e_2}(U_{j_1})=\emptyset,\,\,\text{for}\,\,e_1\in \widehat{E}^{i,j_1}\,\,\text{and}\,\,e_2\in \widehat{E}^{i,j_2}.
			\end{align*}
			Hence $\widehat{G}$ satisfies (GOSC).
			The weighted incidence matrix associated to $\widehat{G}$  is
			$$A_\alpha=\Big(\frac{1}{2}\Big)^\alpha\footnotesize{\begin{bmatrix}
					\begin{array}{cccccc}
						0  &0   &0	&1  &1   &1	   \\
						1  &1   &1 	&0  &0   &0	  \\
						1  &1   &1 	&0  &0   &0	  \\
						0  &0   &0	&1  &1   &1	   \\
						0  &0   &0	&1  &1   &1	   \\
						1  &1   &1 	&0  &0   &0	  \\
					\end{array}
				\end{bmatrix}.}$$
			The spectral radius of $A_\alpha$ is $3(1/2)^\alpha$. Therefore,
			$$\dim_H(K)=\frac{\log 3}{\log 2}= 1.58496\ldots.$$
		\end{proof}
		
		\begin{exam}\label{R_1}
			Let $E_0:=[0,1]$, and let $\{R_t\}_{t=1}^2$ be an IRS on $E_0$ defined as
			\begin{align*}
				R_1(\boldsymbol{x}) &:= (1/2)\boldsymbol{x}+1/2;\\
				R_2(\boldsymbol{x}) &:=
				\begin{cases}
					\{(1/2)\boldsymbol{x},(1/2)\boldsymbol{x}+3/8\}	,  &\boldsymbol{x}\in[0,1/2], \\
					(1/2)\boldsymbol{x}, & \boldsymbol{x}\in (1/2,1].
				\end{cases}
			\end{align*}
		 Let $K$ be the associated attractor (see Figure \ref{fig.2}).  Then
			$$\dim_H(K)=\frac{\log \big((1+\sqrt 5)/2\big)}{\log2}= 0.694242\ldots.$$
		\end{exam}
		The proof of this example is similar to that of Example \ref{exam:osc1}; and is omitted.
		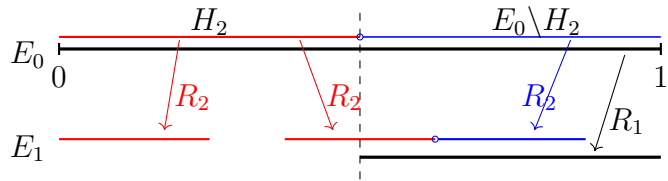
\begin{figure}[htbp]
			\centering
			\begin{tikzpicture}[scale=0.8]
				\draw[black,very thick](0,0)--(10,0);
				\draw[red,thick](0,0.2)--(5,0.2);
				\draw[blue,semithick](5.05,0.2)--(10,0.2);
				\draw[red,thick](0,-1.5)--(2.5,-1.5);
				\draw[red,thick](3.75,-1.5)--(6.25,-1.5);
				%	\draw[purple,thick](0,-1.7)--(2.5,-1.7);
				%	\draw[purple,thick](3.75,-1.7)--(5,-1.7);
				\draw[blue,thick](6.3,-1.5)--(8.75,-1.5);
				\draw[black,very thick](5,-1.8)--(10,-1.8);
				\draw[black,thin,dashed](5,0.6)--(5,-2.3);
				\draw [thick] (0,-.1) node[below]{0} -- (0,0.1);
				\draw [thick] (10,-.1) node[below]{1}-- (10,0.1);
				\draw[blue](5,0.2)circle(.05);
				\draw[blue](6.25,-1.5)circle(.05);
				%	\draw[red] (1,-1.3) node[right]{\footnotesize{$W_2^1$}};
				%	\draw[red] (4.7,-1.3) node[right]{\footnotesize{$W_2^2$}};
				%	\draw[blue] (7,-1.3) node[right]{\footnotesize{$W_2^0$}};
				%	\draw[black] (7.7,-2.05) node[right]{\footnotesize{$W_1^0$}};
				\draw[black] (-1,-0.1) node[right]{$E_0$};
				\draw[black] (-1,-1.65) node[right]{$E_1$};
				\draw[black] (2,0.45) node[right]{$H_2$};
				\draw[black] (7,0.45) node[right]{$E_0\backslash H_2$};
				\draw [red,->] (2,0.15) -- (1.75,-1.35);
				\draw [red,->] (4,0.15) -- (4.55,-1.35);
				\draw [blue,->] (8.5,0.15) -- (7.9,-1.35);
					\draw [black,->] (9.4,-0.1) -- (8.9,-1.7);
				\draw[red] (1.73,-0.8) node[right]{$R_2$};
				\draw[red] (4.25,-0.8) node[right]{$R_2$};
				\draw[blue] (7.5,-0.8) node[right]{$R_2$};
				\draw[black] (8.95,-1.2) node[right]{$R_1$};
				
			\end{tikzpicture}
			\caption{The sets $H_2=[0,1/2]$, $E_0\backslash H_2=(1/2,1]$, and $E_1=\bigcup_{t=1}^2R_t(E_0)$ in Example \ref{R_1}.}\label{fig.2}
		\end{figure}

		\begin{exam}\label{exam:osctri}
			Let $\mathcal{C}^2:=\big\{(\cos\theta,\sin\theta,z):\theta\in[-\pi,\pi],z\in [0,2\pi]\big\}$ be a cylindrical surface. Let
			\begin{align*}
				E_0^1:=&\big\{(\cos\theta,\sin\theta,z):\theta\in[-\pi,0],z\in [0,\sqrt{3}\,\theta+\sqrt{3}\,\pi]\big\},\\ E_0^2:=&\big\{(\cos\theta,\sin\theta,z):\theta\in[0,\pi],z\in [0,-\sqrt{3}\,\theta+\sqrt{3}\,\pi]\big\},
			\end{align*}
			and $E_0:=E_0^1\bigcup E_0^2$.
			Let $\boldsymbol{x}:=(\cos\theta,\sin\theta,z)\in E_0$ and let $\{R_t\}_{t=1}^3$ be an IRS on $E_0$ defined as
			\begin{align*}
				&R_1(\boldsymbol{x}):=\left\{
				\begin{aligned}
					&(\cos (\theta/2-\pi/2),\sin (\theta/2-\pi/2),z/2),\quad  &&\boldsymbol{x}\in E_0\backslash \{(-1,0,0)\},\\
					&\{(1,0,0),(-1,0,0)\} \quad &&\boldsymbol{x}=(-1,0,0);\\
				\end{aligned}	
				\right.\\
				&R_2(\boldsymbol{x}):=\left\{
				\begin{aligned}
					&(\cos (\theta/2+\pi/2),\sin (\theta/2+\pi/2),z/2),\quad &&\boldsymbol{x}\in E_0\backslash \{(-1,0,0)\},\\
					&\{(-1,0,0),(1,0,0)\}, \quad &&\boldsymbol{x}=(-1,0,0);\\
				\end{aligned}	
				\right.\\
				&R_3(\boldsymbol{x}):=\left\{
				\begin{aligned}
					&(\cos (\theta/2),\sin (\theta/2),z/2+\sqrt{3}\,\pi/2),\qquad\,\,  &&\boldsymbol{x}\in E_0\backslash \{(-1,0,0)\},\\
					&\{(0,1,\sqrt{3}\,\pi/2),(0,-1,\sqrt{3}\,\pi/2)\} \qquad &&\boldsymbol{x}=(-1,0,0).
				\end{aligned}	
				\right.
			\end{align*}
			Let $K$ be the associated attractor (see Figure \ref{fig:osctri}).
			Then
			$$\dim_H(K)=\frac{\log3}{\log2}= 1.58496\ldots.$$
		\end{exam}
		
		The proof of this example is similar to that of Example \ref{exam:osc1};  and is again omitted.
		\begin{figure}[H]
			\centering
			\mbox{\subfigure[Front]
				{	\includegraphics[scale=0.24]{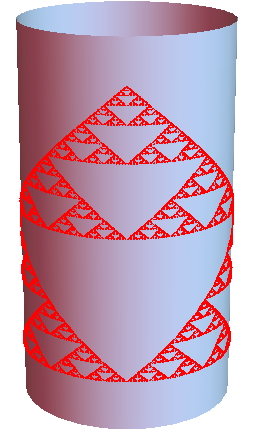}}
			}\qquad\qquad
			\mbox{\subfigure[Back]
				{	\includegraphics[scale=0.24]{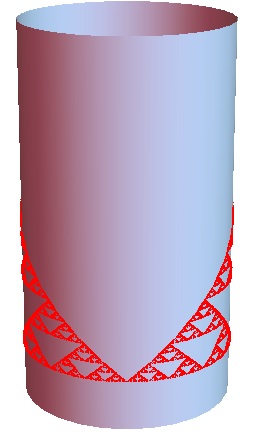}}
			}
			\caption{The attractor of $\{R_t\}_{t=1}^3$ in Example \ref{exam:osctri}.}
			\label{fig:osctri}
		\end{figure}

		\section{Hausdorff dimension of graph self-similar sets with overlaps}\label{S:GFTC}
		\setcounter{equation}{0}
		In this section,  we study IRSs with overlaps (see Definition \ref{defi:rtgosc}).	We give the definition of the graph finite type condition (GFTC) and prove Theorems \ref{thm:main4}--\ref{thm:main5}. Moreover, we illustrate our method for computing the Hausdorff dimension of the associated attractors by several examples.

		\subsection{Graph finite type condition} 
		 The definitions of  an equivalence relation and a 	sequence of nested index sets that appear in the following definition can be found in \cite{Ngai-Wang-Dong_2010,Ngai-Xu_2023} and are included in the Appendix A for completeness.
		\begin{defi}\label{defi:GFTC}
			Let $ X $ be a complete metric space. Let $G = (V, E)$ be a GIFS of contractions $\{f_e\}_{e\in E}$ on $X$, where $V=\{1,\ldots, m\}$. If there exists an invariant family of  nonempty bounded open sets $\mathbf{U}=\{U_i\}_{i=1}^m$ with respect to some
			sequence of nested index sets $\{\mathcal{M}_k\}_{k=0}^\infty$ such that 
			$\#\mathcal{V}/_\sim:=\{[\mathbf{v}]_{\mathbf{U}},\mathbf{v}\in \mathcal{V}\}$ is a finite set,
		    where $\sim$ is an equivalence relation on $\mathcal{V}$, and $\mathcal{V}$ is defined as in \eqref{eq:v},
			then we say that $G=(V,E)$ satisfies {\em the graph finite type condition} (GFTC). We say that
			$\mathbf{U}$ is {\em a  finite type condition}.
		\end{defi}
	
	\begin{defi}\label{defi:rtgftc}
		Let $ X $ be a complete metric space. Let $\{R_t\}_{t=1}^N$ be an IRS on a nonempty compact subset of $ X $. Assume that there exists a GIFS $G$ associated to $\{R_t\}_{t=1}^N$ and assume that  $G$ consists of contractions. If $G$ satisfies (GFTC), then we say that {\em $\{R_t\}_{t=1}^N$
			satisfies (GFTC) with respect to $G$}.
		\end{defi}
		The following theorem provides a sufficient condition for a GIFS to satisfy the finite type condition. Recall that an algebraic integer $\beta > 1$ is called a {\em Pisot number} if all of its algebraic conjugates are in modulus strictly less than one.
		
		\begin{thm}\label{thm:Pisot}
			Let $M$ be a complete smooth $n$-dimensional Riemannian manifold that is locally Euclidean. Let $G = (V, E)$ be a GIFS of contractive similitudes $\{f_e\}_{e\in E}$ on $M$. Let $\{W_i\}_{i=1}^m$ be an invariant family of nonempty compact sets, and let $\mathbf{U}:=\{U_i\}_{i=1}^m$ be an invariant family of nonempty bounded open sets with $\overline{U}_i=W_i$, $i=1,\ldots,m$. For each similitude $f_e$, $e\in E$, assume that there exists an isometry
			$$g_i:U_i\to U'_i\subseteq \R^n$$
			such that for any $e'\in E'$, $f'_{e'}:=g_i\circ f_e\circ g_i^{-1}$ is contractive similitude of the form
			$$f'_{e'}(x)=\beta^{-n_{e'}}R_{e'}(x)+b_{e'},$$
			where $E'$ is a set of directed edges, $\beta>1$ is a Pisot number, $n_{e'}$ is a positive integer, $R_{e'}$ is an orthogonal transformation, and $b_{e'}\in \R^n$. Assume that $\{R_{e'}\}_{e'\in E'}$ generates a finite group $H$ and
			$$H\{b_{e'}|e'\in E'\}\subseteq r_1\Z[\beta]\times\cdots\times r_n\Z[\beta]$$
			for some $r_1,\ldots,r_n\in \R$. Then $G$ is of finite type and $\mathbf{U}$ is a finite type condition family of $G$.
		\end{thm}
		\begin{proof}
			By the assumptions, we let  $V':=\{1,\ldots,m\}$ be a set of vertices. Then  $G' := (V', E')$ is a GIFS of contractive similitudes $\{f'_{e'}\}_{e'\in E'}$ on $\R^n$, and $\mathbf{U'}:=\{U'_i\}_{i=1}^m$ is an invariant family of nonempty bounded open sets for $G'$. By the results of \cite[Theorem 2.7]{Das-Ngai_2004}, we have $G'$ is of finite type and $\mathbf{U'}$ is a finite type condition family for $G'$. Let $\mathbf{v}':=(f'_{e'},i,j,k)\in \mathcal{V'}$ be a vertex (see Appendix A), where $\mathcal{V'}$ is a set of all vertices of $\mathcal{G}'$, $e'\in \widetilde{\mathcal{M}}_k^{i,j}$, $1\leq i,j\leq m$, $k\geq 1$, and $\widetilde{\mathcal{M}}_k^{i,j}\subseteq E'$ is a sequence of nested index sets.
			Then $\{[\mathbf{v}']_{\mathbf{U'}},\mathbf{v}'\in \mathcal{V'}\}$ is a finite set. Let $\mathbf{v}:=(f_{e},i,j,k)\in \mathcal{V}$, where  $e\in \mathcal{M}_k^{i,j}$, $1\leq i,j\leq m$, and $ \mathcal{M}_k^{i,j}\subseteq E^*$ is a sequence of nested index sets.
			It follows from the definition of $f'_{e'}$ that
		$\{[\mathbf{v}]_{\mathbf{U}},\mathbf{v}\in \mathcal{V}\}=\{[\mathbf{v}']_{\mathbf{U'}},\mathbf{v}'\in \mathcal{V'}\}$ is a finite set. This proves the proposition.
		\end{proof}

		We let $\mathcal{T}_1,\ldots, \mathcal{T}_m$ denote the collection of all distinct neighbourhood types, with $[\mathbf{v}_{\text{root}}^i]$, $i=1,\ldots, m$, being the neighbourhood types of the root vertices. As in \cite{Lau-Ngai_2007}, for each $\alpha\geq 0$, we define a {\em weighted incidence
			matrix} $A_\alpha= (A_\alpha(i,j))_{i,j=1}^m$ as follows. Fix $i$ ($1\leq i\leq m$) and a vertex $\mathbf{v}\in \mathcal{V_R}$ such that $[\mathbf{v}] =\mathcal{T}_i$, let $\mathbf{u}_1,\ldots, \mathbf{u}_m$ be the offspring of $\mathbf{v}$ in $\mathcal{V_R}$ and let $\mathbf{k}_l$, $1\leq l\leq m$, be the unique edge in $\mathcal{G_R}$ connecting $\mathbf{v}$ to $\mathbf{u}_l$. Then we define
		\begin{align}\label{eq:mat}
			A_\alpha(i,j):=\sum\{\rho_{\mathbf{k}_l}^\alpha: \mathbf{v} \stackrel{\mathbf{k}_l}{\longrightarrow} \mathbf{u}_l,\, [\mathbf{u}_l]=\mathcal{T}_j\}.
		\end{align}
		
		\begin{proof}[Proof of Theorem \ref{thm:main4}]
			By Proposition \ref{prop:3.1}, we know that $G$ and $\{R_t\}_{t=1}^N$ have the same attractor. The proof follows by using the results of \cite[Theorem 1.6]{Ngai-Xu_2023}; we omit the details.
		\end{proof}

		\begin{proof}[Proof of Theorem \ref{thm:main5}]
			Combining Proposition \ref{prop:3.1} and Definition \ref{simplified graph}, we know that the attractor of $G$ is equal to that of $\widehat{G}$. The proof follows by using the results of Theorem \ref{thm:main4}.
		\end{proof}

		\subsection{Examples}
		In this subsection, we give three examples of IRSs with overlaps that satisfy (GFTC).
		\begin{exam}\label{exam:gftc1}
			Let $\mathcal{C}^2=\big\{(\cos\theta,\sin\theta,z):\theta\in[-\pi,\pi],z\in[0,2\pi]\big\}$ be a cylindrical surface. Let $E_0:=\mathcal{C}^2$. For $r\in [0,\pi/2)$, we let $\boldsymbol{x}:=(\cos\theta,\sin\theta,z)$ and let $\{R_t\}_{t=1}^4$ be an IRS on $E_0$ defined as
			\begin{align*}
				&R_1(\boldsymbol{x}):=\left\{
				\begin{aligned}
					&(-\cos (\theta/2),-\sin (\theta/2),z/2+\pi/2)\qquad\,\,  &&\boldsymbol{x}\in E_0\backslash \{(-1,0,z):z\in[0,2\pi]\},\\
					&\{(0,-1,z/2+\pi/2),(0,1,z/2+\pi/2)\} \quad\, &&\boldsymbol{x}\in\{(-1,0,z):z\in[0,2\pi]\};\\
				\end{aligned}	
				\right.\\
				&R_2(\boldsymbol{x}):=\left\{
				\begin{aligned}
					&(\cos (\theta/2),\sin (\theta/2),z/2+\pi-r) &&\boldsymbol{x}\in E_0\backslash \{(-1,0,z):z\in[0,2\pi]\},\\
					&\{(0,1,z/2+\pi-r),(0,-1,z/2+\pi-r)\}  &&\boldsymbol{x}\in\{(-1,0,z):z\in[0,2\pi]\};\\
				\end{aligned}	
				\right.\\
				&R_3(\boldsymbol{x}):=\left\{
				\begin{aligned}
					&(\cos (\theta/2),\sin (\theta/2),z/2+\pi/2)\, \quad\,\, &&\boldsymbol{x}\in E_0\backslash \{(-1,0,z):z\in[0,2\pi]\},\\
					&\{(0,1,z/2+\pi/2),(0,-1,z/2+\pi/2)\} \quad\,\, &&\boldsymbol{x}\in\{(-1,0,z):z\in[0,2\pi]\};\\
				\end{aligned}	
				\right.\\
				&R_4(\boldsymbol{x}):=\left\{
				\begin{aligned}
					&(\cos (\theta/2),\sin (\theta/2),z/2+r)\quad\quad&&\boldsymbol{x}\in E_0\backslash \{(-1,0,z):z\in[0,2\pi]\},\\
					&\{(0,1,z/2+r), (0,-1,z/2+r)\}\qquad\quad\,\,\,\, &&\boldsymbol{x}\in\{(-1,0,z):z\in[0,2\pi]\}.
				\end{aligned}	
				\right.
			\end{align*}
			Let $K$ be the associated attractor (see Figure \ref{fig:gftc}).  Then
			$$\dim_H(K)=\frac{\log(2+\sqrt2)}{\log2}= 1.77155\ldots.$$
		\end{exam}
	
		\begin{proof}
			For any $t\in\{1,\ldots,4\}$, by the definition of $R_t$, we have $H_t^1=H_t:=\{(-1,0,z):z\in[0,2\pi]\}$. Let $r_t:=R_t|_{H_t}$ and $r_t(H_t)=\bigcup_{l=1}^2h_t^{l,1}(H_t^1)$. Then for any $\boldsymbol{x}\in H_t^1$,
			\begin{align*}
				&h_1^{1,1}(\boldsymbol{x})=(0,1,z/2+\pi/2),\quad	&&h_1^{2,1}(\boldsymbol{x})=(0,-1,z/2+\pi/2),\\
				&h_2^{1,1}(\boldsymbol{x})=(0,-1,z/2+\pi-r),\quad	&&h_2^{2,1}(\boldsymbol{x})=(0,1,z/2+\pi-r),\\
				&h_3^{1,1}(\boldsymbol{x})=(0,-1,z/2+\pi/2),\quad	&&h_3^{2,1}(\boldsymbol{x})=(0,1,z/2+\pi/2),\\
				&h_4^{1,1}(\boldsymbol{x})=(0,-1,z/2+r),\quad	    &&h_4^{2,1}(\boldsymbol{x})=(0,1,z/2+r).
			\end{align*}
			\begin{figure}[H]
			\centering
			\mbox{\subfigure[]
				{	\includegraphics[scale=0.34]{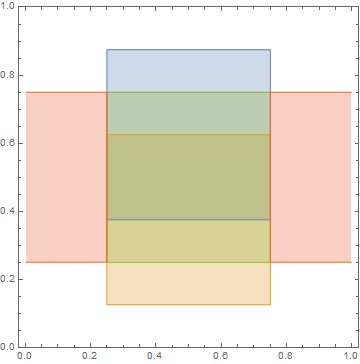}}
			}\quad
			\mbox{\subfigure[]
				{	\includegraphics[scale=0.34]{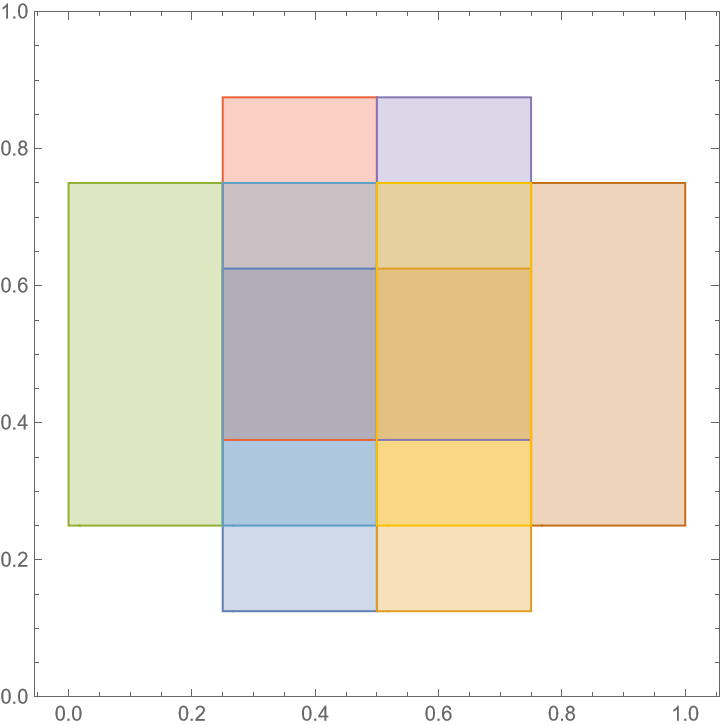}}
			}	\quad
			\mbox{\subfigure[]
				{	\includegraphics[scale=0.34]{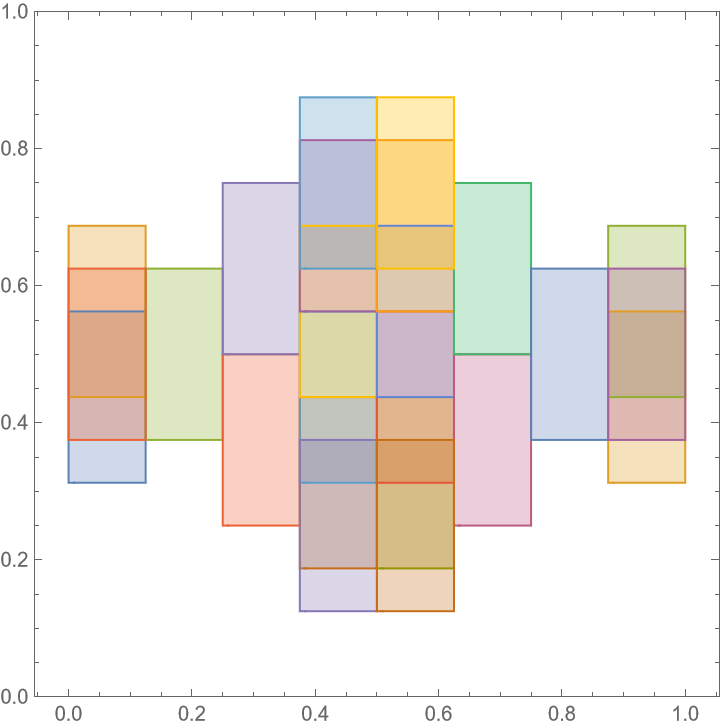}}
			}	\\	
			\mbox{\subfigure[Front]
				{	\includegraphics[scale=0.23]{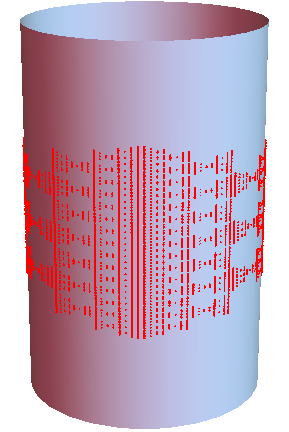}}
			}\qquad\qquad
			\mbox{\subfigure[Back]
				{	\includegraphics[scale=0.23]{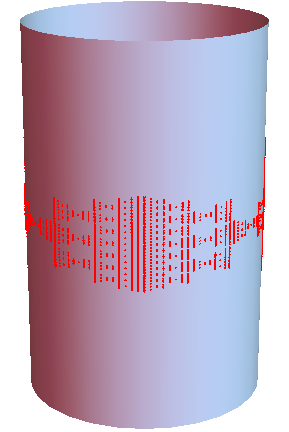}}
			}
			\caption{Figures for Example \ref{exam:gftc1} with $r=\pi/4$. (a)--(c) are drawn on $\mathbb R^2$ and shrunk by $2\pi$. (a) The first iteration of  $E_0$ under $\{R_t\}_{t=1}^4$, where $R_1(E_0)$ consists of the left and right rectangles, $R_2(E_0)$ is the top square, $R_3(E_0)$ is the middle square, and $R_4(E_0)$ is the bottom square. (b) Vertices of the GIFS associated to  $\{R_t\}_{t=1}^4$. (c) The first iteration of the vertices under the GIFS.  (d)--(e) The attractor of $\{R_t\}_{t=1}^4$.}
			\label{fig:gftc}
		\end{figure}
			For any $t\in\{1,\ldots,4\}$, by the definition of $R_t$, we have
			\begin{align*}
				&J_t^1:=\big\{(\cos\theta,\sin\theta,z):\theta\in(-\pi,0],z\in[0,2\pi]\big\}\quad{and}\\
				&J_t^2:=\big\{(\cos\theta,\sin\theta,z):\theta\in(0,\pi),z\in[0,2\pi]\big\}.
			\end{align*}
			Hence $E_0\backslash H_t=\bigcup_{i=1}^2 J_t^i$. Let $\tilde{r}_t:=R_t|_{E_0\backslash H_t}$ and $\tilde{r}_t(E_0\backslash H_t)=\bigcup_{i=1}^2h_t^{0,i}(J_t^i)$. Then for any $\boldsymbol{x} \in J_t^i$, where $i\in \{1,2\}$,
			\begin{align*}
				&h_1^{0,i}(\boldsymbol{x})=(-\cos (\theta/2),-\sin (\theta/2),z/2+\pi/2),\\
				&h_2^{0,i}(\boldsymbol{x})=(\cos (\theta/2),\sin (\theta/2),z/2+\pi-r),\\
				&h_3^{0,i}(\boldsymbol{x})=(\cos (\theta/2),\sin (\theta/2),z/2+\pi/2),\\
				&h_4^{0,i}(\boldsymbol{x})=(\cos (\theta/2),\sin (\theta/2),z/2+r).
			\end{align*}
			Hence for any $t\in \{1,\ldots,4\}$ and $l\in\{1,2\}$, $h_t^{l,1}$ are contractions, and for any $i\in\{1,2\}$, $h_t^{0,i}$ are  contractions. Moreover, for $t=1$, we have
				\begin{align*}
				 \overline{h_t^{1,1}(H_t)}\subseteq \overline{J_t^2},\qquad \overline{h_t^{2,1}(H_t)}\subseteq \overline{J_t^1},\qquad \overline{h_t^{0,1}(J_t^1)}\subseteq \overline{J_t^2}, \qquad \overline{h_t^{0,2}(J_t^2)}\subseteq \overline{J_t^1};
				\end{align*}
			for any $t\in\{2,3,4\}$, 
				\begin{align*}
			 \overline{h_t^{1,1}(H_t)}\subseteq \overline{J_t^1},\qquad \overline{h_t^{2,1}(H_t)}\subseteq \overline{J_t^2},\qquad \overline{h_t^{0,i}(J_t^i)}\subseteq \overline{J_t^i},\,\,\text{where}\,\,i\in \{1,2\}.
			 	\end{align*}
			Hence  $\{R_t\}_{t=1}^4$ satisfies the conditions of Theorem \ref{thm:exi2}.	 
			By Theorem \ref{thm:exi2} and Proposition \ref{prop:3.3}, we can find a minimal simplified GIFS $\widehat{G}=(\widehat{V},\widehat{E})$ with $\widehat{V}=\{1,\ldots,8\}$ and $\widehat{E}=\{e_1,\ldots,e_{32}\}$. The invariant family $\{\widehat{W}_i\}_{i=1}^8$, the set of edges $\widehat{E}^{i,j}$ and the associated similitudes $\{f_e\}_{e\in\widehat{E}}$ are given defined as
			\begin{align*}
				\widehat{W}_1&:=\{(\cos\theta,\sin\theta,z):\theta\in [-\pi,-\pi/2],\,\, z\in[\pi/2,3\pi/2]\},\\	\widehat{W}_2&:=\{(\cos\theta,\sin\theta,z):\theta\in [-\pi/2,0],\,\, z\in[\pi-r,2\pi-r]\},\\
				\widehat{W}_3&:=\{(\cos\theta,\sin\theta,z):\theta\in [-\pi/2,0],\,\, z\in[\pi/2,3\pi/2]\},\\	\widehat{W}_4&:=\{(\cos\theta,\sin\theta,z):\theta\in [-\pi/2,0],\,\, z\in[r,\pi+r]\}\\
				\widehat{W}_5&:=\{(\cos\theta,\sin\theta,z):\theta\in [0,\pi/2],\,\, z\in[\pi-r,2\pi-r]\},\\
				\widehat{W}_6&:=\{(\cos\theta,\sin\theta,z):\theta\in [0,\pi/2],\,\, z\in[\pi/2,3\pi/2]\},\\
				\widehat{W}_7&:=\{(\cos\theta,\sin\theta,z):\theta\in [0,\pi/2],\,\, z\in[r,\pi+r]\},\\
				\widehat{W}_8&:=\{(\cos\theta,\sin\theta,z):\theta\in [\pi/2,\pi],\,\, z\in[\pi/2,3\pi/2]\},
			\end{align*}
			\begin{alignat*}{5}
				&e_1\in \widehat{E}^{1,5},\qquad &e_2\in \widehat{E}^{1,6},\qquad&e_3\in \widehat{E}^{1,7},\qquad&e_4\in \widehat{E}^{1,8},\qquad
				&e_5\in \widehat{E}^{2,1},\qquad&e_6\in \widehat{E}^{2,2},\\&e_7\in \widehat{E}^{2,3},\qquad&e_8\in \widehat{E}^{2,4},\qquad
				&e_9\in \widehat{E}^{3,1},\quad &e_{10}\in \widehat{E}^{3,2},\qquad&e_{11}\in \widehat{E}^{3,3},\qquad&e_{12}\in \widehat{E}^{3,4},\\
				&e_{13}\in \widehat{E}^{4,1},\qquad&e_{14}\in \widehat{E}^{4,2},\qquad&e_{15}\in \widehat{E}^{4,3},\qquad&e_{16}\in \widehat{E}^{4,4},\qquad
				&e_{17}\in \widehat{E}^{5,5}, \qquad&e_{18}\in \widehat{E}^{5,6},\\&e_{19}\in \widehat{E}^{5,7},\qquad&e_{20}\in \widehat{E}^{5,8},\qquad
				&e_{21}\in \widehat{E}^{6,5},\qquad&e_{22}\in \widehat{E}^{6,6},\qquad&e_{23}\in \widehat{E}^{6,7},\qquad&e_{24}\in \widehat{E}^{6,8},\\
				&e_{25}\in \widehat{E}^{7,5}, \qquad&e_{26}\in \widehat{E}^{7,6},\qquad&e_{27}\in \widehat{E}^{7,7},\qquad&e_{28}\in \widehat{E}^{7,8},\qquad
				&e_{29}\in \widehat{E}^{8,1},\qquad&e_{30}\in \widehat{E}^{8,2},\\&e_{31}\in \widehat{E}^{8,3},\qquad&e_{32}\in \widehat{E}^{8,4},\qquad&\qquad&\qquad&\qquad&
			\end{alignat*}
			and
			\begin{alignat*}{5}
				&f_{e_{1}}:=\widetilde{h}_1^{0,2}|_{\widehat{W}_5},  \quad&	f_{e_{2}}:=\widetilde{h}_1^{0,2}|_{\widehat{W}_6},   \quad&	f_{e_{3}}:=\widetilde{h}_1^{0,2}|_{\widehat{W}_7}, \quad& f_{e_{4}}:=\widetilde{h}_1^{0,2}|_{\widehat{W}_8},  \quad&   f_{e_{5}}:=\widetilde{h}_2^{0,1}|_{\widehat{W}_1},\\& f_{e_{6}}:=\widetilde{h}_2^{0,1}|_{\widehat{W}_2}, \quad&
				f_{e_{7}}:=\widetilde{h}_2^{0,1}|_{\widehat{W}_3},  \quad& f_{e_{8}}:=\widetilde{h}_2^{0,1}|_{\widehat{W}_4}, \quad&	f_{e_{9}}:=\widetilde{h}_3^{0,1}|_{\widehat{W}_1}, \quad&
				f_{e_{10}}:=\widetilde{h}_3^{0,1}|_{\widehat{W}_2}, \\& f_{e_{11}}:=\widetilde{h}_3^{0,1}|_{\widehat{W}_3},  \quad&	f_{e_{12}}:=\widetilde{h}_3^{0,1}|_{\widehat{W}_4}, \quad &
				f_{e_{13}}:=\widetilde{h}_4^{0,1}|_{\widehat{W}_1}, \quad& f_{e_{14}}:=\widetilde{h}_4^{0,1}|_{\widehat{W}_2},  \quad&	f_{e_{15}}:=\widetilde{h}_4^{0,1}|_{\widehat{W}_3},\\&
				f_{e_{16}}:=\widetilde{h}_4^{0,1}|_{\widehat{W}_4},  \quad&	f_{e_{17}}:=\widetilde{h}_2^{0,2}|_{\widehat{W}_5},  \quad&	f_{e_{18}}:=\widetilde{h}_2^{0,2}|_{\widehat{W}_6}, \quad&
				f_{e_{19}}:=\widetilde{h}_2^{0,2}|_{\widehat{W}_7},  \quad&  f_{e_{20}}:=\widetilde{h}_2^{0,2}|_{\widehat{W}_8}, \\&	f_{e_{21}}:=\widetilde{h}_3^{0,2}|_{\widehat{W}_5}, \quad&
				f_{e_{22}}:=\widetilde{h}_3^{0,2}|_{\widehat{W}_6},  \quad&	f_{e_{23}}:=\widetilde{h}_3^{0,2}|_{\widehat{W}_7},  \quad&	f_{e_{24}}:=\widetilde{h}_3^{0,2}|_{\widehat{W}_8}, \quad&
				f_{e_{25}}:=\widetilde{h}_4^{0,2}|_{\widehat{W}_5}, \\&  f_{e_{26}}:=\widetilde{h}_4^{0,2}|_{\widehat{W}_6},  \quad&	f_{e_{27}}:=\widetilde{h}_4^{0,2}|_{\widehat{W}_7}, \quad&
				f_{e_{28}}:=\widetilde{h}_4^{0,2}|_{\widehat{W}_8},  \quad&	f_{e_{29}}:=\widetilde{h}_1^{0,1}|_{\widehat{W}_1},  \quad&	f_{e_{30}}:=\widetilde{h}_1^{0,1}|_{\widehat{W}_2}, \\&
				f_{e_{31}}:=\widetilde{h}_1^{0,1}|_{\widehat{W}_3},  \quad&	f_{e_{32}}:=\widetilde{h}_1^{0,1}|_{\widehat{W}_4},  \quad& \quad& \quad&
			\end{alignat*}
			where $\tilde{h}_t^{l,i}$ and $\tilde{h}_t^{0,i}$ are defined as in \eqref{eq:hl} and \eqref{eq:h0}, respectively.
			Let $\{U_i\}_{i=1}^8$ be an invariant family of nonempty bounded open sets with $\overline{U}_i=\widehat{W}_i$, $i\in\{1,\ldots,8\}$.  By Theorem \ref{thm:Pisot}, $\widehat{G}$ is of finite type.
			
			For convenience, we let $f_{e_i}:=f_i$, $i\in\{1,\ldots,32\}$. Let $\mathcal{M}_k:=\{1,\ldots,32\}^k$ for $k\geq0$. Let $\mathcal{T}_1,\ldots,\mathcal{T}_8$ be the neighborhood types of the root neighborhoods $[U_1],\ldots,[U_8]$, respectively. All neighborhood types are generated after two iterations. To construct the weighted incidence matrix in the minimal simplified reduced GIFS $\mathcal{\widehat{G}_R}$ (see Appendix A). We note that
			$$\mathcal{V}_1=\{(f_1,1),\ldots,(f_{32},1)\}.$$
			Denote by $\mathbf{v}_1,\ldots,\mathbf{v}_{32}$ the vertices in $\mathcal{V}_1$ according to the above order. Then
			\begin{align*}
				[\mathbf{v}_5]=[\mathbf{v}_9]=[\mathbf{v}_{13}]=[\mathbf{v}_{29}]=\mathcal{T}_1\qquad\text{and}\qquad
				[\mathbf{v}_4]=[\mathbf{v}_{20}]=[\mathbf{v}_{24}]=[\mathbf{v}_{28}]=\mathcal{T}_2.
			\end{align*}
			Let
			\begin{alignat*}{2}
				&\mathcal{T}_{9}:=[\mathbf{v}_6]=[\mathbf{v}_{10}]=[\mathbf{v}_{14}]=[\mathbf{v}_{30}],\qquad&
				\mathcal{T}_{10}:=[\mathbf{v}_7]=[\mathbf{v}_{11}]=[\mathbf{v}_{15}]=[\mathbf{v}_{31}],\\
				&\mathcal{T}_{11}:=[\mathbf{v}_8]=[\mathbf{v}_{12}]=[\mathbf{v}_{16}]=[\mathbf{v}_{32}],\qquad&
				\mathcal{T}_{12}:=[\mathbf{v}_1]=[\mathbf{v}_{17}]=[\mathbf{v}_{21}]=[\mathbf{v}_{25}],\\
				&\mathcal{T}_{13}:=[\mathbf{v}_2]=[\mathbf{v}_{18}]=[\mathbf{v}_{22}]=[\mathbf{v}_{26}],\qquad&
				\mathcal{T}_{14}:=[\mathbf{v}_3]=[\mathbf{v}_{19}]=[\mathbf{v}_{23}]=[\mathbf{v}_{27}].
			\end{alignat*}
			Then
			\begin{align*}
				&\mathcal{T}_1\rightarrow\mathcal{T}_2+\mathcal{T}_{12}+\mathcal{T}_{13}+\mathcal{T}_{14},\qquad
				&&\mathcal{T}_2\rightarrow\mathcal{T}_1+\mathcal{T}_{9}+\mathcal{T}_{10}+\mathcal{T}_{11},\\
				&\mathcal{T}_3\rightarrow\mathcal{T}_1+\mathcal{T}_{9}+\mathcal{T}_{10}+\mathcal{T}_{11},\qquad
				&&\mathcal{T}_4\rightarrow\mathcal{T}_1+\mathcal{T}_{9}+\mathcal{T}_{10}+\mathcal{T}_{11},\\
				&\mathcal{T}_5\rightarrow\mathcal{T}_2+\mathcal{T}_{12}+\mathcal{T}_{13}+\mathcal{T}_{14},\qquad
				&&\mathcal{T}_6\rightarrow\mathcal{T}_2+\mathcal{T}_{12}+\mathcal{T}_{13}+\mathcal{T}_{14},\\
				&\mathcal{T}_7\rightarrow\mathcal{T}_2+\mathcal{T}_{12}+\mathcal{T}_{13}+\mathcal{T}_{14},\qquad
				&&\mathcal{T}_8\rightarrow\mathcal{T}_1+\mathcal{T}_{9}+\mathcal{T}_{10}+\mathcal{T}_{11}.
			\end{align*}
	 \begin{comment}
	The edge $e_1e_3$ is removed in $\mathcal{\widehat{G}_R}$ because $f_1f_3=f_2f_1$.  $\mathbf{v}_1$ generates three offspring
	$$(f_1f_1,2),(f_1f_2,2),(f_1f_4,2)\in\mathcal{V}_2,$$
	where $[(f_1f_1,2)]=\mathcal{T}_{12}$, $[(f_1f_2,2)]=\mathcal{T}_{13}$ and $[(f_1f_4,2)]=\mathcal{T}_{2}\in\mathcal{V}_2$. Hence
	$$\mathcal{T}_{12}\rightarrow\mathcal{T}_2+\mathcal{T}_{12}+\mathcal{T}_{13}.$$
	The edge $e_2e_3$ is removed in $\mathcal{\widehat{G}_R}$ because $f_2f_3=f_3f_1$. $\mathbf{v}_2$ generates three offspring
	$$(f_2f_1,2),(f_2f_2,2),(f_2f_4,2)\in\mathcal{V}_2,$$
	where $[(f_2f_1,2)]=\mathcal{T}_{12}$, $[(f_2f_2,2)]=\mathcal{T}_{13}$ and $[(f_2f_4,2)]=\mathcal{T}_{2}\in\mathcal{V}_2$. Hence
	$$\mathcal{T}_{13}\rightarrow\mathcal{T}_2+\mathcal{T}_{12}+\mathcal{T}_{13}.$$
	$\mathbf{v}_3$ generates four offspring
	$$(f_3f_1,2),(f_3f_2,2),(f_3f_3,2),(f_3f_4,2)\in\mathcal{V}_2,$$
	where $[(f_3f_1,2)]=\mathcal{T}_{12}$, $[(f_3f_2,2)]=\mathcal{T}_{13}$, $[(f_3f_3,2)]=\mathcal{T}_{14}$ and $[(f_3f_4,2)]=\mathcal{T}_{2}$. Hence
	$$\mathcal{T}_{14}\rightarrow\mathcal{T}_2+\mathcal{T}_{12}+\mathcal{T}_{13}+\mathcal{T}_{14}.$$
	\end{comment}
	Since $f_6f_8=f_7f_6$, the edge $e_6e_8$ is removed in $\mathcal{\widehat{G}_R}$. $\mathbf{v}_6$ generates three offspring
		$$(f_6f_5,2),(f_6f_6,2),(f_6f_7,2)\in\mathcal{V}_2,$$
		where $[(f_6f_5,2)]=\mathcal{T}_{1}$, $[(f_6f_6,2)]=\mathcal{T}_{9}$ and $[(f_6f_7,2)]=\mathcal{T}_{10}$. Hence
		$$\mathcal{T}_{9}\rightarrow\mathcal{T}_1+\mathcal{T}_{9}+\mathcal{T}_{10}.$$
		As $f_7f_8=f_8f_6$, the edge $e_7e_8$ is removed in $\mathcal{\widehat{G}_R}$. $\mathbf{v}_7$ generates three offspring
		$$(f_7f_5,2),(f_7f_6,2),(f_7f_7,2)\in\mathcal{V}_2,$$
		with $[(f_7f_5,2)]=\mathcal{T}_{1}$, $[(f_7f_6,2)]=\mathcal{T}_{9}$ and $[(f_7f_7,2)]=\mathcal{T}_{10}$. Thus
		$$\mathcal{T}_{10}\rightarrow\mathcal{T}_1+\mathcal{T}_{9}+\mathcal{T}_{10}.$$
		$\mathbf{v}_8$ generates four offspring
		$$(f_8f_5,2),(f_8f_6,2),(f_8f_7,2),(f_8f_8,2)\in\mathcal{V}_2,$$
		where $[(f_8f_5,2)]=\mathcal{T}_{1}$, $[(f_8f_6,2)]=\mathcal{T}_{9}$, $[(f_8f_7,2)]=\mathcal{T}_{10}$ and $[(f_8f_8,2)]=\mathcal{T}_{11}$. Therefore,
		$$\mathcal{T}_{11}\rightarrow\mathcal{T}_1+\mathcal{T}_{9}+\mathcal{T}_{10}+\mathcal{T}_{11}.$$
		Using the same argument, we have
		\begin{align*}
			\mathcal{T}_{12}\rightarrow\mathcal{T}_2+\mathcal{T}_{12}+\mathcal{T}_{13},\,\,\,\,\mathcal{T}_{13}\rightarrow\mathcal{T}_2+\mathcal{T}_{12}+\mathcal{T}_{13},\,\,\,\,\mathcal{T}_{14}\rightarrow\mathcal{T}_2+\mathcal{T}_{12}+\mathcal{T}_{13}+\mathcal{T}_{14}.
	\end{align*}
			Since no new neighborhood types are generated, we conclude that the $\mathcal{\widehat{G}_R}$ is of  finite
			type. The weighted incidence matrix is
			\begin{align}\label{eq:matrixtou}
				A_\alpha=\Big(\frac{1}{2}\Big)^\alpha
				\begin{bmatrix}
						\begin{array}{cccccccccccccc}
							0 &1 &0 &0 &0 &0 &0 &0 &0 &0 &0 &1 &1 &1\\
							1 &0 &0 &0 &0 &0 &0 &0 &1 &1 &1 &0 &0 &0\\
							1 &0 &0 &0 &0 &0 &0 &0 &1 &1 &1 &0 &0 &0\\
							1 &0 &0 &0 &0 &0 &0 &0 &1 &1 &1 &0 &0 &0\\
							0 &1 &0 &0 &0 &0 &0 &0 &0 &0 &0 &1 &1 &1\\
							0 &1 &0 &0 &0 &0 &0 &0 &0 &0 &0 &1 &1 &1\\
							0 &1 &0 &0 &0 &0 &0 &0 &0 &0 &0 &1 &1 &1\\
							1 &0 &0 &0 &0 &0 &0 &0 &1 &1 &1 &0 &0 &0\\
							1 &0 &0 &0 &0 &0 &0 &0 &1 &1 &0 &0 &0 &0\\
							1 &0 &0 &0 &0 &0 &0 &0 &1 &1 &0 &0 &0 &0\\
							1 &0 &0 &0 &0 &0 &0 &0 &1 &1 &1 &0 &0 &0\\
							0 &1 &0 &0 &0 &0 &0 &0 &0 &0 &0 &1 &1 &0\\
							0 &1 &0 &0 &0 &0 &0 &0 &0 &0 &0 &1 &1 &0\\
							0 &1 &0 &0 &0 &0 &0 &0 &0 &0 &0 &1 &1 &1\\
						\end{array}
					\end{bmatrix}.
			\end{align}
			The spectral radius $\lambda_\alpha$ of $A_\alpha$ is $(2+\sqrt{2})/2^\alpha$, and by Theorem \ref{thm:main5},
			$$\dim_H(K)=\alpha= 1.77155\ldots,$$
			where $\alpha$ is the unique solution of the equation $\lambda_\alpha=1.$
		\end{proof}

		The following example is from \cite[Example 7.6]{Ngai-Xu_2023}. Here we use the method in the present paper to compute the Hausdorff dimension of the same fractal. The method here is more systematic.
		\begin{exam}\label{exam:gftc2}
			Let $\mathbb{T}^2 := \mathbb{S}^1\times \mathbb{S}^1$ be a flat 2-torus, viewed as
			$[0, 1]\times[0, 1]$ with opposite sides identified, and $\mathbb{T}^2$ be endowed with the Riemannian metric induced from $\R^2$.  We consider the following IFS with overlaps on $\R^2$:
			\begin{align*}
				g_1(\boldsymbol{x})=\frac{1}{2}\boldsymbol{x}+\Big(0,\frac{1}{4}\Big),\qquad g_2(\boldsymbol{x})=\frac{1}{2}\boldsymbol{x}+\Big(\frac{1}{4},\frac{1}{4}\Big),\\
				g_3(\boldsymbol{x})=\frac{1}{2}\boldsymbol{x}+\Big(\frac{1}{2},\frac{1}{4}\Big),\qquad g_4(\boldsymbol{x})=\frac{1}{2}\boldsymbol{x}+\Big(\frac{1}{4},\frac{3}{4}\Big).
			\end{align*}
			Let $H:=\{0\}\times [0,1]\big)\bigcup  \big([0,1]\times\{0\}\big)$. Iterations of $\{g_t\}_{t=1}^4$ induce an IRS $\{R_t\}_{t=1}^4$ on $E_0:=\mathbb{T}^2=\R^2/ \Z^2$, defined as
				\begin{align*}
					&R_1(\boldsymbol{x}):=\left\{
					\begin{aligned}
						&\frac{1}{2}\boldsymbol{x}+\Big(0,\frac{1}{4}\Big),\qquad  &&\boldsymbol{x}\in E_0\backslash H,\\
						& \Big\{\frac{1}{2}\boldsymbol{x}+\Big(0,\frac{1}{4}\Big),\frac{1}{2}\boldsymbol{x}+\Big(\frac{1}{2},\frac{1}{4}\Big) \Big\} \qquad &&\boldsymbol{x}\in \{0\}\times [0,1],\\
						& \Big\{\frac{1}{2}\boldsymbol{x}+\Big(0,\frac{1}{4}\Big),\frac{1}{2}\boldsymbol{x}+\Big(0,\frac{3}{4}\Big) \Big\} \qquad &&\boldsymbol{x}\in  [0,1]\times\{0\};
					\end{aligned}	
					\right.\\
					&R_2(\boldsymbol{x}):=\left\{
					\begin{aligned}
						&\frac{1}{2}\boldsymbol{x}+\Big(\frac{1}{4},\frac{1}{4}\Big),\quad  &&\boldsymbol{x}\in E_0\backslash H,\\
						&\Big\{\frac{1}{2}\boldsymbol{x}+\Big(\frac{1}{4},\frac{1}{4}\Big),\frac{1}{2}\boldsymbol{x}+\Big(\frac{3}{4},\frac{1}{4}\Big) \Big\} \quad &&\boldsymbol{x}\in \{0\}\times [0,1],\\
						& \Big\{\frac{1}{2}\boldsymbol{x}+\Big(\frac{1}{4},\frac{1}{4}\Big),\frac{1}{2}\boldsymbol{x}+\Big(\frac{1}{4},\frac{3}{4}\Big) \Big\} \qquad &&\boldsymbol{x}\in  [0,1]\times\{0\};
					\end{aligned}	
					\right.\\
					&R_3(\boldsymbol{x}):=\left\{
					\begin{aligned}
						&\frac{1}{2}\boldsymbol{x}+\Big(\frac{1}{2},\frac{1}{4}\Big),\qquad  &&\boldsymbol{x}\in E_0\backslash H,\\
						&\Big\{\frac{1}{2}\boldsymbol{x}+\Big(\frac{1}{2},\frac{1}{4}\Big),\frac{1}{2}\boldsymbol{x}+\Big(1,\frac{1}{4}\Big) \Big\} \qquad &&\boldsymbol{x}\in \{0\}\times [0,1],\\
						& \Big\{\frac{1}{2}\boldsymbol{x}+\Big(\frac{1}{2},\frac{1}{4}\Big),\frac{1}{2}\boldsymbol{x}+\Big(\frac{1}{2},\frac{3}{4}\Big)\Big\} \qquad &&\boldsymbol{x}\in  [0,1]\times\{0\};
					\end{aligned}	
					\right.\\
					&R_4(\boldsymbol{x}):=\left\{
					\begin{aligned}
						&\frac{1}{2}\boldsymbol{x}+\Big(\frac{1}{4},\frac{3}{4}\Big),\quad  &&\boldsymbol{x}\in E_0\backslash H,\\
						&\Big\{ \frac{1}{2}\boldsymbol{x}+\Big(\frac{1}{4},\frac{3}{4}\Big), \frac{1}{2}\boldsymbol{x}+\Big(\frac{3}{4},\frac{3}{4}\Big) \Big\} \quad &&\boldsymbol{x}\in \{0\}\times [0,1/2],\\
						&\Big\{\frac{1}{2}\boldsymbol{x}+\Big(\frac{1}{4},-\frac{1}{4}\Big),\frac{1}{2}\boldsymbol{x}+\Big(\frac{3}{4},\frac{1}{4}\Big) \Big\} \quad &&\boldsymbol{x}\in \{0\}\times [1/2,1],\\
						&\Big\{\frac{1}{2}\boldsymbol{x}+\Big(\frac{1}{4},\frac{1}{4}\Big), \frac{1}{2}\boldsymbol{x}+\Big(\frac{1}{4},\frac{3}{4}\Big) \Big\} \quad &&\boldsymbol{x}\in  [0,1]\times\{0\}.
					\end{aligned}	
					\right.
			\end{align*}
			Let $K$ be the associated attractor.  Then
			$$\dim_H(K)=\frac{\log(2+\sqrt2)}{\log2}= 1.77155\ldots.$$
		\end{exam}
		\begin{proof} The proof of this example is similar to that of Example \ref{exam:osc1}; we only give an outline. 
			
			First, for any $t\in\{1,\ldots,4\}$, by the definition of $R_t$, we let $H_t^1=\{0\}\times [0,1/2]$, $H_t^2=\{0\}\times (1/2,1]$, and $H_t^3=  [0,1]\times\{0\}$. Then $H_t=\bigcup_{i=1}^3 H_t^i$. Let $r_t:=R_t|_{H_t}$. Then $r_t(H_t)=\bigcup_{l=1}^2\bigcup_{i=1}^3h_t^{l,i}(H_t^i)$.  For any $t\in\{1,\ldots,4\}$,
	let
			\begin{align*}
				J_t^1:=\big\{[0,1]\times [0,1/2]\big\}\qquad \text{and}\qquad J_t^2:=\big\{[0,1]\times [1/2,1]\big\}.
			\end{align*}
		 Hence $E_0\backslash H_t=\bigcup_{i=1}^2J_t^i$. Let $\tilde{r}_t:=R_t|_{E_0\backslash H_t}$. Then $\tilde{r}_t(E_0\backslash H_t)=\bigcup_{i=1}^2h_t^{0,i}(J_t^i)$.
			We can show that for any $t\in \{1,\ldots,4\}$, $h_t^{l,i}$ are contractions, where $l,i\in\{1,2\}$, and that for $i\in\{1,2\}$, $h_t^{0,i}$ are  contractions. Moreover, there exist  $\alpha,\beta\in\{1,2,3\}$ or $\sigma,\tau\in\{1,2\}$ such that for any $l\in \{1,2\}$ and any $i\in\{1,2,3\}$,
			$$\overline{h _t^{l,i}(H_t^i)}\subseteq \overline{H_t^\alpha}\qquad \text{or}\qquad \overline{h _t^{l,i}(H_t^i)}\subseteq \overline{J_t^\sigma},$$
			and for any $i\in\{1,2\}$,
			$$\overline{h _t^{0,i}(J_t^i)}\subseteq \overline{H_t^\beta}\qquad \text{or}\qquad \overline{h _t^{0,i}(J_t^i)}\subseteq \overline{J_t^\tau}.$$
			Hence  $\{R_t\}_{t=1}^4$ satisfies the conditions of Theorem \ref{thm:exi2}.	 	
		By Theorem \ref{thm:exi2} and Proposition \ref{prop:3.3}, we can obtain a minimal simplified GIFS $\widehat{G}=(\widehat{V},\widehat{E})$ where $\widehat{V}=\{1,\ldots,8\}$ and $\widehat{E}=\{e_1,\ldots,e_{32}\}$, along with the invariant family $\bigcup_{i=1}^8\widehat{W}_j$ and the associated similitudes $\{f_e\}_{e\in\widehat{E}}$.
			Let $\{U_i\}_{i=1}^8$ be an invariant family of nonempty bounded open sets with $\overline{U}_i=\widehat{W}_i$, $i=1,\ldots,8$.  It follows from Theorem \ref{thm:Pisot} that the system $\widehat{G}$ is of finite type.
			
			Next, let $\mathcal{T}_1,\ldots,\mathcal{T}_8$ be the neighborhood types of the root neighborhoods $[U_1],\ldots,[U_8]$, respectively. All neighborhood types are generated after two iterations.  The weighted incidence matrix happens to be the same as that in \eqref{eq:matrixtou}.
			Thus the spectral radius $\lambda_\alpha$ of $A_\alpha$ is $(2+\sqrt{2})/2^\alpha$, and by Theorem \ref{thm:main5}, we have
			$$\dim_H(K)=\alpha= 1.77155\ldots,$$
			where $\alpha$ is the unique solution of the equation $\lambda_\alpha=1.$
		\end{proof}

		\begin{exam}\label{exam:gftctri}
			Let $\mathcal{C}^2:=\mathbb{S}^1\times\R^1=\big\{(\cos\theta,\sin\theta,z):\theta\in[-\pi,\pi],z\in [0,2\pi]\big\}$ be a cylindrical surface. Let
			\begin{align*}
				E_0^1:=&\big\{(\cos\theta,\sin\theta,z):\theta\in[-\pi,0],z\in [0,\sqrt{3}\,\theta+\sqrt{3}\,\pi]\big\},\\ E_0^2:=&\big\{(\cos\theta,\sin\theta,z):\theta\in[0,\pi],z\in [0,-\sqrt{3}\,\theta+\sqrt{3}\,\pi]\big\},
			\end{align*}
			and $E_0:=E_0^1\bigcup E_0^2$.
			Let $\rho:=(\sqrt5-1)/2$, $\boldsymbol{x}:=(\cos\theta,\sin\theta,z)\in E_0$, and $\{R_t\}_{t=1}^4$ be an IRS on $E_0$ defined as
			\begin{align*}
				&R_1(\boldsymbol{x}):=\left\{
				\footnotesize{\begin{aligned}
						&(\cos (\rho\theta-\rho^2\pi),\sin (\rho\theta-\rho^2\pi),\rho z),\qquad\qquad\quad\qquad\qquad\qquad \qquad &&\boldsymbol{x}\in E_0\backslash \{(-1,0,0)\},\\
						&\{(\cos(\rho^3\pi),\sin(\rho^3\pi),0), (-1,0,0)\}\qquad &&\boldsymbol{x}=(-1,0,0);\\
				\end{aligned}	}
				\right.\\
				&R_2(\boldsymbol{x}):=\left\{
				\footnotesize{	\begin{aligned}
						&(\cos (\rho\theta+\rho^2\pi),\sin (\rho\theta+\rho^2\pi),\rho z),\qquad\qquad\qquad\qquad\qquad\qquad\quad &&\boldsymbol{x}\in E_0\backslash \{(-1,0,0)\},\\
						&\{(-1,0,0),(\cos(-\rho^3\pi),\sin(-\rho^3\pi),0)\}  \qquad&&\boldsymbol{x}=(-1,0,0);\\
				\end{aligned}	}
				\right.\\
				&R_3(\boldsymbol{x}):=\left\{
				\footnotesize{\begin{aligned}
						&(\cos (\rho^2\theta),\sin (\rho^2\theta),\rho z+ \sqrt3\rho\pi),\qquad &&\boldsymbol{x}\in E_0\backslash \{(-1,0,0)\},\\
						&\{(\cos (\rho^2\pi),\sin (\rho^2\pi),\sqrt3\rho\pi),(\cos(-\rho^3\pi),\sin(-\rho^3\pi),\sqrt3\rho\pi)\} \qquad &&\boldsymbol{x}=(-1,0,0).\\
				\end{aligned}	}
				\right.
			\end{align*}
			Let $K$ be the associated attractor (see Figure \ref{fig:gftctri}).
			Then
			$$\dim_H(K)= 1.68239\ldots.$$
		\end{exam}
		The proof of this example is similar to that of Example \ref{exam:gftc1}; we omit the proof.
		\begin{figure}[H]
			\centering
			\mbox{\subfigure[Front]
				{	\includegraphics[scale=0.24]{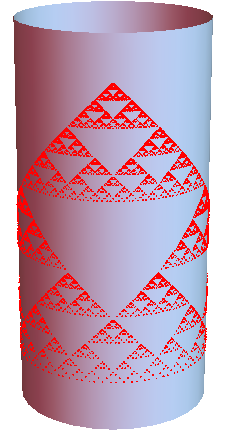}}
			}\qquad\qquad
			\mbox{\subfigure[Back]
				{	\includegraphics[scale=0.24]{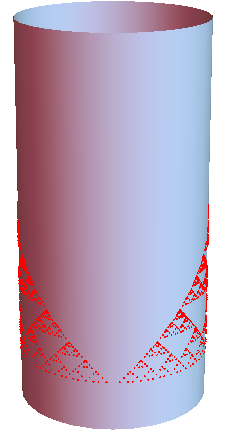}}
			}
			\caption{The attractor of $\{R_t\}_{t=1}^3$  in Example \ref{exam:gftctri}.}
			\label{fig:gftctri}
		\end{figure}

	\begin{appendix}\section{Definitions related to (GFTC)}\label{app}
		\setcounter{equation}{0}
	The following definitions can be found in  \cite{Lau-Ngai_2007,Ngai-Wang-Dong_2010,Ngai-Xu_2023}.
	For two directed paths, $\mathbf{e},\mathbf{e}'\in E^*$, we write $\mathbf{e}\preceq \mathbf{e}'$ if $\mathbf{e}$ is an initial segment of $\mathbf{e}'$ or $\mathbf{e}=\mathbf{e}'$, and write $\mathbf{e}\npreceq \mathbf{e}'$ if $\mathbf{e}$ is not an initial segment of $\mathbf{e}'$. Let $\{\mathcal{M}_k\}_{k=0}^\infty$ be a sequence of index sets such that for any $k\geq 0$, $\mathcal{M}_k$ is a finite subset of $E^*$. We say that $\mathcal{M}_k$ is an {\em antichain} if for any $\mathbf{e},\mathbf{e}'\in \mathcal{M}_k$, $\mathbf{e}\npreceq \mathbf{e}'$ and $\mathbf{e}'\npreceq \mathbf{e}$. Let
	$$\underline{m}_k=\underline{m}_k(\mathcal{M}_k):=\min\{|\mathbf{e}|:\mathbf{e}\in \mathcal{M}_k\}\quad\text{and}\quad \overline{m}_k=\overline{m}_k(\mathcal{M}_k):=\max\{|\mathbf{e}|:\mathbf{e}\in \mathcal{M}_k\}.$$
	
	\begin{defi}\label{defi:nes}
		We say that $\{\mathcal{M}_k\}_{k=0}^\infty$ is a {\em sequence of nested index sets} if it satisfies the following conditions:
		\begin{enumerate}
			\item[(1)]  both $\{\underline{m}_k\}$ and  $\{\overline{m}_k\}$ are nondecreasing, and $\lim_{k\to\infty} \overline{m}_k=\lim_{k\to\infty} \overline{m}_k=\infty$;
			\item[(2)] for each $k\geq 1$, $\mathcal{M}_k$ is an antichain in $E^*$;
			\item[(3)] for each $\mathbf{e}'\in E^*$ with $|\mathbf{e}'|>\overline{m}_k$, there exists $\mathbf{e}\in \mathcal{M}_k$ such that $\mathbf{e}\preceq \mathbf{e}'$;
			\item[(4)] for each $\mathbf{e}'\in E^*$ with $|\mathbf{e}'|<\underline{m}_k$, there exists some $\mathbf{e}\in \mathcal{M}_k$ such that $\mathbf{e}'\preceq \mathbf{e}$;
			\item[(5)] there exists a positive integer $L$, independent of $k$, such that for all $\mathbf{e}\in \mathcal{M}_k$ and $\mathbf{e}'\in \mathcal{M}_{k+1}$ with $\mathbf{e}\preceq \mathbf{e}'$, we have $|\mathbf{e}'|-|\mathbf{e}|\leq L$.
		\end{enumerate}
		(We allow $\mathcal{M}_k\bigcap \mathcal{M}_{k+1}\neq \emptyset$. Very often, $\bigcup_{k=1}^\infty\mathcal{M}_k$ is a proper subset of $E^*$.)
	\end{defi}
	
	Fix a sequence $\{\mathcal{M}_k\}_{k=0}^\infty$ of nested index sets. It is useful to partition $\mathcal{M}_k$ into $\mathcal{M}_k^{i,j}$, $1\leq i,j\leq m$, defined as
	$$\mathcal{M}_k^{i,j}:=\mathcal{M}_k\bigcap \Big(\bigcup_{p\geq 1}E_p^{i,j}\Big)=\big\{\mathbf{e}=(e_1,\ldots, e_p)\in \mathcal{M}_k:\mathbf{e}\in E_p^{i,j}\,\, \text{for some}\,\, p\geq 1\big\}.$$
	That is, $\mathcal{M}_k^{i,j}$ is the subset of $\mathcal{M}_k$ consisting of all directed paths from vertex $i$ to vertex $j$. Note that $\mathcal{M}_k=\bigcup_{i,j=1}^q\mathcal{M}_k^{i,j}$.
	
	For $i,j=1,\ldots, m$, $k\geq 1$, define
	$$\mathcal{V}_k^{i,j}:=\{(f_\mathbf{e},i,j,k),e\in \mathcal{M}_k^{i,j}\}\qquad \text{and}\qquad \mathcal{V}_k:=\bigcup_{i,j=1}^q\mathcal{V}_k^{i,j}.$$
	For $\mathbf{e}\in \mathcal{M}_k^{i,j}$, we call $(f_\mathbf{e},i,j,k)$ (or simply $(f_e,k)$) a {\em vertex}.
	
	For convenience, we let $\mathcal{M}_0=\{1,\ldots,m\}$ and $\mathcal{V}_0:=\{\mathbf{v}_{\text{root}}^1,\ldots,\mathbf{v}_{\text{root}}^m\}$, where $\mathbf{v}_{\text{root}}^i:=(I,i,i,0)$, with $I$ being the identity map on $M$. We call $\mathcal{V}_0$ the set of {\em root vertices}.
	
	The {\em vertex set } is the set of all vertices (not counting multiplicity)
	\begin{align}\label{eq:v}
	\mathcal{V}:=\bigcup_{k\geq 0}\mathcal{V}_k.
	\end{align}
	
	Let $\pi: \bigcup_{k\geq 0}\mathcal{M}_k\to \mathcal{V}$ be a map from the set of all directed paths in $\bigcup_{k\geq 0}\mathcal{M}_k$ to the vertex set of $G$ defined naturally as follows:
	\begin{align}
		\pi(\mathbf{e}):=\left\{
		\begin{aligned}
			&(f_\mathbf{e},i,j,k),\qquad  &&\text{if}\,\,e\in \mathcal{M}_k^{i,j}\,\,\text{and}\,\,k\geq 1,\\
			&\mathbf{v}_{\text{root}}^i,\qquad &&\text{if}\,\,\mathbf{e}=i\in \mathcal{M}_0.
		\end{aligned}	
		\right.
	\end{align}
	
	Define a graph $\mathcal{G}:=(\mathcal{V},\mathcal{E})$, where $\mathcal{E}$ is the set of all directed edges of $\mathcal{V}$. Given two vertices $\mathbf{v}$ and $\mathbf{v}'$ in $\mathcal{G}$, suppose there exist directed paths $\mathbf{e}\in \mathcal{M}_k$, $\mathbf{e}'\in \mathcal{M}_{k+1}$ and $\mathbf{k}\in \mathcal{E}^*$ such that $\mathbf{v}=\pi(\mathbf{e})$, $\mathbf{v}'=\pi(\mathbf{e}')$ and $\mathbf{e}=\mathbf{e}\mathbf{k}$ (the concatenation of $\mathbf{e}$ and $\mathbf{k}$). Then we connect a directed edge $\mathbf{k}:\mathbf{v}\to \mathbf{v}'$. Note that if $\mathbf{k}:\mathbf{v}_1\to \mathbf{v}'$ and $\mathbf{k}:\mathbf{v}_2\to \mathbf{v}'$, then $\mathbf{v}_1=\mathbf{v}_2$. This way we obtain the set of all directed edges $\mathcal{E}$ of $\mathcal{G}$. We call $\mathbf{v}$ a {\em parent} of $\mathbf{v}'$ and $\mathbf{v}'$ an {\em offspring} of $\mathbf{v}$.
	
	To construct the graph $\mathcal{G_R}$, called the {\em reduced graph}, we first fix an order for $\mathcal{E}^*$. Here we
	use the lexicographical order induced on the index set. This means that if $\mathbf{e}=(e_{i_1},\ldots,e_{i_p})$ and
	$\mathbf{e}'=(e_{i_1'},\ldots,e_{i_p'})$, then $\mathbf{e}<\mathbf{e}'$ if and only if $(i_1,\ldots,i_p) <(i_1',
	\ldots,i_p')$ in the lexicographical order. Start with the vertex set $\mathcal{V}$. The
	set of directed edges $\mathcal{E_R}$ of $\mathcal{G_R}$ is obtained from $\mathcal{G}$ by removing all but the smallest directed
	edge going to a vertex. More precisely, for each vertex $\mathbf{v}$ let $\mathbf{k}_1,\ldots, \mathbf{k}_q$ be all the directed
	edges going from some vertex to $\mathbf{v}$. Suppose that $\mathbf{k}_1 < \mathbf{k}_2 < \cdots < \mathbf{k}_q$ in the order described above. Then we keep only the directed edge $\mathbf{k}_1$ and remove all other edges. This way we
	obtain $\mathcal{E_R}$.
	
	To finish the construction of the reduced graph, we remove all vertices that do not have
	offspring in $\mathcal{G_R}$, together with all the vertices and edges leading only to them (see \cite[appendix A]{Lau-Ngai_2007} for an example of such a vertex). We denote the resulting graph by the same
	symbol $\mathcal{G_R}$ and write $\mathcal{G_R} = (\mathcal{V_R}, \mathcal{E_R})$, where $\mathcal{V_R}$ is the set of all vertices and $\mathcal{E_R}$ is the set of all
	edges.
	
	Let $\widehat{\mathcal{G}}:=(\widehat{\mathcal{V}},\widehat{\mathcal{E}})$ be a minimal simplified GIFS associated to $\mathcal{G}$. We can use a similar method to construct a corresponding reduced GIFS $\widehat{\mathcal{G}}_{\mathcal{R}}:=(\widehat{\mathcal{V}}_{\mathcal{R}},\widehat{\mathcal{E}}_{\mathcal{R}})$.
	
	Fix an invariant family $\mathbf{U}:=\{U_i\}_{i=1}^m$ for $G=(V,E)$. Let $\mathbf{v}:=(f_\mathbf{e},i,j,k)\in \mathcal{V}_k$ with
	$\mathbf{e}\in \mathcal{M}^{i,j}_r$ and $\mathbf{v}'=(f_{\mathbf{e}'},i',j',k)\in \mathcal{V}_k$ with
	$\mathbf{e}'\in \mathcal{M}^{i',j'}_s$. We say that two vertices $\mathbf{v}$ and $\mathbf{v}'$ are {\em neighbours} (with respect to $\mathbf{U}$) if
	\begin{align*}
		i=i'\qquad\text{and}\qquad f_\mathbf{e}(U_j)\bigcap f_{\mathbf{e}'}(U_{j'})\neq \emptyset.
	\end{align*}
	We call the set
	\begin{align*}
		\mathcal{N}(\mathbf{v}):=\{\mathbf{v}':\mathbf{v}'\,\,\text{is a neighbour of}\,\, \mathbf{v}\}
	\end{align*}
	the {\em neighbourhood} of $\mathbf{v}$ (with respect to $\mathbf{U}$).
	
	We now define an equivalence relation on the set of vertices $\mathcal{V}$. Two vertices $\mathbf{v}\in \mathcal{V}_k$
	and $\mathbf{v}'\in \mathcal{V}_{k'}$ are said {\em equivalent}, denoted $\mathbf{v}\sim\mathbf{v}'$ (or more precisely, $\mathbf{v}\sim_\tau\mathbf{v}'$ ), if
	$\#\mathcal{N}(\mathbf{v})=\#\mathcal{N}(\mathbf{v}')$ and $\tau=f_{\mathbf{v}'}\circ f_{\mathbf{v}}^{-1}:M\to M$ induces a bijection $g_\tau:\mathcal{N}(\mathbf{v})\to \mathcal{N}(\mathbf{v}')$ defined by
	\begin{align}\label{eq:equi}
		g_\tau(f_\mathbf{u},i,j,k)=(\tau\circ f_\mathbf{u},i',j',k')
	\end{align}
	so that the following condition are satisfied:
	\begin{enumerate}
		\item[(a)] In \eqref{eq:equi}, $j=j'$.
		\item[(b)] For $\mathbf{u}\in \mathcal{N}(\mathbf{v})$ and $\mathbf{u}'\in\mathcal{N}(\mathbf{v}')$ such that $g_\tau(\mathbf{u})=\mathbf{u}'$, and for any positive integer $l\geq 1$, a directed path $\mathbf{e}\in E^*$ satisfies $(f_\mathbf{u}\circ f_\mathbf{e},k+l)\in \mathcal{V}_{k+l}$ if and only if it satisfies $(f_{\mathbf{u}'}\circ f_\mathbf{e},k'+l)\in \mathcal{V}_{k'+l}$.
	\end{enumerate}
	
	It is easy to check that $\sim$ is an equivalence relation. Denote the equivalence class of
	$\mathbf{v}$ by $[\mathbf{v}]$, and call it the \textit{neighborhood types} of $\mathbf{v}$ (with respect to $\mathbf{U}$).
	\end{appendix}
	
	\section*{Acknowledgement}
	
	Part of this work was carried out while the third author was visiting Beijing Institute of Mathematical Sciences and Applications (BIMSA). She thanks the institute for its hospitality and support.

	\end{document}